\documentclass[12pt]{article}
\usepackage[cp1251]{inputenc}
\usepackage{amsmath,amsthm,amssymb}

\usepackage{graphicx}
\usepackage{epstopdf} 
\graphicspath{}
\DeclareGraphicsExtensions{.pdf,.eps}

\newtheorem{theorem}{\noindent Theorem}

\newtheorem{definition}{\noindent Definition}
\newtheorem{theorem-definition}{\noindent Theorem-Definition}

\newtheorem{conjecture}{\noindent Conjecture}
\newtheorem{statement}{\noindent Proposition}
\newtheorem{remark}{\noindent Remark}
\newtheorem{problem}{\noindent Problem}

\newtheorem{definition-theorem}{\noindent definition-theorem}

\def\Erg{\operatorname{Erg}}
\def\Inv{\operatorname{Inv}}

\def\Aut{\operatorname{Aut}}

\title{Asymptotic theory of  path spaces of graded graphs and its applications}

\author{A.~M.~Vershik\thanks{
St.~Petersburg Department of Steklov Institute of Mathematics, Mathematical Department of St.~Petersburg State University,
Moscow Institute for Information  Transmission Problems. E-mail: avershik@gmail.com.
Supported by the Russian Science Foundation grant 14-11-00581.}}

\begin{document}

\maketitle

\centerline{\fontsize{15}{15}{A version of Takagi Lectures (27--28 June 2015, Tohoku University, Sendai, Japan)}}

\bigskip
\bigskip

\begin{abstract}
The survey covers several topics related to the
asymptotic structure of various combinatorial and analytic objects such as the path spaces  in graded graphs (Bratteli diagrams), invariant measures with respect to countable groups, etc.  The main subject is the asymptotic structure of filtrations and a new notion of standardness. All graded
graphs and all filtrations of Borel or measure spaces can be divided into two classes: the standard ones, which have a regular behavior at infinity, and the other ones. Depending on this property, the list of invariant measures can either be well parameterized or have no good parametrization at all. One of the main results is a general standardness criterion for filtrations. We consider some old and new examples which  illustrate the usefulness of this point of view and the breadth of its applications.
\end{abstract}
\newpage
\tableofcontents

\section{Introduction. What is the asymptotic theory of algebraic and combinatorial objects}

In this survey, I will describe several facts which belong to various areas of mathematics, such as  functional analysis, dynamical systems, representations, combinatorics, random processes, etc., and which can be briefly formulated as the \textbf{\it asymptotic theory of inductive limits in various categories}. The fundamental object
for a theory of this type is a graded graph, or a branching graph (Bratteli diagram); it was originally defined in the theory of $AF$-algebras, but then it became clear that the role of this simple notion is of much wider importance. {\it Most important is the structure and asymptotics of the space of paths of this graph}. The so-called tail filtration in the space of paths
can be regarded from the viewpoint  of the theory of filtrations. In particular, the notion of a standard filtration
allows one to give a preliminary classification of graded graphs. The classification of metric spaces with measures and its generalization give further invariants
of filtrations and, consequently, graded graphs.

If we equip a graded graph with an additional structure (such as a lexicographic order on the paths, or cotransition, or the tail filtration, etc.), we obtain a very rich theory which is related to many areas of mathematics.

\medskip
  First of all, I want to emphasize that there are two main problems concerning a graded graph:

 1/ To list the so-called central, or invariant, measures (probability or not) on the paths of the graph. This will be one of the fundamental problems for us. We will see that many questions from representation theory, the theory of Markov processes, as well as from ergodic theory, group theory, asymptotic combinatorics, can be reduced to this problem.

2/ To find typical objects and their asymptotics,
representations, Young diagrams, generic configurations,
limit shapes with respect to statistics and invariant measures on the space of paths.

There are many other problems related to the above ones, such as the calculation of the K-functor of the algebra (group) with a given branching graph, the analysis of the generating functions of ``generalized binomial coefficients,'' which appear in  combinatorics and statistical physics, etc.

\medskip

These questions are related to what in the 1970s  I called ``Asymptotic Representation Theory,'' but in this paper I can only briefly mention this, and will talk about a wider understanding of asymptotic theory of  graded graphs:

1) a new part of ergodic theory (adic dynamics);

2) a new look on the theory of various boundaries regarded as sets of invariant measures, and
on the classification of traces and characters in the asymptotic theory of representations;

3) the theory of filtrations (= decreasing sequences of $\sigma$-algebras in measure theory), the
notion of standardness, and the classification of  measurable functions using invariant measures.

\medskip

This article is written as an extension of my talk at the 15th Takagi Lectures, and I more or less follow  the preliminary
text published in \cite{Tak}.

In the second section, we define the main notions related to graphs, the space of paths, boundaries, additional structures, and discuss  links  to dynamics and measure theory.

The third section is devoted to the geometric approach to projective limits and the theory of boundaries; we define the main notions of standardness and intrinsic metric.

In Section~4, we present the current state of the theory of filtrations in measure-theoretic and Borel categories, and define the general notion of standardness.

In Section~5, we illustrate the link between the problem of finding invariant measures and the problem of classification of measurable functions of several variables.

Section~6 contains examples. Some of them are old, but we  also give recent examples of exit (or absolute) boundaries for random walks on trees and for invariant random subgroups\footnote{IRS appeared simultaneously and independently in several areas \cite{Ab,Fr2010,Fr2012}, in particular, in connection with totally nonfree actions
with an invariant measure and the structure of factor representations of countable groups.}  of the infinite symmetric group.
   The last example is closely related to the theory of characters of the infinite symmetric group and to our  model of factor representations of type II$_1$ for this group \cite{VeK81D}.

I produced many (perhaps, not all) references to known theorems. The proofs of the new results mentioned in the paper will be published in an article which is currently in preparation.
\medskip

{\bf Acknowledgments.} The anonymous referees provided very important and detailed remarks, questions, and suggestions on the exposition of  the paper.  N.~Tsilevich helped with the language, all figures presented  in this article were prepared by A.~Minabutdinov. To all of them the I express my deep gratitude.

\newpage

\section{The combinatorial and dynamical theory of  ${\Bbb N}$-graded graphs}

 \emph{${\Bbb N}$-graded graphs (= Bratteli diagrams), special structures on the space of paths, the tail filtration, central measures, the exit boundary, adic dynamics.}

In this section, we define the main structures on a branching graph  and its Markov interpretation, the lexicographic order and the adic transformation, formulate the list of specific problems on invariant and central measures.

We formulate the main problem, that of the description of the ergodic Markov measures with a given set of cotransition probabilities and, in particular, the description of the set of central measures on the space of paths.
The notion of an adic transformation and ``Bratteli--Vershik diagrams'' provides a kind of new universal dynamics and opens a new direction in ergodic theory. The Markov interpretation of a graph gives a new approach to the problem of
different kinds of boundaries in harmonic and probabilistic analysis. We also obtain a universal model in the metric theory of filtrations.

\subsection{Locally finite $\mathbb N$-graded graphs, path space, tail filtration, group of transformations}

Consider a locally finite, infinite $\mathbb N$-graded graph~$\Gamma$ (= Bratteli diagram).
The set of vertices graded by $n \in \{0,1, \dots\}$ will be denoted by $\Gamma_n$ and called the $n$th {\it level} of~$\Gamma$:
 $$\Gamma=\coprod_{n\in \mathbb N} \Gamma_n;$$
for the $0$th level, we have $\Gamma_0=\{\emptyset\}$, that is, it consists of the single vertex~$\emptyset$.  We assume that every edge joins two vertices of neighboring levels, every vertex has at least one successor, every vertex except the initial one has at least one predecessor. In what follows, we also assume that the edges of $\Gamma$ are simple,\footnote{For our purposes, allowing Bratteli diagrams to have multiple edges does not give anything new, since the cotransition probabilities introduced below must be replaced and generalized in the case of multiplicities of edges. But in the framework of general filtration theory, multiple edges are needed.} and no other assumptions are imposed (see Figure 1).

The graph  $\Gamma$ can, obviously, be defined by the sequence of $0-1$ matrices $M_n$, $n=1,2,\dots$, where $M_n$ is the  $|\Gamma_{n-1}|\times |\Gamma_n|$ adjacency matrix for the bipartite graph $\Gamma_{n-1} \cup \Gamma_n$.
A very special case of a graded graph is as follows: all levels $\Gamma_n$ are identified with each other and all adjacency matrices are the same; these are so-called {\it stationary graded graphs} (they correspond to stationary Markov chains, see the next section).
\begin{figure}[h!]
\center{\includegraphics[scale=2]{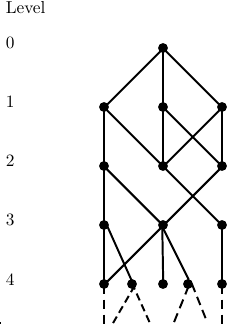}}
\caption{A graded (Bratteli) diagram.}
\label{fig:diagramVertical}
\end{figure}

\medskip
It is well known (see \cite{Brat}) how one can construct a locally semisimple algebra
 ${\cal A}(\Gamma)$ over $\mathbb C$ canonically associated with a graded graph $\Gamma$: this is the direct limit of sums of matrix algebras:
 $${\cal A}(\Gamma)=\lim_n \{{\cal A}_n; I_n\},$$
where
$${\cal A}_n=\sum_{v\in \Gamma_n}{\mathbb M}_{l(v)}(\mathbb C);$$
 here $l(v)$ is the number of paths between $\emptyset$ and $v \in \Gamma_n$; the restriction of the embedding $I_n:{\cal A}_n \mapsto {\cal A}_{n+1}$ to each subalgebra ${\mathbb M}_{l(v)}$ is  the block diagonal embedding of ${\mathbb M}_{l(v)}$ to all algebras ${\mathbb M}_{l(u)}$ for which the vertex $u\in \Gamma_{n+1}$ follows the vertex $v\in \Gamma_n$.

  However, here we do not consider the algebra ${\cal A}(\Gamma)$ in detail, and do not discuss the fundamental relation of the notions introduced below with this algebra and its representations; this problem is worth a separate study. This important link between algebras and graphs has been studied in many papers; this is the so-called theory of AF-algebras etc. See
\cite{Brat,Br,EffHa,Strat,Ell,Ror,VeKe}.

\medskip
 A path $t$ in $\Gamma$ is, by definition, a (finite or infinite) sequence of edges of $\Gamma$ starting at the initial vertex $\emptyset$ in which the end of every edge is the beginning of the next edge (for graphs without multiple edges, this is the same as a sequence of vertices with the appropriate condition). The space of all infinite paths in  $\Gamma$ is denoted by $T(\Gamma)$. It is a very important object for us; in a natural sense, it is the inverse limit of the spaces of finite paths (leading from the initial vertex to vertices of some fixed level), and thus is a Cantor-like compact set with the weak topology. Cylinder sets in $T(\Gamma)$ are sets defined in terms of conditions on initial segments of paths up to level $n$; they are clopen (= closed and open) and determine a base of the topology of $T(\Gamma)$. There is a natural notion of {\it tail equivalence relation} $\tau_{\Gamma}$ on $T(\Gamma)$: two infinite paths are tail-equivalent if they eventually coincide; one also says that such  paths lie in the same block of the {\it tail partition}.

  The {\it tail filtration} $\Xi(\Gamma)=\{\mathfrak{A}_0\supset \mathfrak{A}_1 \supset \cdots\}$
  is the decreasing sequence of $\sigma$-algebras $\mathfrak{A}_n$, $n \in \mathbb N$,  where $\frak {A}_n$  consists of all Borel sets $A\subset T(\Gamma)$ such that along with every path
  $A$ contains all paths coinciding with it from the $n$th level. In an obvious sense,  ${\frak A}_n$  is complementary to the finite $\sigma$-algebra of cylinder sets of order
  $n$. The key idea is to apply the theory of decreasing filtrations  to the analysis of the structure of path spaces and measures on them.

\begin{definition}
On the path space  $T(\Gamma)$, we define the {\it tail partition $\xi_{\Gamma}$ and the {\it tail equivalence relation $\tau_{\Gamma}$}}: two paths are in the same class of $\tau_{\Gamma}$, or belong to the same element of
  $\xi_{\Gamma}$, if they eventually coincide.
  \end{definition}

  The equivalence relation $\tau_{\Gamma}$ is a {\it hyperfinite equivalence relation}, which means that it is the limit of the decreasing sequence of finite relations $\tau_{\Gamma}^n$,  which are defined in the same way with the superscript $n$ meaning that the corresponding class of paths consists of paths coinciding starting from the $n$th level.

   Let us introduce a group of transformations of the path space  $T(\Gamma)$. Note that for every vertex $v\in \Gamma_n$, the set of all finite paths from $\emptyset$ to $v$ has a natural structure of a tree of height $n$, and, by definition, the group $G_n^v(\Gamma)\equiv G_n^v$ is the finite group of transformations of $T(\Gamma)$ that is the group of automorphisms of this tree. Consider the group $\sum_{v\in \Gamma_n}G_n^v\equiv G_n(\Gamma)$; this is the group of transformations which can be called cylinder transformations of rank $n$. The sequence of groups $G_n(\Gamma)$, $n=1,2, \dots$, increases monotonically with respect to the natural embeddings $G_n(\Gamma)\subset G_{n+1}(\Gamma)$, and finally we obtain the group of all {\it cylinder transformations}:
  $$G(\Gamma)=\operatorname{lim ind}_n G_n(\Gamma).$$
  It is clear that any element of the group $G(\Gamma)$ preserves the tail equivalence relation $\tau(\Gamma)$; moreover,
  $G(\Gamma)$ is the group of all transformations of the space $T(\Gamma)$ that fix all classes of the tail partition $\tau(\Gamma)$.

   Later we will define more general {\it adic transformations of paths}, which are defined not for all paths, but, in a natural sense, are limits (in measure) of sequences of cylinder transformations.

The properties of the graphs and groups defined above are very different for various graded graphs and must be analyzed carefully.

\medskip
Let us give a list of first examples of graphs:

$\bullet$ Stationary graphs (e.g. Fig.~\ref{fig:diagramVertical}), for which all levels and all sets of edges between two levels are isomorphic, e.g., the dyadic graph and the Fibonacci graph (Fig.~\ref{fig:diagramFibonacci}).

$\bullet$ Classical graphs: the Pascal graph (Fig.~\ref{fig:Pascal}), the Euler graph, the Young graph (Fig.~\ref{fig:Young}), their multidimensional generalizations.

$\bullet$ More complicated examples: the graph of unordered pairs, the graph of ordered pairs, Hasse diagrams of the general posets, etc.

The author believes that these objects are hidden in many mathematical problems and the study of asymptotic problems
related to graded graphs is especially important.

\begin{figure}[h!]
{\includegraphics[scale=1]{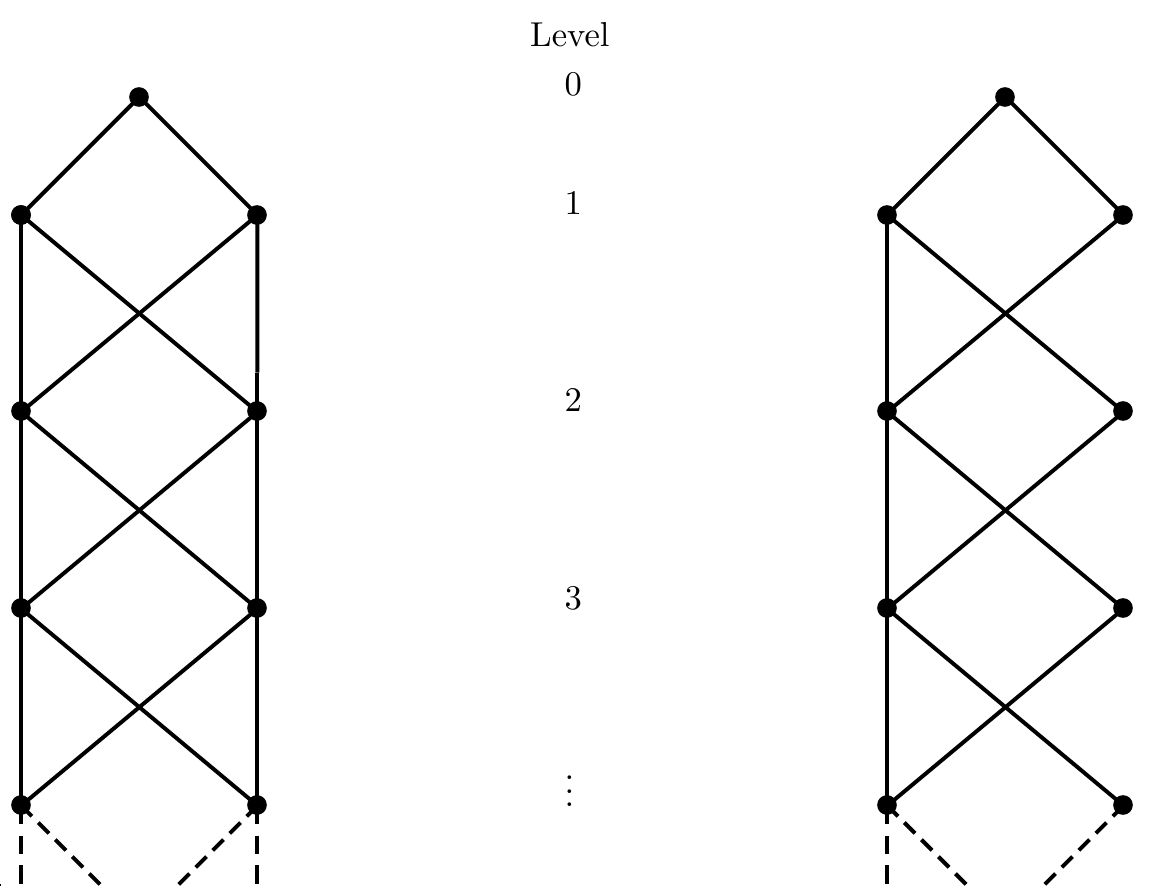}}
\caption{The dyadic diagram (left), the Fibonacci diagram (right).}
\label{fig:diagramFibonacci}
\end{figure}

\newpage

\subsection{The Markov interpretation of graded graphs,
equipment structure, central measures, boundaries}

\subsubsection{Toward a Markov compactum and measures of maximal entropy (central measures)}

Now we will consider the same object --- the space $T(\Gamma)$ of all infinite paths of a graded graph $\Gamma$ --- from another point of view. If we rotate the above picture of a graded graph (with the initial vertex on the top), see Figure~1, by 90 degrees counterclockwise, we obtain a picture that is well known to probabilists (see Figure~3).

\begin{figure}[t]
\center{\includegraphics[scale=0.5]{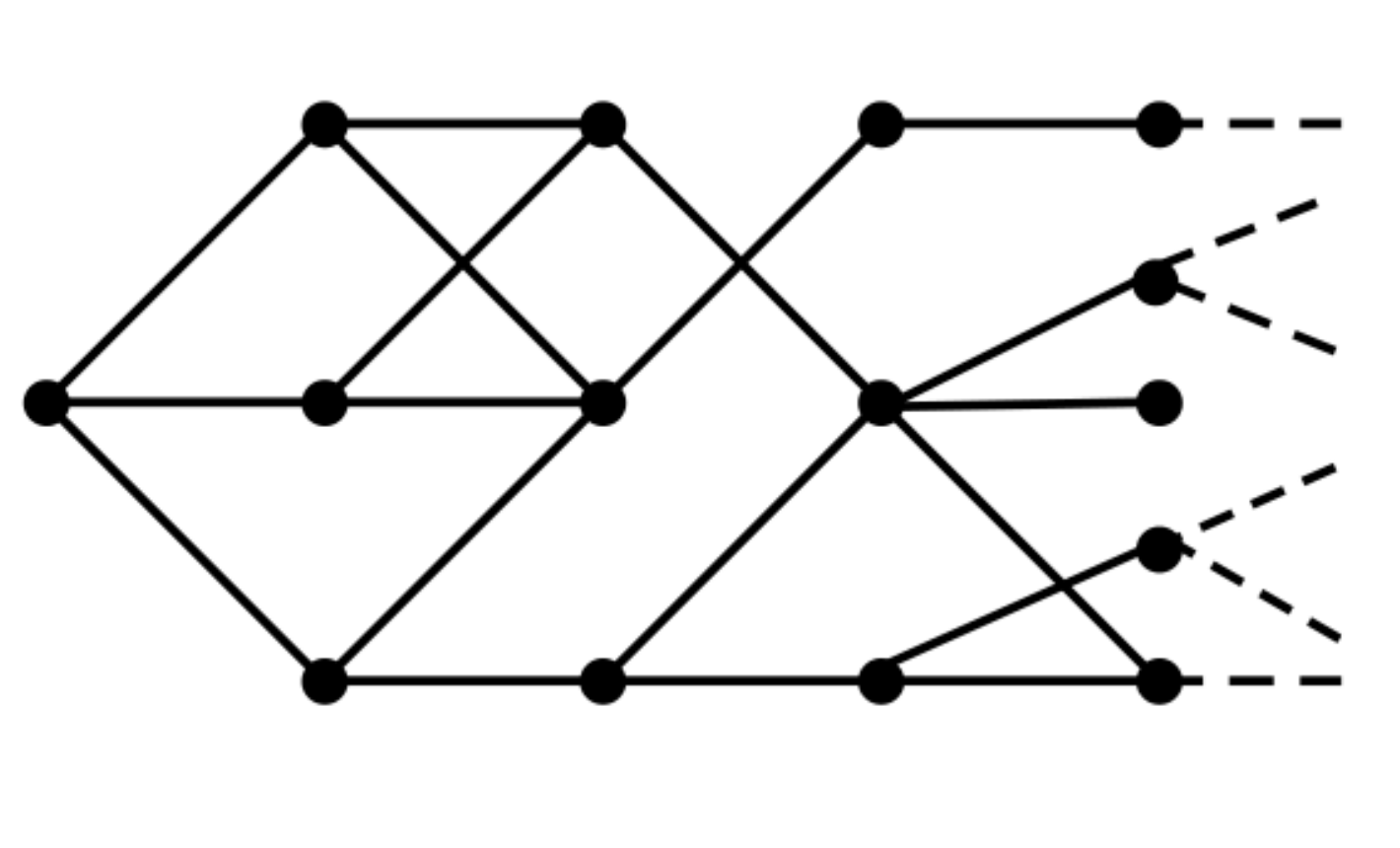}}
\caption{A Markov compactum.}
\label{fig:diagramHorizontal}
\end{figure}

 Let us regard the $\Bbb N$-grading of our graph as the discrete time $0,1,\dots$ of a topological Markov chain (in general, nonstationary) and the set of vertices of level $n$ as the state space of the chain at the time $n$.  We can view a path $\{v_n\}_{n+0}^{\infty}$ as a {\it trajectory} of the process, and the whole space of paths as the space of trajectories of the {\it Markov topological chain}; the transitions of this chain are determined by the matrices $M_n$
 defined above. We do not fix any probability measure on the space of trajectories.

The well-known notion of a (stationary) topological Markov chain (see \cite{Erg}) is a special example of our definition: in this case, all levels are mutually isomorphic and the sets of transitions do not depend on the levels.

So, in the study of the path spaces $T(\Gamma)$ of graded graphs $\Gamma$, it is convenient  to use the terminology and theory of Markov chains, more precisely, the  theory of one-sided {\it Markov compacta}, not stationary in general.
However, as compared to the stationary case,
 many examples of  graded graphs give completely new examples of the behavior of Markov chains.  After the rotation, the combinatorial and algebraic world associated with Bratteli diagrams turns into the probabilistic and dynamical world of Markov chains. This link is extremely important and fruitful, especially for us, because we will consider
probability measures on the path space  $T(\Gamma)$.

 Recall the notion of a Markov probability measure on a
 Markov compactum; this is a measure $\mu$ with the following property: for every $n$, the conditional measure of $\mu$ under the condition $v_n=v\in \Gamma_n$ is the direct product of a measure on $\prod_{k=0}^{n-1}\Gamma_k$
and a measure on $\prod_{k=n+1}^{\infty}\Gamma_k$. In other words, the past and the future are independent for every fixed state at time $n$ and for every $n=1,2, \dots$.
We will consider the theory of Markov measures on the path space $T(\Gamma)$. The following special case of Markov measures is very important in what follows.

\begin{definition}
A Markov measure $\nu$ on  $T(\Gamma)$ is called a central measure if for every vertex $v$ the conditional measure $\nu_v$ induced by  $\nu$ on the finite set of all finite paths that join the initial vertex $\emptyset$ with $v$ is the uniform measure.
\end{definition}

It is clear from the definition that any cylinder transformation preserves any central measure. In the case of a stationary Markov compactum, central measures are called, for a certain reason, measures of {\it maximal entropy}.

 The notion of a central measure on the space $T(\Gamma)$ is determined intrinsically by the structure of the branching graph $\Gamma$. The set of all central measures on the path space  $T(\Gamma)$ will be denoted by $\Sigma(\Gamma)$; this is a Choquet simplex with respect to the ordinary convex structure on the space of probability measures with the weak topology, see \cite{Fel}. The set of extreme points (Choquet boundary) of this simplex is the set of ergodic central measures, and we denote it by $\Erg(\Gamma)$.
 Any central measure can be uniquely decomposed into an integral over the set of ergodic measures. The set $\Erg(\Gamma)$ is of most interest to us. Note that for every ergodic central measure,
 the action of the group $G(\Gamma)$ of cylinder transformations
 on the space $T(\Gamma)$ is ergodic in the sense of ergodic theory (no nontrivial\footnote{Here the word ``nontrivial'' means that the measure of the subset is not equal to zero or one.} invariant measurable subsets), and vice versa: if for a central measure $\mu$, the action of the group $G(\Gamma)$ is ergodic, then this measure is ergodic as a central measure.

\subsubsection{Cotransition probabilities and an equipment of a graded graph}

 Now we introduce an additional structure on a graded graph, in order to extend the notion of central measures. Namely, we define a {\it system of cotransition probabilities}, which we call a {\it $\Lambda$-structure},
 $$\Lambda=\{\lambda=\lambda_v^u;\; u\in \Gamma_n, v\in \Gamma_{n+1}, (u,v)\in \mbox{edge}(\Gamma_n,\Gamma_{n+1}),\;n=0,1, \dots\},$$
 by associating with each vertex $v \in \Gamma_n$, a probability vector whose component
  $\lambda_v^u$ is the probability
 of an edge $u\prec v$ entering $v$ from the previous level; here $\sum\limits_ {u:\, u\prec v}\lambda_v^u=1$ and $\lambda_v^u > 0$.
 We emphasize that  a $\Lambda$-structure (e.g., cotransition probabilities) is defined for all vertices $v,u\in \Gamma$ with  $u\prec v $, and $\lambda_v^u$ may not have zero values.

\begin{definition}
An equipped graph is a pair $(\Gamma, \Lambda)$ where $\Gamma$ is a graded graph and $\Lambda$ is a
$\Lambda$-structure, i.e., a system of cotransition probabilities on its edges.
 \end{definition}

The term ``cotransition probabilities'' is borrowed from the theory of Markov chains:
if we regard the vertices of $\Gamma$ as the states of a Markov chain starting from the initial state $\emptyset$ at time $t=0$,
and the numbers of levels as moments of time, then  $\Lambda=\{\lambda_v^u\}$ is interpreted as the system of cotransition probabilities for this Markov chain:
 $$\mbox{Prob} \{x_{t}=u|x_{t+1}=v\}=\lambda_v^u. $$

In the probability literature (e.g., in the theory of random walks),  cotransition probabilities are usually defined not explicitly, but as the cotransition probabilities of a given Markov process. We prefer to define them directly, i.e., include them into the input data of the problem.

Recall that in general a system of cotransition probabilities does not  uniquely determine the transition probabilities $\mbox{Prob} \{x_{t+1}=v|x_t=u\}$. At the same time, since  the initial distribution is fixed (in our case, it is the $\delta$-measure at $\emptyset$), the transition probabilities uniquely determine the list of cotransition probabilities. So, every Markov measure on $T(\Gamma)$ determines a $\Lambda$-structure.

The most important special case of a system of cotransition probabilities, corresponding to the central measures which we  have already defined,  is the following one:
 $$
 \lambda_v^u=\frac{\dim(u)}{\sum\limits_ {w:\,w\prec v} \dim(w)}=
 \frac{\dim(u)}{\dim(v)},
 $$
where $\dim(u)$ is the number of paths leading from the initial vertex $\emptyset$ to  $u$ (i.e., the dimension of the representation of the algebra $A(\Gamma)$ corresponding to the vertex $u$). In other words, the probability to get from $v$ to $u$ is equal to the fraction of paths that lead from   $\emptyset$ to $u$ among all the paths that lead from
 $\emptyset$ to  $v$. This system of cotransition probabilities is canonical, in the sense that it is determined by the graph only.
  Central measures have been studied  very intensively in the literature on Bratteli diagrams, as well as in combinatorics, representation theory, and algebraic settings, but mainly for specific diagrams (see \cite{VeK81,VeK85,KOO,GK,GO}). In terms of the theory of
C$^*$-algebras, central measures are nothing more than traces on the algebra $A(\Gamma)$, or characters of locally finite groups in the case when the graded graph corresponds to a group algebra. Ergodic central measures correspond to  indecomposable traces or characters.

It is convenient to regard a system of cotransition probabilities as a system of $d_n\times d_{n+1}$ Markov matrices:
 $$
 \{\lambda_v^u\}, \quad u \in \Gamma_n, v \in \Gamma_{n+1};\quad |\Gamma_n|=d_n,\; |\Gamma_{n+1}|=d_{n+1}, \;n \in \mathbb N;
 $$
 these matrices generalize the  $0-1$ adjacency
  matrices of the graph
 $\Gamma$.
 Our main interest lies in the asymptotic properties of this sequence of matrices. In this sense, the whole theory developed here is a part of the asymptotic theory of infinite products of Markov matrices, which is important in itself.

\subsubsection{Measures, central measures, boundaries}

A measure on the path space of a graph is called {\it ergodic} if the tail $\sigma$-algebra (i.e., the intersection of all  $\sigma$-algebras of the tail filtration) is trivial $\bmod 0$,\footnote{The symbol ``$\bmod 0$'' means that the object or notion preceding it is understood up to changes on a subset of zero measure.} i.e., consists of two elements.

A Markov measure $\mu$ agrees with a given system $\Lambda$
of cotransition probabilities if the collection of cotransition probabilities of $\mu$ (for all vertices) coincides with $\Lambda$.

\begin{definition}  Denote by $\Sigma(\Gamma)_\Lambda$ the set of all  Markov measures on $T(\Gamma)$ with  cotransition probability $\Lambda$.
The set of ergodic Markov measures from $\Sigma(\Gamma)_{\Lambda}$ will be denoted by $\Erg(\Gamma)_{\Lambda}$.

  The set of all central measures on the path space  of a graph $\Gamma$ will be denoted by $\Sigma(\Gamma)$, and the set of ergodic central measures, by $\Erg(\Gamma)$.
  The  list of measures $\Erg(\Gamma)_{\Lambda}$ will be called the absolute boundary of the equipped graph  $(\Gamma,\Lambda)$. The set of ergodic central measures will be called the \textbf{absolute boundary of the graph $\Gamma$} and denoted by $\Erg(\Gamma)$.\footnote{We use the term ``absolute boundary'' instead of other terms, such as ``exit,'' ``entrance,'' Martin boundary, etc. It seems that in specific situations,  such as the theory of Markov processes,  these terms (which were used by E.~Dynkin) are natural, but in the context of graded graphs and general dynamics it is better to have a more neutral term. It is important that the absolute boundary is an invariant of an ergodic equivalence relation, while the Martin boundary is not: it depends on an approximation of this relation (see \cite{2014,VM2015}).}
 \end{definition}

We will see that $\Sigma_{\Gamma}(\Lambda)$ is a projective limit of finite-dimensional simplices.

 The absolute boundary is a topological boundary, and, as we will see, it is the Choquet boundary of a certain simplex (a projective limit of finite-dimensional simplices).

\begin{problem}\label{problem1} Enumerate the set $\Sigma(\Gamma)_{\Lambda}$ of all Markov measures with a given system of cotransition probabilities $\Lambda$ and, in particular, the set of ergodic measures $\Erg(\Gamma)_{\Lambda}$, and to study its asymptotic behavior.\footnote{Recall that to describe a Markov measure on the path space  means to describe its transition probabilities.}
\end{problem}

\begin{remark}{\rm
It may happen that for some measure from  $\Sigma(\Gamma)_{\Lambda}$ and for a given vertex $v\in \Gamma$,
the measure of paths that go through $v$ vanishes.
This means that the measure is concentrated on the path space  $T(\Gamma')$ of some subgraph $\Gamma'\subsetneqq \Gamma$, for whose vertices the cotransition probabilities  are positive.
}\end{remark}

The asymptotic behavior of central measures can be very different even for the same graph.  For example, in the case of the graph of unordered pairs (see below), there are central measures with chaotic behavior, as well as those whose behaviour  is more smooth (``standard'' in the sense which will be defined later). On the contrary, for classical graphs such as the Pascal graph, the Young graph, etc., all central measures have a more regular character (``standard''), in particular, we have so-called ``limit shape theorems.''

 Recall that the {\it Poisson--Furstenberg boundary of a given Markov measure on $T(\Gamma)$} is its tail measure space, or the quotient space over the tail equivalence relation. This boundary is regarded as a measure space and, in some sense, it is only a part of the absolute boundary.

In connection with cotransition probabilities, it makes sense to point out the following general terminology which does not use a graded structure on the space of paths. The system of cotransition probabilities allows us to define a {\it cocycle\footnote{A cocycle on an equivalence relation is a function (in our case, with values in ${\mathbb R}_+$) on the set of pairs of equivalent elements satisfying the following properties:
$c(\alpha,\beta)c(\beta,\alpha)=1$, $c(\alpha,\beta)c(\beta,\gamma)=c(\alpha,\gamma)$.} on the tail equivalence relation}, i.e., an ${\mathbb R}_+$-valued function $(\gamma_1,\gamma_2) \mapsto c(\gamma_1,\gamma_2)$  on the space of pairs of tail-equivalent paths, as the ratio of the conditional measures of these two paths, or the ratio of the products of cotransition probabilities along the paths:
$$c(\gamma_1,\gamma_2)=\frac {\prod \lambda_{a_i}^{a_{i-1}}}
{ \prod \lambda_{b_i}^{b_{i-1}}},$$
where $\gamma_1=\{a_1,a_2,\dots, a_k,\dots\}$, $\gamma_2=\{b_1,b_2,\dots, b_k,\dots\}$, $a_n=b_n$, $n>k$
 (the product is well defined, because the ratio is finite).

Consider a measure that agrees with the tail equivalence relation. For any two paths that coincide starting from the $n$th level, for every $m>n$, the ratio of the conditional measures of the partition $\xi_m$ into classes of paths that coincide starting from the $m$th level does not depend on $m$, and thus we have a well-defined cocycle.\footnote{For every subrelation of an equivalence relation with finite blocks, we have the usual conditional measures, and the ratio of the conditional measures of two points in a block does not depend on the choice of this subrelation. This is a simple and fundamental transitivity property of conditional measures which is never mentioned in  textbooks and which holds not only in the hyperfinite case. In the framework of the theory of dynamical systems, the cocycle is simply the Radon--Nikodym density, and the set of measures with a given cocycle is the set of quasi-invariant measures with a given Radon--Nikodym density.} So, our main problem~\ref{problem1} is the problem of describing the probability measures on the path space with a given cocycle.

Note that if an equivalence relation is the orbit partition for an action of a group  with a quasi-invariant measure, then the cocycle coincides with the Radon--Nikodym cocycle for the transformation group   (see, e.g., \cite{Schm}):
  $$c(g\alpha,\alpha)=\frac{d\mu(g\alpha)}{d\mu(\alpha)}.$$

 In our case, the cocycle has a special form (the product of probabilities over edges) and is called a Markov cocycle.

 \begin{remark} {\rm It is possible to generalize the notion of  cotransition probabilities and define an equipped graph
 for any oriented graphs: one can define an arbitrary system of probabilities on the set of ingoing edges of each vertex. The problem is still to describe the absolute boundary, i.e., the collection of all ergodic measures on the set of directed paths with given conditional entrance probabilities. This generalization could give  interesting new examples of exit boundaries for general graphs.}
\end{remark}

\subsubsection{Borel equivalence relations}

Assume that in a standard Borel space $X$ a hyperfinite equivalence relation~$\tau$ is defined; this means that
$\tau$ is an increasing limit of a sequence of Borel
equivalence relations $\xi_n$, $n=1,2, \dots$,\footnote{The term ``Borel'' means that $\xi_n$ is the partition into the preimages of a Borel map defined on $X$.} with finite equivalence classes.
It is not difficult to prove the following proposition.

\begin{statement}
For every pair $(X,\tau)$ where $X$ is a standard Borel space and $\tau $ is a hyperfinite equivalence relation on $X$  there exists a graded graph $\Gamma$ and a Borel isomorphism between $(X,\tau)$ and $(T(\Gamma), \tau_{\Gamma})$ that sends $\tau $ to the tail equivalence relation $\tau_{\Gamma}$.  Every ergodic Borel measure on $X$ with a given cocycle defined for the equivalence relation $\tau$ corresponds under this isomorphism to an ergodic Markov measure on the equipped graph $T(\Gamma)$. In particular, an invariant ergodic measure on the equivalence relation $\tau$ corresponds to a central measure on
$T(\Gamma)$.
\end{statement}

This proposition is essentially known (see \cite {Kech,Schm,VEq}), but it is usually considered in the framework of group actions.

Thus, the general theory of hyperfinite Borel equivalence relations is a special case of our theory of Markov measures
for some graph $\Gamma$. But an additional structure is
a fixed approximation of the tail equivalence  relation which we have on graded graphs.

From this point of view, we try to construct a theory of realizations of hyperfinite equivalence relations on a standard Borel space as tail equivalence relations on the path space $T(\Gamma)$.

In the category of Borel spaces, the classification of hyperfinite equivalence relations was obtained in \cite{Kech}; in the measure-theoretic category, by the famous Dye theorem, there is only one, up to isomorphism, ergodic hyperfinite invariant relation.

But we want to consider another, more delicate, category, with a more detailed notion of isomorphism. In brief, {\it it is the category of spaces of the type $T(\Gamma)$, or, more exactly,  Cantor spaces equipped with a decreasing filtration of finite type},  with ``asymptotic isomorphisms'' as morphisms (\cite{2015}). The meaning of these notions will be discussed later in the section on filtrations.

  \subsection{A lexicographic ordering, the adic transformation, and the globalization of Rokhlin towers}

\subsubsection{The definition of the adic transformation}

In this section, we define another additional structure on a graded graph: a linear order on each class of tail-equivalent paths. We will call it an ``adic structure'' on the graded graph. It is similar to a $\Lambda$-structure on an equipped graded graph, but has different applications.

We start with the definition of a local order on the set of edges with a given endpoint, and then define a lexicographic ordering on the paths.

\begin{definition}
Let $\Gamma$ be a graded graph; for each vertex
$v\in \Gamma$, define a linear order $\operatorname{ord}_v$ on the set of ingoing edges of $v$. Consider two paths   $\{t^i_k\}_{k=1}^{\infty}$, $i=1,2$, where $t^i_k$ is the edge
 that joins vertices of levels $k-1$ and $k$. If these paths belong to the same class of the tail equivalence relation, then for some minimal $n$ the edges $t^i_k$ with $k>n$ coincide; if $v$ is the first common vertex of both paths, then
   $$t^1>t^2$$  if and only if $t^1_{n+1} > t^2_{n+1}$  in the
sense of the order  $\operatorname{ord}_v$ on the edges.
This definition makes sense also for graded graphs with muliple edges. If there are no multiple edges, then the simplest way to define an order  on the ingoing edges is to define an order on the vertices of each level, and then  introduce an order on the ingoing edges as the order on the corresponding vertices.
\end{definition}

It is obvious that this definition gives a linear (lexicographic ``from below'') ordering on each class of the tail equivalence relation.

Consider the subset $T_0(\Gamma)$ of all paths from $T(\Gamma)$ that have the preceding and the following paths in the sense of this ordering. For a large and interesting class of graphs, $T_0(\Gamma)$ is a generic (dense open) subset of $T(\Gamma)$; moreover, we can restrict ourselves to the case where there are only two exceptional paths, as in the Pascal graph (see \cite{2013}).

Now we are ready to define an action of the group $\Bbb Z$ on the set $T_0(\Gamma)$ as follows: the generator acts as the transformation $P$ that sends a path $t$ to the next path in the sense of our ordering; this transformation is called the {\it adic transformation}.\footnote{Sometimes, the adic transformation is called the ``Vershik transformation,'' and a branching graph equipped with a lexicographic ordering is called a ``Bratteli--Vershik diagram,'' see \cite{V81}.}

The simplest example is a lexicographic ordering in the {\it dyadic graph}; all positive levels of this graph consist of two vertices, and any two vertices of neighboring levels are joined by an  edge. If we identify a path in this graph with a number from the interval $[0,1]$, then we have a natural linear ordering defined on the classes of irrational numbers from $[0,1]$ that differ by a dyadic rational number: for two numbers $t_1,t_2 \in [0,1]$ with dyadic rational difference $t_1-t_2$,
the greater one is that for which the first different digit in the dyadic decomposition is $1$. The corresponding adic transformation is the so-called odometer. The word ``adic'' is the result of deleting $p$ from the word ``$p$-adic.''

 This type of dynamics for the group $\Bbb Z$ (``adic," or ``transversal," dynamics) was defined by the author in 1981. For the stationary case, a similar definition was given by S.~Ito \cite{It}.

 For example the simplest automorphism --- so called odometer, or dyadic shift --- is realized with the dyadic graph (see Fig.~\ref{fig:diagramFibonacci}),
 the shift on the homoclinic point on the 2-torus is realized on
 Fibonacci graph  (see Fig.~\ref{fig:diagramFibonacci}) etc.

 But it is possible to realize any ergodic transformation in this form. The main fact is the following theorem (\cite{V81}).

\begin{theorem}[\cite{V81,V82}]
For every measure-preserving ergodic transformation~$S$ of a standard (Lebesgue) measure space $(X,\mu)$ with a continuous measure there exists a graded graph $\Gamma$ with a Borel probability measure $\nu$ on the path space $T(\Gamma)$ invariant under the adic transformation $P$ such that
$$(X,\mu,S) \sim (T(\Gamma),\nu,P);$$
here $\sim$ means isomorphism $\bmod\, 0$ in the sense of the theory of measure spaces.
\end{theorem}

Related facts can be found in \cite{Pims}. See also several papers which follow the idea of the adic transformation as a transformation of a Cantor space: \cite{Gi} and subsequent papers by the same authors.

 This means that an adic realization of an action of $\Bbb Z$  gives another (as compared with so-called symbolic dynamics) universal model for the dynamics of the group $\Bbb Z$. This approach to dynamics is nothing more than its realization as a sequence of successive periodic approximations which, in a sense, exhaust the automorphism.  The classical Rokhlin lemma about periodic approximations gives a universal periodic approximation of an aperiodic automorphism, but it  provides no information on the measure-theoretical type of the automorphism. Moreover, it shows that there is no finite invariants of aperiodic automorphisms.  An adic realization puts a single Rokhlin tower (not of constant height, in general) into a comprehensive sequence of towers. One may say that we globalize the set of Rokhlin towers.

It is important that an adic realization of a free (aperiodic) action of the group $\Bbb Z$ brings to each orbit an additional structure, namely, the hierarchy that is the restriction of the tail filtration to the orbit. More exactly, for each point (which is a path) $x$, on its orbit $O(x)\sim \Bbb Z$ we have the sequence of partitions (hierarchy) $\xi_n\bigcap O(x)$, and the asymptotic behavior of these partitions gives an important invariant of the automorphism. One can generalize this consideration to actions of amenable groups.

 Of course, the properties of an adic transformation strongly depend on the adic structure --- the linear ordering of the paths.
 For example, the dyadic odometer and the Morse automorphism have realizations on the same dyadic graph, but with  different orderings.

\subsubsection{An example: the Pascal automorphism}

The new point, which appeared in the papers \cite{V81,V82}, was to define adic automorphisms for distinguished graphs. It turns out that this provides a new source of interesting problems in dynamics and ergodic theory.
The simplest nontrivial example  \cite{V82} was the Pascal automorphism\footnote{It is very interesting that this automorphism (without any connection to the Pascal graph, as well as without a name), for a completely different reason, appeared in a paper by S.~Kakutani (see \cite{Kak,Haj,V15}).} $P$ (see Figure~4), which is the adic transformation of paths of the infinite Pascal triangle with the natural lexicographic ordering
(see \cite{2013}).
 Since the path space of the Pascal triangle is
$\prod \{0;1\}$, we can compare the orbit partition of $P$ with that of
 the simplest ergodic automorphism, odometer, which is the  transformation $x\mapsto x+1$ in the compact additive group ${\mathbb Z}_2\cong \prod \{0;1\}$ of dyadic integers.
Clearly, the orbit partition of the Pascal automorphism is finer  than  that of the odometer  and coincides with the orbit partition of the natural action of the infinite symmetric group which permutes coordinates in the product space.

In spite of the simplicity of its definition, the Pascal automorphism
has very interesting and even mysterious properties, see \cite{Mel,Jan,Lod}. For example,
in \cite{Del} a theorem on the Takagi ``bridge" (similar to a Wiener process bridge) was proved, which uses the remarkable Takagi function (see Figure 5).
The main question was about the spectrum of the unitary operator in $L^2({\mathbb Z}_2)$
corresponding to the Pascal automorphism.
Up to now, there is no doubt that this spectrum is pure continuous (and so $P$  is weakly mixing),  this was claimed as a hope in \cite{2013}, but a precise proof is still absent.

\begin{figure}[h!]
\center{\includegraphics[scale=1]{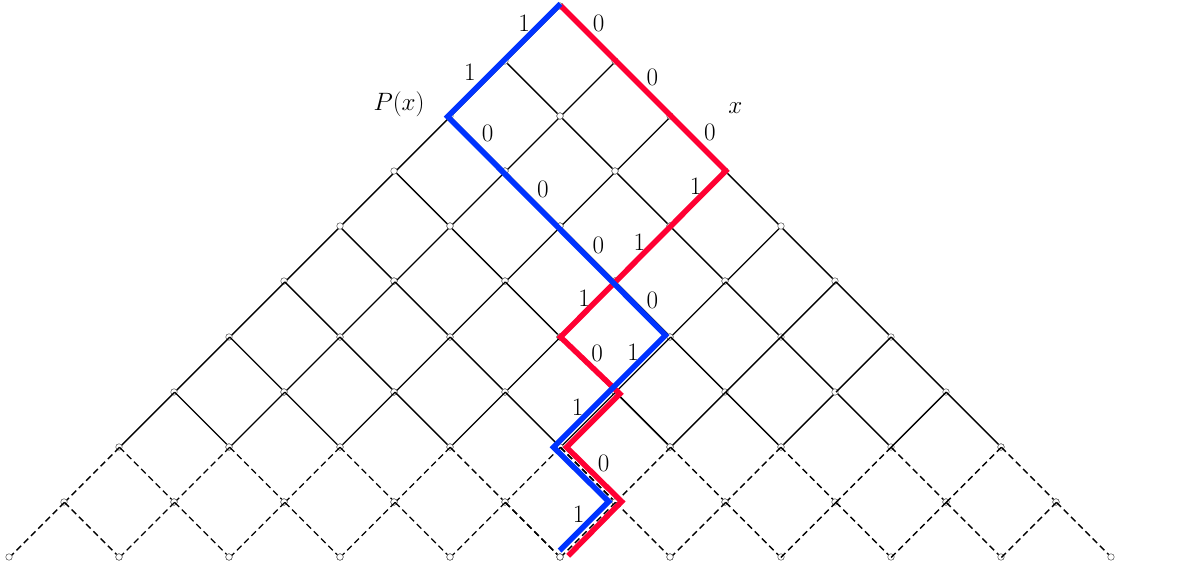}}
\begin{equation} x\mapsto Px; \ \ P(0^{m-s}1^s\textbf{10}\dots)=1^s0^{m-s}\textbf{01}\dots
\label{eq:VershikDef}
\end{equation}\caption{The Pascal automorphism
(Pascal (17th century), Kakutani (1976), Vershik (1981)).}
\label{fig:Pascal}
\end{figure}


\begin{figure}[h!]
\center{\includegraphics[scale=0.4]{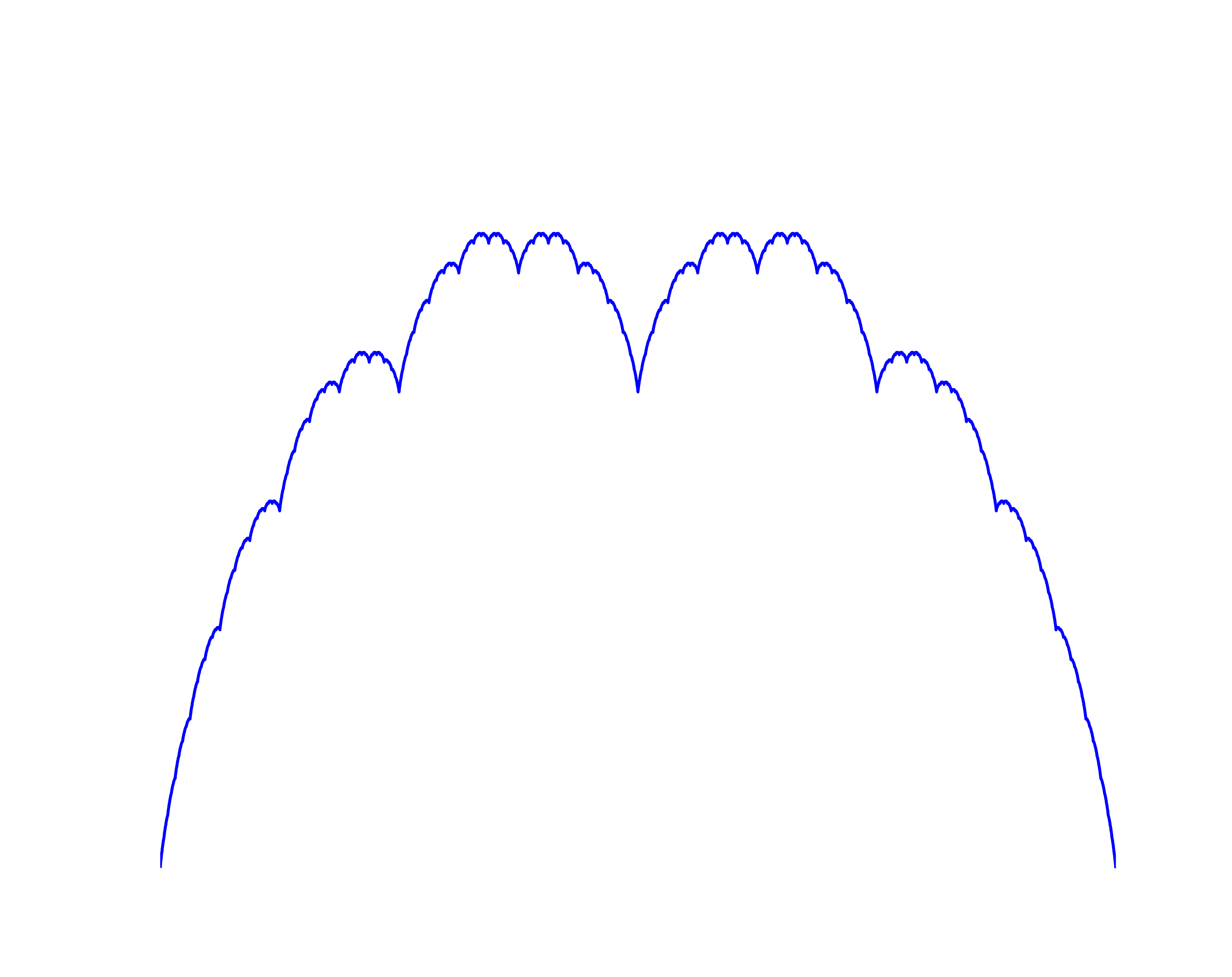}}
\caption{The Takagi curve (1903); the Pascal bridge (2005,
\cite{Jan}).}
\label{fig:TakagiCurve}
\end{figure}

\subsubsection{Adic actions on the space $T(\Gamma)$; the graphs of unordered and ordered pairs}

Adic realizations can be defined for any amenable group. But first we must give an abstract definition of an adic transformation. In general, this transformation is {\it partial}, which means that it is defined not on the whole space~$T(\Gamma)$.

\begin{definition}
A partial transformation $P$ of the path space $T(\Gamma)$ is called  an adic transformation if it preserves
 all classes of the tail equivalence relation, or if it sends each path to an equivalent path. All adic transformations in the sense of the previous section are, by definition, adic in the new sense, too.
\end{definition}

The group (or semigroup) of adic transformations is a subgroup
(or sub-semigroup) of the group of all Borel transformations of the Cantor-like compactum. It is clear that there is a natural approximation of such a transformation with cylinder transformations (see Section~2.2.4). It is a useful
question how to describe this group.

Using our results from filtration theory and the Connes--Feldman--Weiss theorem (\cite{CFW}) on the hyperfiniteness of actions of any amenable group together with some combinatorial arguments, one can prove the following generalization of the theorem on adic models of automorphisms.

\begin{theorem}
For any action $\tau$ of an amenable countable group $G$ on a separable Borel space $X$ there exist a graded graph $\Gamma$ and a Borel isomorphism $\Phi:X \rightarrow T(\Gamma)$ such that for every $g\in G$ the transformation $\Phi \tau(g) \Phi^{-1}$
is an adic transformation on $T(\Gamma)$.
\end{theorem}

For a proof, it suffices to prove that every hyperfinite filtration can be realized as the tail filtration of some graded graph and apply the theorem (\cite{CFW}) on the existence of Rokhlin towers or equivalent facts. This is a generalization of Theorem~1 (see \cite{V81,V82}), but the latter was proved by an explicit construction. Some new details will be given in the new article by the author which was mentioned in the Introduction.

 This theorem gives a globalization of  semihomogeneous Rokhlin towers for actions of a given amenable group with invariant measures. It means that an adic isomorphic realization of an action of an amenable group $G$ with an invariant measure on the space of paths of a graded graph produces
 a sequence of approximations of this action
by actions of a sequence of  finite groups $G_n$, $n=1,2, \dots$; the length of the orbits of the action of the group~$G_n$ can be nonconstant (in contrast to the Rokhlin lemma).

One of the conclusions of this theorem is as follows: the description of  the set of invariant measures  for an  action of an amenable group $G$ on a compact metric space can be reduced to the problem of describing the central measures for some graded graph.

The same is true for quasi-invariant measures with a given
cocycle on the orbit equivalence relation; in this case, we must consider an equipped graded graph with given cotransition probabilities. Note that, by a theorem from \cite{Schm78,Schm}, for an action of a countable group on a standard Borel space, there is a Borel universal measurable set that has measure~1 for all  $G$-quasi-invariant probability Borel measures with a given Radon--Nikodym density (in our terms, with a given cocycle). This means that the theorem can be extended to the case of quasi-invariant measures for amenable groups using equipped graded graphs. Of course, the choice of a graph in the theorem is not unique.

Now we consider {\it universal adic realizations} of actions of a group. Let $G$ be a countable amenable group; assume that we fix a class of actions of $G$ with invariant or quasi-invariant measures; a typical example  is the class of actions that have a generator\footnote{A finite or countable partition $\xi$ of a space $(X,\mu)$ is called a generator of an action of a group $G$ if the product of the shifts of  $\xi$ coincides  $\mu$-$\bmod 0$ with the partition $\epsilon$ of   $(X,\mu)$ into singletons: $\bigvee_{g \in G} g\xi = \epsilon$.  If the number of blocks in $\xi$
is at most $n$, we say that the action has an $n$-generator.} with the number of parts at most $n$.

\begin{definition}
 A graded graph $\Gamma$ with a fixed adic structure
is called universal for a class of actions of the group $G$
with invariant (respectively, quasi-invariant measure) if an arbitrary action from this class is metrically isomorphic to the adic action of the group $G$ on the path space $T(\Gamma)$
with some central (respectively, $\Lambda$-) measure.
\end{definition}

A universal graph for a given class of actions plays the same role for a given group as a symbolic version of actions of  groups. For example, all measure-preserving actions with 2-generators can be realized as (left or right) shifts in the space $2^G$. We give an analog of this fact for adic actions of the  groups $\Bbb Z$ and $\sum_n {\Bbb Z}_2$.

  We introduce two remarkable graded graphs which play an
  important role in this theory.  These are the graphs of ordered ($OP$) and unordered ($UP$) pairs; in a similar way we could consider the graphs of ordered and unordered $k$-tuples, but here we
  will briefly analyze the case of pairs ($k=2$).

 The graph $OP$ of ordered pairs and the graph $UP$ of unordered pairs  (see Figure~\ref{fig:UnPGraph}) are constructed as follows:

 ($\emptyset)$ The initial vertex  is $\emptyset$.

 (1) The first level consists of two vertices $0$ and $1$; they are joined by edges with the vertex $\emptyset$.

 ($n$) The vertices of the $(n+1)$th level  are all ordered (unordered) pairs of vertices of the $n$th level; an
edge between the $n$th and $(n+1)$th levels corresponds to an inclusion of a vertex of the  $n$th level into a pair  which is a vertex of the $(n+1)$th level.\footnote{It is convenient to use  multi-edges (with multiplicity 2) for pairs of the type $(v,v)$ in the graph $UP$, see Figure~\ref{fig:UnPGraph}.}

In order to equip the graphs $OP$ and $UP$ with an adic structure, it suffices to define by induction the order on the set of  pairs.  Assume that we have defined an order on the first level (say, $0<1$) and  on the $n$th level.
Then an order on the $(n+1)$th level is defined as follows. In the case of the graph $OP$ of ordered pairs, we put $(a,b)>_{n+1}(c,d)$ if $a>_n c$
or if $ a=c$ and $b>d$. For the graph $UP$, we put $(a,b)>_{n+1}(c,d)$ if $\max_n(a,b)>_n\max_n(c,d)$, where $\max_n$ means the maximum with respect to $>_n$, or $\max_n(a,b)=\max_n(c,d), \min_n(a,b)>\min(c,d)$.

\begin{figure}[h!]
\center{\includegraphics[scale=1.0]{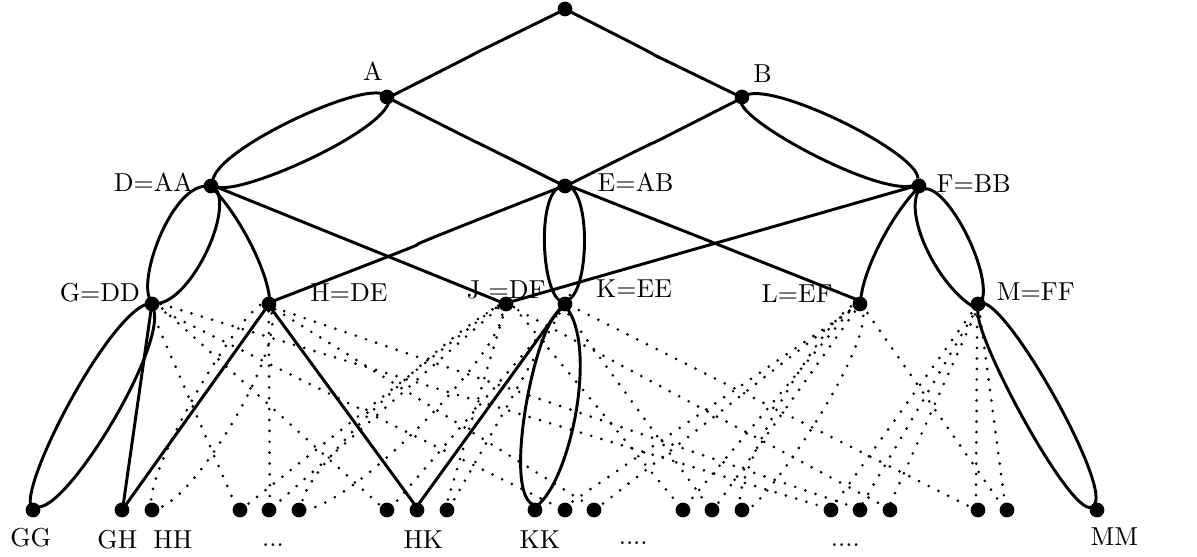}}
\caption{The graph of unordered pairs.}
\label{fig:UnPGraph}
\end{figure}
\begin{theorem}
 Both graphs, the graph of ordered pairs $OP$ and the graph of unordered pairs $UP$, are universal for all actions with $2$-generators for the groups $\Bbb Z$ and $\sum_n {\Bbb Z}_2$. In a similar way one can construct universal graphs for generators with a given number of elements.
\end{theorem}

This fact for the graph $OP$ follows from the analysis of the structure of paths of $OP$. For the graph $UP$, it is not so obvious. The proof in that case uses an important theorem on filtrations which we will discuss later, but formulate here.

\begin{remark}{\rm
1) The vertices of  level $n$ of the graph $UP$
canonically correspond to the orbits of the action
of the group  $\Aut T_n$ of all automorphisms of the dyadic tree
$T_n\sim \{2^n\}$ on the space $2^{T_n}$.

2) The graph $UP$ has another important interpretation: it is a tower of dyadic measures, the set of vertices of level $n$ being the set of probability measures on the vertices of level $n-1$ with possible values $0, 1/2, 1 $; for the corresponding picture of a beginning of the inverse limit of simplices, see Figure~\ref{fig:UnPGraph}.}
\end{remark}

The proof of the universality theorem for the graph $UP$
is based on the universality of this graph for dyadic filtrations, see Section~4.

\begin{theorem}
The tail filtration of the space $T(UP)$ for the graph $UP$ of unordered pairs is universal with respect to dyadic filtrations
in the following sense: every ergodic dyadic filtration of a
standard measure space $(X,\mu)$ with dyadic generator \footnote{The notion of a finite generator for a filtration is the same as for an action of a group: this is a finite measurable partition $\xi$ such that the product $\bigvee_{g}g\xi$ is the partition into singletons,  where $g$ runs over the group of all automorphisms for which all $\sigma$-algebras of the filtration are invariant.} is isomorphic to the tail filtration $(T(UP),\nu)$ with some central measure $\nu$.
\end{theorem}

The proof uses the interpretation of $UP$ as a tower of measures
and the so-called universal projector in the theory of filtrations.
It is very interesting to study the $C^*$-algebras for which $OP$ and $UP$ are the corresponding Bratteli diagrams.

A detailed description of the properties
of the graphs $UP$ and $ OP$ will be given in the forthcoming article.

Universal adic realizations became a source of various
combinatorial constructions of new and paradoxical actions of the groups $\Bbb Z$ and $\sum_n {\Bbb Z}_2$. They will be considered elsewhere.

It is an interesting problem to find universal graphs for other groups. The adic realization of actions of amenable groups
(``adic dynamics") is very different from the classical symbolic realization (= actions by (left or right) shifts in the space of functions on the group). We hope that it will give a new class of examples of dynamical systems.

\subsubsection{Strong and weak approximations in ergodic theory}

There are two theories of approximations of automorphisms in ergodic theory. Both theories are based on the fundamental Rokhlin lemma on approximation of automorphisms with periodic automorphisms. The first of them, weak approximation, was very popular in the 1960--70s and gave many concrete results; it used Rokhlin towers for which the corresponding periodic automorphisms converge in the sense of the weak topology on the group of automorphisms. For details, see \cite{Erg,Kat} and references therein.

At the same time (the 1970s), another kind of approximation, based on the uniform convergence of automorphisms, was developed by the author. In this case, additionally, the orbit partition of the approximation is finer than the orbit partition of the group action;  in other words,
 this is an approximation in the sense of the (nonseparable) metric on the group of measure-preserving automorphisms given by
$$ d(T,S)=\mbox{meas}\{x:Tx\ne Sx\}.
$$
If a periodic automorphism $S_n$ is close in the sense of this metric to a given automorphism $T$, then the orbit partition of $S$ is, up to a set of small measure, a subpartition of the orbit partition of $T$.
A monotonic approximation of $T$ by periodic automorphisms $S_n$
(or a coherent family of Rokhlin towers) defines a filtration: the sequence of the orbit partitions of $S_n$, whose tail partition is just the orbit partition of $T$. Of course,  in order to obtain invariant properties
of automorphisms, the convergence to zero of the distance between $S_n$ and $T$ must be complemented with additional conditions. The main source of such conditions is the theory of filtrations. Thus, our concept of strong approximation (or {\it globalization of Rokhlin towers}) leads to additional structures on the orbits of the automorphism, so-called ``hierarchies," which are merely the restrictions of the filtration to the orbits. One of the principal notions that came from the theory of filtrations is the notion of {\it standardness}, see Section~4.  This notion,  for the special homogeneous
case, appeared in the 1970s in my theory  called at that time the ``theory of decreasing sequences of measurable partitions" (this was the name of filtrations at that time, see \cite{V73a,V73,V94}). Now this theory is combined with a more general theory of filtrations on  paths of graded graphs and the theory of central measures.

 Thus, an adic approximation is nothing more than a globalization of Rokhlin towers; a graded graph appears naturally from a filtration, and vice versa: a filtration can be realized as the tail filtration of the path space of a graph. In contrast to realizations of automorphisms in symbolic dynamic (as shifts in the space $A^{\Bbb Z}$), an adic realization of an automorphism is very similar to a periodic or to a local transformation.

The same ``adic'' realization of a group action can be constructed for an arbitrary amenable group\footnote{And maybe also for nonamenable groups.}. For this, we must find a decreasing sequence of measurable partitions with finite blocks whose intersection  is  the orbit partition.  This can be done due to a theorem from \cite{CFW}. We return to this question in the forthcoming article.

\newpage

\section{A geometric approach to the asymptotics of the space of paths and measures. Standardness and limit shape theorems}

\subsection{Projective limits of simplices of measures}

We consider the problem of describing the invariant (central) measures from a geometric point of view.

 \medskip

   Consider a Markov compactum  $\cal X$  which is the space $T(\Gamma)$ of  paths on a graded graph $\Gamma$; the set  $\mbox{Meas}(\cal X)$ of all Borel probability measures on $\cal X$ is an affine compact (in the weak topology) simplex
   (Chouqet simplex), whose extreme points are $\delta$-measures.
Since $\cal X$ is an inverse (projective) limit of finite spaces (namely, the spaces of finite paths), it obviously follows that
$\mbox{Meas}(\cal X)$ is also an inverse limit of finite-dimensional simplices
 $\hat \Sigma_n$, where $\hat \Sigma_n$ is the set of formal convex combinations of finite paths (or just the set of probability measures on these paths) leading from the initial vertex to vertices of level $n$, $n=1,2, \dots $, and the projections $\hat\pi_n: \hat\Sigma_n \to \hat \Sigma_{n-1}$ correspond to ``forgetting'' the last vertex of a path. Every measure is determined by its finite-dimensional projections to cylinder sets (i.e., is a so-called cylinder measure). We will be interested only in invariant (central) measures, which form a subset of $\mbox{Meas}(\cal X)$.
Recall the definition which was given earlier for the special case of the path space  $T(\Gamma)$.
We repeat the definition of a central measure in slightly different terms.
\begin{definition}
A Borel probability measure $\mu$ on a Markov compactum  is called central if for any vertex of an arbitrary level, the projection of this measure to the subalgebra of cylinder sets of finite paths ending at this vertex is the uniform measure on this (finite) set of paths.
\end{definition}

Other, equivalent, definitions of a central measure
 $\mu \in  \mbox{Meas}(\cal X)$ are as follows.

\medskip
{\bf1.} The conditional measure of $\mu$ obtained by fixing the ``tail'' of infinite paths passing through a given vertex, i.e., the conditional measure of $\mu$ on the elements of the partition  $\xi_n$, is the uniform measure on the initial segments of paths for any vertex.

\medskip
{\bf2.} The measure is invariant under any adic shift (for any choice of orderings on the edges).

\medskip
{\bf3.} The measure is invariant with respect to the tail equivalence relation.
\medskip

The term ``central measure" stems from the fact that in the application to the representation theory of algebras and groups, measures with these properties determine traces on algebras (respectively, characters on groups). In the theory of stationary (homogeneous) topological Markov chains, central measures are called measures of maximal entropy.

The set of central measures on a Markov compactum $\cal X$ (on the path space $T(\Gamma)$ of  a graph $\Gamma$) will be denoted by $\Sigma({\cal X})$ or $\Inv (\Gamma)=\Sigma(\Gamma)$.
Clearly, the central measures form a convex weakly closed subset of the simplex of all measures:
$$\Sigma({\cal X}) \subset \mbox{Meas}(\cal X).$$ The set  $\Sigma({\cal X})$ of central measures is also a simplex, which can be naturally presented as a projective limit of the sequence of the finite-dimensional simplices of convex combinations of uniform measures on the $n$-cofinality classes. In more detail, the following proposition holds.

\begin{statement}
The simplex of central measures can be written in the form
$$\Sigma({\cal X})=\lim_{\leftarrow}(\Sigma_n; p_{n,m}),$$
or
$$ \Sigma_1 \leftarrow \Sigma_2 \leftarrow \dots\leftarrow \Sigma_n\leftarrow \Sigma_{n+1}\leftarrow
\dots  \leftarrow\Sigma_{\infty}\equiv \Sigma({\cal X}),$$
where $\Sigma_n$ is the simplex of formal convex combinations of vertices of the $n$th level
$\Gamma_n$ (i.e., points of $X_n$), and the projection
$p_{n,n-1}: \Sigma_n \to \Sigma_{n-1}$ sends a vertex $\gamma_n \in \Gamma_n$ to the convex combination
$\sum \lambda_{\gamma_n}^{\gamma_{n-1}} \delta_{    \gamma_{n-1}}\in \Sigma_{n-1}$ where the numbers $\lambda_{\gamma_n}^{\gamma_{n-1}}$ are uniquely determined by the condition that $\lambda_{\gamma_n}^{\gamma_{n-1}}$ is proportional to the number of paths leading from
$\emptyset$ to $\gamma_{n-1}$ (which is denoted, as already mentioned, by
$\dim\gamma_{n-1}$).\footnote{In the general (noncentral) case, the coefficients  $\lambda$ are the cotransition probabilities (see above).}  The general form of the projection is $p_{n,m}=\prod_{i=m}^{n+1} p_{i,i-1}$, $m>n$.
\end{statement}

\begin{proof}
The set of all Borel probability measures on the path space is a simplex which is a projective limit of the simplices generated by the spaces of finite paths of length $n$ in the graph, which follows from the fact that the path space itself is a projective limit with the obvious projections of ``forgetting'' the last edge of a path. The space of invariant measures is a weakly closed subset of this simplex, and we will show that it is also a projective limit of simplices (the fact that it is a simplex is well  known). The projection~$\mu_n$ of any invariant measure  $\mu$ to a finite cylinder of level $n$ is a measure invariant under changes of initial segments of paths and hence lies in the simplex defined above; since the projections preserve this invariance, $\{\mu_n\}$ is a point of the projective limit. It remains to observe that a measure is uniquely determined by its projections, which establishes a bijection between the points of the projective limit and the set $\Sigma(\Gamma)$ of invariant measures.
\end{proof}

\begin{figure}[h!]
\hspace*{-1cm}{\includegraphics[scale=0.3]{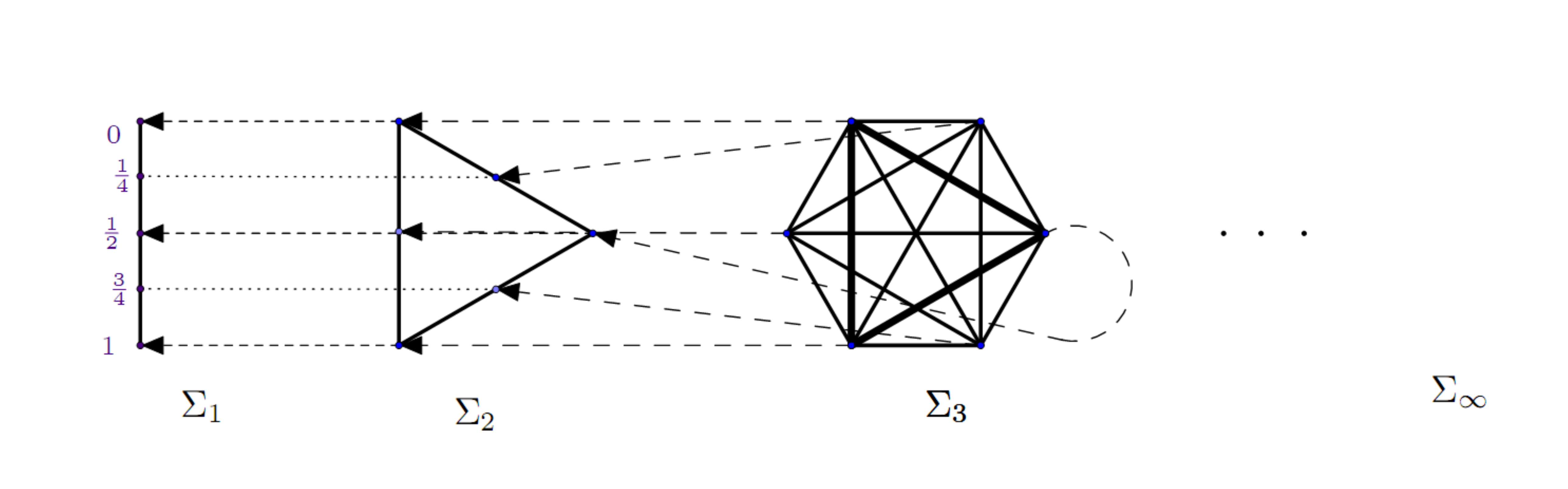}}
\caption{A projective limit of simplices.}
\label{fig:simplices }
\end{figure}


The fact that the set of  measures invariant with respect to a countable group acting on a compactum, as a subset of the simplex of all measures on the compactum, is an affine simplex
 (Choquet simplex)  can easily be  deduced from the ergodic decomposition of invariant measures; this is well known (see \cite{Fel}). It is less known that the same is true for the set of probability  measures that agree with a cocycle (see Section~1), or that have given cotransition probabilities or given Radon--Nikodym derivatives (for the action of adic transformations). Using the ergodic decomposition for quasi-invariant measures and the above interpretation of a general projective limit of simplices as a set of $\Lambda$-measures, we obtain a natural proof of the following statement.

\begin{statement}
A projective limit of finite-dimensional simplices is a Choquet simplex.
\end{statement}

This is a nontrivial fact even if the projective limit is finite-dimensional, see \cite{Wink} and the proof given there.  Our proof, which is based on the uniqueness of an ergodic decomposition of measures with a given cocycle, seems more natural.

Recall that points of the simplex $\Sigma_n$ are probability measures on the points of $X_{n}$ (i.e., on the vertices of the $n$th level $\Gamma_{n}$), and the extreme points of $\Sigma_n$ are exactly these vertices. Note that distinct vertices of the graph correspond to distinct vertices of the simplex.

Extreme points of the simplex $\Sigma(\Gamma)$  of invariant measures on the whole path space $T(\Gamma)$
are indecomposable invariant measures, i.e., measures that cannot be written as nontrivial convex combinations of other invariant measures. Then it follows from the theorem on the decomposition of measures invariant with respect to a hyperfinite equivalence relation into ergodic components that an indecomposable measure is ergodic (= there are no invariant subsets of intermediate measure). It is these measures that are of most interest to us, since the other measures are their convex combinations, possibly continual. The set of ergodic central measures of a Markov compactum
 $\cal X$ (of a graph~$\Gamma$) will be denoted by $\mbox{Erg}(\cal X)$ or $\mbox{Erg}(\Gamma)$.

The problem which we discuss here is about the
description of the set of all central ergodic measures for a given Markov compactum. A meaningful question is the following: for which Markov compactum (or graded graph) the set of ergodic central measures has an analytic description in terms of combinatorial characteristics of this compactum (graph), and what are these characteristics? In what cases such a description does exist? The role of such characteristics can be played by some properties of the sequence of matrices
 $\{M_n\}$ determining the compact (graph), frequencies, spectra, etc.

This problem is similar to the problem of describing unitary factor representations of finite type of discrete locally finite groups, finite traces of some $C^*$-algebras, Dynkin's entrance and exit boundaries; it is very closely related to the problems of finding Martin boundaries, Poisson--Furstenberg boundaries, etc. Since the 1950s, it is  well known that the situation with classification of irreducible representations of groups and algebras can be either ``tame'' (there exists a Borel parametrization) or ``wild'' (such a parametrization does not exist).  By Thoma's theorem \cite{Thoma2}, the classification problem for the irreducible representations of  a countable group is tame only if the group is eventually Abelian (i.e., has a normal Abelian subgroup of finite index).
This also happens, though more rarely, with factor representations. But in many classical situations, the answer is ``tame,'' which is a priori far from obvious.

For example, the characters of the infinite symmetric group, i.e., the invariant measures on the path space of the Young graph (see the next section), have a nice parametrization, and this is a deep result, see Section~6; however, for the graph of unordered pairs (see Figure~\ref{fig:UnPGraph}), there is no nice parametrization,  because of the nonstandardness of the graph. We emphasize that the presentation of  $\Sigma(\Gamma)$ as a projective limit of simplices essentially relies on the approximation, i.e., on the structure of the Markov compactum (graph). Obviously, the answer to the stated question also depends on the approximation. The fact is that we can change the approximation without changing the stock of invariant measures, which is determined only by the tail equivalence relation. The dependence of our answers on the approximation will be discussed later (see the remark on the lacunary isomorphism theorem in the last section). But since in actual problems the approximation is explicit already in the setting of the problem, the answer should also be stated in its terms. See examples in the next section.

\subsection{Geometric formulations}

We will recall some well-known geometric formulations, since the language of convex geometry is convenient and illustrative in this context.

\medskip
{\rm1.} The set of all Borel probability measures on a separable compact set invariant under the action of a countable group (or equivalence relation) is a simplex (= Choquet simplex), i.e.,\ a separable affine compact set in the weak topology whose any point has a unique decomposition into an integral with respect to a measure on the set of extreme points.\footnote{Choquet's theorem on the decomposition of points of a convex compact set into an integral with respect to a probability measure on the set of extreme points is a strengthening, not very difficult, of the previous fundamental Krein--Milman theorem saying that a convex affine compact set is the weak closure of the set of convex combinations of extreme points.}
The set of ergodic measures is the Choquet boundary, i.e., the set of extreme points, of this simplex; it is always a
$G_{\delta}$ set.
\medskip

{\rm2.} Terminology (somewhat less than perfect): a Choquet simplex is called a Poulsen simplex if its Choquet boundary is weakly dense in it, and it is called a Bauer simplex if the boundary is closed (see \cite{Fel}).  Cases intermediate between these two ones are possible.

\medskip
{\rm3.} A projective limit of simplices (see below) is a Poulsen simplex if and only if for any $n$ the union of the projections of the vertex sets of the simplices with greater numbers to the $n$th simplex is dense. The universality of a Poulsen simplex was later observed and proved by several authors.

\begin{statement}
All separable Poulsen simplices are topologically isomorphic as affine compacta; this unique, up to isomorphism, simplex is universal in the sense of model theory.\footnote{That is, for every separable simplex there exists an injective affine map of this simplex into the Poulsen simplex, and an isomorphism of any two isomorphic faces of the Poulsen simplex can be extended to an automorphism of the whole simplex.}
\end{statement}

One can easily check that every projective limit of simplices arises when studying {\it quasi-invariant measures} on the path space of a graph, or Markov measures with given cotransition probabilities (see above). But in what follows we consider only central measures, i.e., take a quite special system of projections in the definition of a projective limit. However, there is no significant difference in the method of investigating the general case compared with the case of central measures. We will return to this question elsewhere.

We formulate two simple facts, which follow from definitions.

\medskip
{\rm4.} Every ergodic central measure on a Markov compactum (on the path space of a graph) is a Markov measure with respect to the structure of the Markov compactum (the ergodicity condition is indispensable here).

\medskip
{\rm5.} The tail filtration is semi-homogeneous with respect to every ergodic central measure, which means exactly that almost all conditional measures for every partition $\xi_n$, $n=1,2, \dots$, are uniform.

The metric theory of semi-homogeneous filtrations will be treated in a separate paper.

\subsection{The extremality of points of a projective limit, and the ergodicity of Markov measures}

We give a criterion for the ergodicity of a measure in terms of general projective limits of simplices, in other words, a criterion for the extremality of a point of a projective limit of simplices.

Assume that we are given an arbitrary projective limit of simplices
$$ \Sigma_1 \leftarrow \Sigma_2 \leftarrow \dots \leftarrow\Sigma_n\leftarrow \Sigma_{n+1}\leftarrow \dots\leftarrow \Sigma_{\infty}$$
with affine projections $p_{n,n-1}:\Sigma_n\rightarrow \Sigma_{n-1}$, $n=1,2, \dots$ (the general projection
$p_{m,n}:\Sigma_m\rightarrow \Sigma_n$ is given above).

Consider an element $x_{\infty} \in \Sigma_{\infty}$ of the projective limit; it determines, and is determined by, the sequence of its projections $\{x_n\}_{n= 1,2, \dots}$, $x_n \in \Sigma_n$,  to the finite-dimensional simplices. Fix positive integers
$n<m$ and take the (unique) decomposition of the element $x_m$, regarded as a point of the simplex $\Sigma_m $, into a convex combination of its extreme vertices $e_i^m$:
$$x_m=\sum_i c_m^i \cdot e_i^m, \quad \sum_i c_m^i=1, \quad c_m^i\geq 0;$$
denote by $\mu_m=\{c_m^i\}_i$ the measure on the vertices of $\Sigma_m$ corresponding to this decomposition. Project this measure $\mu_m$ to the simplex $\Sigma_n$, $n<m$, and denote the obtained projection
by $\mu_m^n$; this is a {\it measure on $\Sigma_n$, and thus a random point of $\Sigma_n$}; note that this measure is not, in general, concentrated on the vertices of the simplex $\Sigma_n$.

\begin{statement}[Extremality of a point of a projective limit of simplices]\label{prop5}
A point $x_{\infty}=\{x_n\}_n$ of the limiting simplex $\Sigma_{\infty}$ is extreme if and only if the sequence of measures $\mu_n^m$ weakly converges, as $m \to \infty$, to the $\delta$-measure $\delta_{x_n}$ for all values of $n$:
\begin{eqnarray*}\mbox{for every } \epsilon>0, \mbox{ for every } n \mbox{ there exists }K=K_{\epsilon,n}\mbox{ such that }\\  \mu_n^m(V_{\epsilon}(\mu_n))>1-\epsilon\quad \mbox{for every }m>K,
\end{eqnarray*}
 where $V_\epsilon(\cdot)$ is the $\epsilon$-neighborhood of a point in the usual (for instance, Euclidean) topology.
\end{statement}

It suffices to use the continuity of the decomposition of an arbitrary point~$x_{\infty}$ into extreme points in the projective limit topology, and project this decomposition to the finite-dimensional simplices; then for extreme points, and only for them, the sequence of projections must converge to a $\delta$-measure. For details, see \cite{2014}.

One can easily rephrase this criterion for our case $\Sigma_{\infty}=\Sigma(\Gamma)=\Sigma(\cal X)$.
Now it is convenient to regard the coordinates (projections) of a central measure
$\mu_{\infty}$ not as points of finite-dimensional simplices, but as measures  $\{\mu_n\}_n$ on their vertices
 (which is, of course, the same thing). Then the measures $\mu_m^n$
 should be regarded as measures on probability vectors indexed by the vertices of the simplex, and the measure
$\mu$ on the Markov compactum~$\cal X$ (or on $T(\Gamma)$), as a point of the limiting simplex $\Sigma$. The criterion then says that $\mu$ is an ergodic measure (i.e., an extreme point of $\Sigma$) if and only if the sequence of measures
$\mu_m^n$ (on the set of probability measures on the vertices of the simplex $\Sigma_n$) weakly converges as $m \to \infty$
to the measure $\mu_n$ (regarded as a measure on the vertices of $\Sigma_n$) for all $n$.

In probabilistic terms, our assertion is a topological version of the theorem on convergence of martingales in measure and has a very simple form: for every $n$, the conditional distribution of the coordinate $x_{n}$ given that the coordinate
 $x_m$, $m>n$, is fixed converges in probability to the unconditional distribution of  $x_n$ as $m\to \infty$.

According to Proposition~\ref{prop5}, in order to find the finite-dimensional projections of ergodic measures, one should enumerate all $\delta$-measures that are weak limits of measures $\mu_n^m$ as $m\to \infty$. But, of course, this method is inefficient and tautological. The more efficient {\it ergodic method} requires, in order to be justified, a strengthening of this proposition, namely, one should replace convergence in measure with convergence almost everywhere.

\subsection{All boundaries in geometric terms}

The following definition is a paraphrase of the definition of the Martin boundary in terms of projective limits.

\begin{definition} A point   $\{x_n\}\in \Sigma_{\infty}$ of a projective limit of simplices belongs to the Martin boundary if there is a sequence of vertices $\alpha_n \in \mbox{\rm ex}(\Sigma_n)$, $n =1,2, \dots$, such that for every $m$ and an arbitrary neighborhood $V_{\epsilon}(x_m)\subset \Sigma_m$ there exists $N$ such that  $$p_{n,m}(\alpha_n) \in V_{\epsilon}(x_m)$$ for all $n>N$.
\end{definition}

 Less formally, a point of the limiting simplex belongs to the Martin boundary if there exists a sequence of vertices that weakly converges to this point (``from the outside'').
The condition of belonging to the Martin boundary is a weakening of the almost extremality criterion, hence the following assertion is obvious.

\begin{statement}
The Martin boundary contains the closure of the Choquet boundary.
\end{statement}

However, there are examples where the Martin boundary contains the closure of the Choquet boundary as a proper subset. A question arises: can one describe the Martin boundary in terms of the limiting simplex itself? The negative answer was obtained in \cite{VM2015}.

\subsection{A probabilistic interpretation of properties of projective limits}

Parallelism between considering pairs \{a graded graph, a system of cotransition probabilities\} on the one hand and considering projective limits of simplices on the other hand means that the latter subject has a probabilistic interpretation. It is useful to describe it without appealing to the language of pairs. Recall that in the context of projective limits, a path is a sequence
$\{ t_n\}_n$ of vertices $t_n \in \mbox{ex}\Sigma_n$ that agrees with the projections $p_{n,n-1}$ for all $n\in \mathbb N$ in the following sense: $p_{n,n-1}t_n$ has a nonzero barycenter coordinate with respect to $t_{n-1}$. First of all, every point $x_{\infty} \in \Sigma_{\Lambda}$ of the limiting simplex is a sequence  $\{x_n\}$ of points of the simplices $\Sigma_n$ that agrees with the projections: $\pi_{n,n-1}x_n=x_{n-1}$, $n \in \mathbb N$. As an element of the simplex, $x_n$ determines a measure on its vertices, and, since all these measures agree with the projections, $x_{\infty}$ determines a measure $\mu_x$
 on the path space with fixed cotransition probabilities. Conversely, every such measure comes from a point $x_\infty$. Thus the limiting simplex is the simplex of all measures on the path space with given cotransition probabilities. The extremality of a point $\mu\in \mbox{ex}(\Sigma_{\Lambda})$ means the ergodicity of the measure $\mu$, i.e., the triviality with respect to $\mu$ of the tail $\sigma$-algebra on the path space. The above extremality criterion  has a simple geometric interpretation, on which we do not dwell.

So, we have considered the following boundaries of a projective limit of simplices (or an equipped graph):

  {\it  the Poisson--Furstenberg boundary}  $\subset$ {\it the Dynkin boundary} = {\it the Choquet boundary} $\subset$ {\it the closure of the Choquet boundary} $\subset$
    {\it the Martin boundary}  $\subset$ {\it the limiting simplex}.

The first boundary is understood as a measure space; all inclusions are, in general, strict; the answer to the question of whether the Martin boundary is a geometric object is negative, see \cite{VM2015}.

We summarize this section with the following conclusion:
{\it the theory of equipped graded graphs  {\rm(}i.e., pairs $\{$a graded graph $+$ a system of cotransition probabilities$\}${\rm)}
is identical to the theory of Choquet simplices regarded as projective limits of finite-dimensional simplices.}

\subsection{The definition of the intrinsic metric
on the path space  for central measures}

We proceed to our main goal, which is to construct an approximation of a projective limit of simplices, i.e., a simplex of measures with a given cocycle, and to define the ``intrinsic metric (topology)'' on this limit. This metric was defined in  recent papers by the author \cite{2014,VEq,In,Sm,2015} on path spaces of graphs, only for central measures and under some additional conditions on the graph (the absence of vertices with the same predecessors). In this section, we give this definition in the same generality, for an arbitrary graded graph and the trivial cocycle ($c=1$), or for central measures (see Section~2); most importantly, we consider the whole limiting simplex and not only its Choquet boundary. This allows us to study the boundary for graphs with nonstandard (noncompact) intrinsic metrics. We formulate definitions and results both in terms of equipped graded graphs and in terms of projective limits of simplices spanned by the vertices of different levels.

In the next section (Section~4) devoted to the theory of filtrations, we will give a general definition of the intrinsic metric (topology) and the definition of a general standard filtration. Note that the definition given in the current section also makes sense in the general case (for noncentral measures), and in the case of central measures these two definitions are equivalent, but for noncentral measures they are,  in general,  not equivalent. The difference between two definitions consists in
the different manner of iterating the main operation, which is defined below.

We start with the definition of an important topological operation which will be repeatedly used, that of ``transferring a metric."

Let $(X, \rho_X)$ be a metric space and $\phi: X \to Y$ be a (Borel-)measurable map from $X$ to a Borel space $Y$;
assume that the preimages of points $\phi^{-1}(y)$, $y \in \phi(X)\subset Y$, are endowed with Borel probability measures
 $\nu_y$ that depend on $y$ in a Borel-measurable way;  $\phi$ will be called an {\it equipped map}.

\begin{definition}
The result of transferring the metric $\rho_X$ on the space $X$ to the Borel space $Y$ along the equipped map $$\phi:X\rightarrow Y$$ is the metric $\rho_Y$ on  $Y$ defined by the formula $$\rho_Y (y_1,y_2)=k_{\rho_X}(\nu_{y_1},\nu_{y_2}),$$
where $k_{\rho}$ is the classical Kantorovich metric on Borel probability measures on  $(X,\rho_X)$.
\end{definition}

1. Consider an equipped graph $(\Gamma, \Lambda)$ and the corresponding projective limit of simplices
 $\Sigma_{\Lambda}(\Gamma)$. Define an arbitrary metric $\rho=\rho_1$  on the path space $T(\Gamma)$ that agrees with the Cantor topology on
$T(\Gamma)$; denote by $k_{\rho_1}$ the Kantorovich metric on the space $\Delta(\Gamma)$ of all Borel probability measures on  $T(\Gamma)$ constructed from the metric  $\rho_1$.
  See the original definition of the Kantorovich metric (1942) in \cite{Kant}; see also \cite{V100} and the definition below).

2. Given an arbitrary path $v \equiv\{v_n\}$, consider the finite set of paths $v(u)=\{u,v_2, \dots\}$
whose coordinates coincide with the corresponding coordinates of $v$ starting from the second one, and assign each of these paths the measure $\lambda_{v_2}^u$. Now define an equipped map
$\phi_1:T(\Gamma) \rightarrow \Delta(\Gamma)=\Delta_1$, which sends the path $v$ to the measure $\sum\limits_ {u:\,u\prec v_2}\lambda_{v_2}^u \delta_{v(u)}$. It is more convenient to regard it as a map from the simplex
 $\Delta(\Gamma)$ to itself, by identifying a path with the $\delta$-measure at it.

Observe two important properties of the operation that associates with a metric space the simplex of probability measures on this space equipped with the Kantorovich metric:

1) monotonicity proved in \cite{2014}: the inequality $\rho \leq k\rho'$ implies $K_{\rho}\leq rK_{\rho'}$;

2) linearity in the metric: $K_{a\rho_1+ b\rho_2}=aK_{\rho_1}+bK_{\rho_2}$, $a,b > 0$.

Transferring the metric $\rho_1$ along the equipped map $\phi_1$, we obtain a metric~$\rho_2$ on a subset $\Delta_2=\phi_1(\Delta_1)$ of the simplex $\Delta (\equiv \Delta_1(\Gamma))$.

3. In a similar way we define the map $\phi_2$ that sends every measure from $\Delta_2$ concentrated on paths of the form
$\{u_1,v_2, \dots\}$, ${u_1\!\prec\!v_2}$, to the measure on the finite collection of paths of the form $\{u_1,u_2,v_3, \dots\}$
whose coordinates coincide with $v_i$  starting from the third one and the second coordinate $u_2$ runs over all vertices $u_2 \prec v_3$ with probabilities $\lambda_{v_3}^{u_2}$. Again transferring the metric~$\rho_2$ from  the space $\Delta_2$ along the equipped map $\phi_2$, we obtain a metric~$\rho_3$ on the image $\Delta_3 \equiv \phi_2(\Delta_2)=\phi_2\phi_1(\Delta)$.

Note that the images of the maps $\phi_n$, i.e., the sets $\Delta_n$, are simplices, but their vertices are no longer $\delta$-measures on the path space, but measures with finite supports of the form $\sum\limits_ {u_1,u_2, \dots, u_k} \lambda_{u_2}^{u_1}\cdots \lambda_{v_{k+1}}^{u_k}\cdot \delta_{u_1,\dots, u_k,v_{k+1}, \dots}$. The definition of the simplices $\Delta_n$ does not depend on the metrics $\rho_n$.

 4. Continuing this process indefinitely, we obtain an infinite sequence of metrics on the decreasing sequence of simplices
 \begin{align*}
 \Delta_n&=\phi_{n-1}(\Delta_{n-1})=\phi_n\phi_{n-1}\cdots \phi_1(\Delta_1),
 \\
  \Delta&=\Delta_1 \supset \Delta_2 \supset \Delta_3 \supset \cdots, \quad  \bigcap_n \Delta_n =\Delta_{\infty}.
  \end{align*}
  Thus we have a sequence of equipped maps of the decreasing sequence of simplices
   $$\Delta_1 \to \Delta_2\to \cdots\to \Delta_n \to \dots \to\Delta_{\infty}.$$

  The following assertion does not involve the metric.

\begin{statement}
The intersection $\Delta_{\infty}$ of all simplices $\Delta_n$ consists exactly of those measures on the path space
 $T(\Gamma)$ (i.e., those points of the simplex  $\Delta(\Gamma)$ of all measures) that have given cotransition probabilities (given cocycle), and, therefore, this intersection coincides with the projective limit of the simplices:
$$\Delta_{\infty}=\Sigma_{\Lambda}(\Gamma).$$
\end{statement}

Now we define the main notion which grasps the drastic difference between two types of Markov compacta or graded graphs (Bratteli diagrams).

\begin{definition}
A Markov compactum ${\cal X}=T(\Gamma)$, or the path space  of a graded graph $\Gamma$, is called standard if there exists a limit $\lim_{n \to \infty}\rho_n = \rho_{\infty}$ of semimetrics on the space  $\Delta_{\infty}$ (= $\Sigma_{\Lambda}(\Gamma)$,
where $\Lambda$ is the system of cotransition probabilities).
More exactly, for every pair of paths  $x,y \in T(\Gamma)$
there exists a limit $$\lim_n \rho_n(x_n,y_n)\equiv \rho_{\infty}(x,y).$$
The existence of the limit does not depend on the choice of the initial metric.
In this case, the limiting simplex $\Sigma_{\Lambda}(\Gamma)$ is equipped with this intrinsic metric.
\end{definition}
Note that in this case $\rho_{\infty}$ generates the projective limit  topology on $\Sigma_{\Lambda}(\Gamma)$.

It is easy to check that the limiting intrinsic semimetric is the same for the whole class of cylinder semimetrics;
a more detailed analysis will be presented elsewhere.

\subsection{Standardness and limit shape theorems}

Now we can formulate a new alternative for the problem of studying  the asymptotics of the path space and measures on it. For simplicity, we state the problem only for central measures, but the case of an equipped graph and measures with given cotransition probabilities can be considered in the same way.

\begin{theorem}
Consider a graded graph $\Gamma$, and let $\Sigma_{\infty}$ be the simplex of all central measures.
The following two properties of the simplex are equivalent.

{\rm1.} The graph $\Gamma$ and the space $T(\Gamma)$ are standard
(i.e., there exists a limit of the semimetrics $\rho_n$).

{\rm2.} For every ergodic central measure $\mu$ on $T(\Gamma)$ the following is true: for $(\mu\times \mu)$-almost all pairs of paths, $\lim\rho_n(x,y)=0$ (``limit shape theorem'').

In this case, the Choquet boundary (= the set of extreme points) of the simplex $\Sigma_{\infty}$ is open in its closure.
\end{theorem}

\noindent{\bf Remark.} The latter property is not characteristic for standard graphs: there exist nonstandard graphs
with the same property.

\medskip
The important and new problem is how to calculate the intrinsic metric for graphs and how to distinguish standard and nonstandard graphs. Standard graphs are graphs for which the problem of  describing the central measures is feasible, because the list of these measures has a reasonable parametrization. We say that the problem of describing the central measures for a given graph $\Gamma$ is ``smooth'' (or tame) if the graph is standard. For a nonstandard graph like $NUP$ or $UP$,
it is impossible to give a natural parametrization of the set of all ergodic central measures. The precise meaning of this is that there is no way to distinguish between ergodic and nonergodic measures in terms of a finite approximation: the projection of the set of all ergodic measures is the whole finite simplex  (see Figure~7)
of the ``tower of measures," the projective limit of simplices with epimorphic projections.

We will consider a given central measure on $T(\Gamma)$ and the tail filtrations of graded graphs for this measure in the next section.

The notion of standardness of graded graphs (as well as of filtrations of measure spaces, or filtrations of pasts of Markov processes) is also a far-reaching generalization of the notion of independence. For example, the central measures on the Young graph are parameterized by the frequencies of rows and columns of Young diagrams, but the lengths of rows (and columns) are not independent in the literal sense. At the same time, the Young graph is standard, which means that there is an
asymptotic independence of statistics for parameters of random diagrams with respect to any central measure. See our recent paper \cite{2015}.

The term ``limit shape theorem'' is due to the interpretation of many our problems as problems  concerning the behavior of geometric configurations (like Young diagrams or some other type of diagrams). According to this interpretation, these are theorems on the concentration of the random (in the sense of a given measure) configuration near a deterministic one, see \cite{Ve96}. The metric used for measuring the distance can be different. We have used the intrinsic metric, which is universal by definition, but, for example, the well-known limit shape theorem for the Plancherel measure used the much stronger and natural, in that case, uniform metric for diagrams, see \cite{VeK81} and Figure~\ref{fig:PlanchereStatistics}. The intrinsic metric corresponds to the notion of weak standardness, so a limit shape theorem takes place for graded graphs and measures on the path spaces for which the tail filtration is weakly standard,  see the next section.

       \begin{figure}[h!]
                \centering
              \includegraphics[scale=0.4]{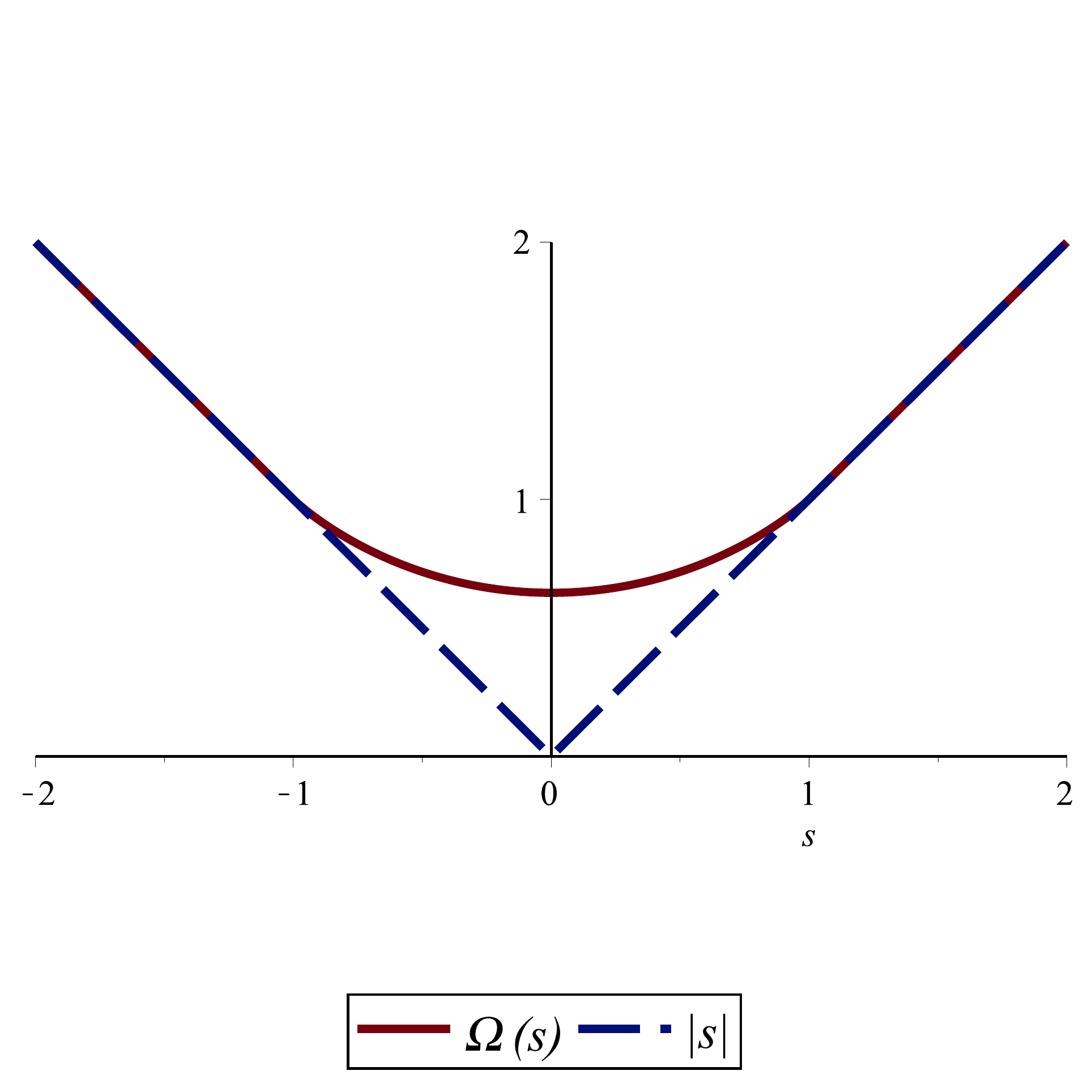}
             \[ \Omega(s) = \begin{cases}(2/\pi)(s\arcsin s+\sqrt{1-s^2})  &\text{for } \ |s|\leq 1,\\|s|  &\text{for} \ |s|\geq 1.\end{cases}\]
                \caption{The limit shape for the Plancherel measure. Vershik--Kerov (1977), Shepp--Logan (1977).}
                \label{fig:PlanchereStatistics}
        \end{figure}

  A completely different limit shape arises for the uniform statistics on  partitions (Young diagrams); it has the following form (for details, see \cite{Ve96}):

    \begin{figure}[h!]
                \centering
               \includegraphics[ scale=0.3]{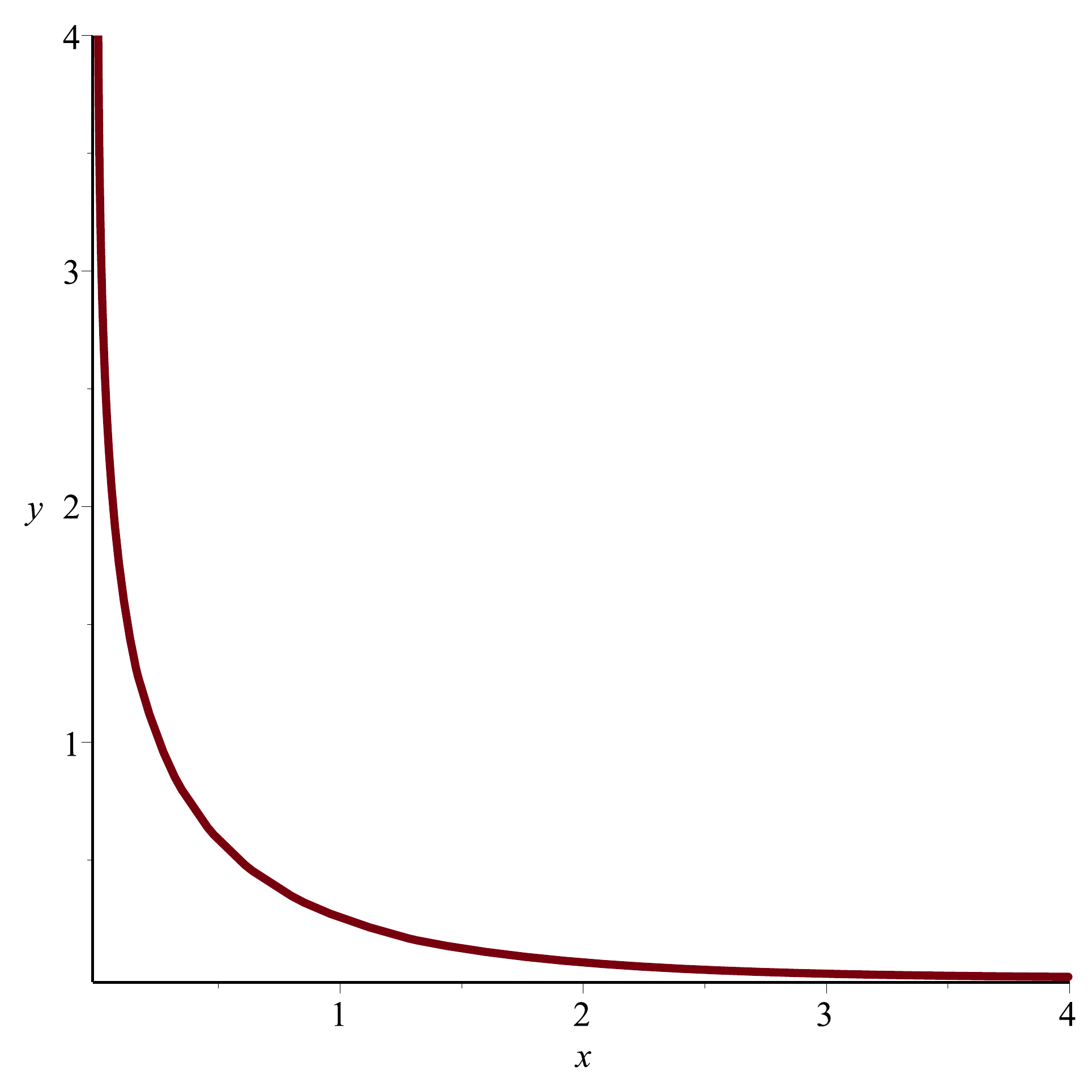}
                \caption{$\exp(-\frac{\pi}{\sqrt{6}}\, x)+\exp(-\frac{\pi}{\sqrt{6}}\, y) = 1$.}
                \label{fig:Caseb}
        \end{figure}%

\newpage

\section{The theory of filtrations, standardness, and nonstandardness}

One of the ideas of this article is to show the connection between the asymptotic theory of graded graphs and the theory of filtrations, or decreasing sequences of $\sigma$-algebras.
In this section, we briefly describe the main facts of the latter theory.
It started (see \cite{V68,V70}) with the definition of dyadic decreasing sequences of measurable partitions (= dyadic filtrations) and two fundamental facts: the lacunary theorem for ergodic dyadic filtrations and the standardness criterion for them (= isomorphism with Bernoulli filtrations). During the 1990s, several important papers and generalizations appeared; we will focus our attention only on the notion of standardness in the whole generality and explain the connection with the standardness of graded equipped graphs. In the previous section, we defined standardness for central measures; the notion was inspired by the measure-theoretic standardness for semihomogeneous filtrations (for example, dyadic filtrations)
introduced in  the 1970s  (see \cite{V70,V73,V94}).

\subsection{Filtrations in measure theory and in the theory of Borel spaces}

We bring together several definitions and preliminary nontrivial facts about filtrations.
A decreasing filtration (hereafter called just a filtration) is a decreasing sequence
$${\mathfrak A}_0 \supset {\frak A}_1 \supset {\frak A}_2\supset \cdots$$
of $\sigma$-algebras of a standard  measure space
 $(X,\mu,\mathfrak A)$ with a continuous measure~$\mu$. Here ${\frak A}_0=\frak A$ coincides with the $\sigma$-algebra of all measurable sets.

Filtrations arise

$\bullet$ in the theory of random processes (stationary or not), as the sequences of ``pasts'' (or ``futures'');

$\bullet$ in the theory of dynamical systems, as filtrations generated by orbits of periodic approximations of group actions;

$\bullet$ in statistical physics, as filtrations of families of configurations coinciding outside some volume;

$\bullet$ and, finally and most importantly,
in the theory of $C^*$-algebras and combinatorics, as  tail filtrations of the path spaces  of equipped $\mathbb N$-graded locally finite graphs.

The problem of  classification of filtrations in the category of measure spaces or other categories is deep and quite topical.

We will consider filtrations either in a standard separable Borel space, as

$\bullet$ the tail filtration in the path space $T(\Gamma)$ of an equipped graded graph~$\Gamma$;

or in the standard separable measure space (Lebesgue space), as

$\bullet$ the filtration of the ``pasts'' of a discrete time random process  $\{\xi_n\}$, $-n\in {\Bbb N}$, in the space of realizations of this process.

Recall that any $\sigma$-algebra in a standard measure space is determined by a measurable partition of the measure space, which is in turn determined by a system of conditional measures (canonical system of measures in the sense of Rokhlin):
the system of conditional measures determines the partition, and hence the $\sigma$-algebra, up to isomorphism.
Thus a filtration of $\sigma$-algebras gives rise to an infinite decreasing sequence of measurable partitions
 $\{\xi_n; n\leq 0\}$, which we will use in what follows.

The partition $\xi_0$ corresponding to the whole $\sigma$-algebra ${\frak A}_0$ is the partition into singletons.

In terms of functional analysis, a filtration can also be given by a decreasing sequence of subalgebras of
 $L^{\infty}$ type in the space $L^{\infty}_{\mu}(X)$.
  From this point of view, we discuss problems concerning decreasing sequences of subalgebras of a given algebra.

   A filtration determines a limiting equivalence relation on the measure space (i.e., in general, a nonmeasurable partition) and gives rise, in a canonical way, to a von Neumann algebra, but here we will not discuss these relations.

In what follows, we assume that almost all elements of all partitions  $\xi_n$, $n\leq 0$, are finite, and thus they are finite spaces equipped with (conditional) measures; hence the conditional measure on every element of the partition is determined by a finite-dimensional probability vector.

If all conditional measures on almost all elements of the partition  $\xi_n$ coincide and are uniform, with the number of points in the elements equal to $r_n$, then the filtration is called {\it homogeneous $r_n$-adic} (in particular, {\it dyadic} if  $r_n=2^n$); if the conditional measures are uniform, but the number of points in different elements can be different, then the filtration is called {\it semi-homogeneous}; this is the most interesting and important case. It corresponds to central measures on  path spaces of Bratteli diagrams. But the case of dyadic filtrations already contains all difficulties of the general theory. A specific case is the study of continuous filtrations, for which all conditional measures of all quotient partitions
 $\xi_n/\xi_{n-1}$, $n=1,2, \dots$, are continuous; here we do not consider this case, but the main methods described below apply to it, too.

Let us impose an additional finiteness condition:   {\it  a filtration is said to be of finite type if not only the elements of all partitions are finite, but also the collections of the probability vectors of conditional measures corresponding to the elements of the partition $\xi_n$ for every
 $n$ are finite.} In other words, the collection of all elements of the partition $\xi_n$  can be divided into finitely many subsets so that in each subset the vectors of conditional measures coincide. By the very definition, the class of such filtrations is invariant under measure-preserving transformations.

The following assertion holds.

\begin{statement}
An arbitrary  finite type filtration is isomorphic to a filtration corresponding to a Markov chain with finite state sets.
\end{statement}

The Markov filtration corresponding to a Markov process with finite state sets is, obviously, a  finite type filtration. Conversely, an arbitrary  finite type filtration can be realized as a filtration corresponding to a Markov chain with finite space sets, by choosing these sets by recursion on $n$, starting from the subsets mentioned in the definition of finiteness and subdividing these subsets if necessary to obtain a basis of the original space. Note that the filtration generated by a Markov chain with arbitrary state sets can be isomorphic to the filtration generated by a chain with finite state sets.

In what follows, we deal with  finite type filtrations generated by Markov chains.

It is not difficult to prove the following assertions.

\begin{statement}[Universality of tail filtrations]

{\rm1.} Every discrete filtration in a Lebesgue space is isomorphic to the tail filtration of an equipped graded graph with some system of cotransition probabilities.

{\rm2.} Every semihomogeneous filtration is isomorphic to the
tail filtration of a graded graph $\Gamma$ equipped with a central measure.
\end{statement}

As was already mentioned, we can restrict ourselves to one-sided Markov chains (in general, nonstationary).

\begin{definition}
{\rm1.} Two filtrations  ${\frak A}_n$, $n \in \Bbb N$, and  ${\frak A}'_n$, $n \in \Bbb N$, of a standard measure space $(X,\mu)$ are finitely isomorphic if for every $n$ there exists an automorphism $T_n$ of  $(X,\mu)$ such that
$$T_n({\frak A}_k)={\frak A}'_k, \qquad k=1,\dots, n.$$

{\rm2.} A filtration is called ergodic (or quasi-regular, Kolmogorov, satisfying the
zero--one law) if the intersection of the $\sigma$-algebras is trivial:
$$\bigcap_n {\frak A}_n
={\frak N},$$
where ${\frak N}$ is the trivial $\sigma$-algebra.
\end{definition}

It is clear that the ergodicity of a filtration is not an invariant of finite equivalence, because ergodic and nonergodic filtrations can be finitely isomorphic. For example, obviously,  any two dyadic filtrations are finitely isomorphic.
The crucial fact \cite{V70} is the existence of a continuum of ergodic dyadic filtrations that are not mutually isomorphic.

The fundamental problem is to classify (or to give efficient invariants up to automorphism of the space) the filtrations of  a given class of finite isomorphism.

Our goal is to introduce a simplest class of filtrations
such  that every ergodic filtration is finitely isomorphic to a filtration of this class. We called
 such filtrations {\it standard}. This implicitly determines the class of Markov processes or graded graphs for which the tail filtration is standard.

\subsection{Weak standardness}

We define a property of  filtrations that is a very natural step to the general notion of standardness. For  homogeneous filtrations, both notions coincide.

Consider an arbitrary Markov chain $\{x_n\}_{n \geq 0}$ with finite state sets and denote by $X$ the space of all its realizations, which is a general Markov compactum (in particular, it may be nonstationary). We will define a sequence of semimetrics
 $\rho_0=\rho,\rho_1,\rho_2, \dots $ on the space $X$. Denote by $\nu$
a Markov measure on $X$ and consider the filtration $\{{\frak A}\}_{n>0}$ of (``future'') $\sigma$-algebras. We will study metrics and semimetrics on the space $X$ that agree with the topology. A semimetric $\rho$  on $X$ is called a cylinder semimetric if there is $n$ such that $\rho(y,z)$ depends only on coordinates $y_k, z_k$ with
$k<n$. A metric $\rho$ is called an {\it almost cylinder metric} if it is a limit of cylinder semimetrics:
 $\rho(y,z)=\lim \rho_n(y,z)$. Clearly, the topology determined by an almost cylinder metric coincides with the topology of the compact space $X$.

We fix an almost cylinder metric or semimetric $\rho$ and use the Markov filtration in order to define a sequence of {\it transferred metrics}.

\begin{definition} Let $(X,\rho,\mu)$ be a (semi)metric measure space and $\xi$ be a measurable partition of $X$. Then on the quotient space
 $(X_{\xi},\mu_{\xi})$  there is a canonically defined semimetric $\rho_{\xi}$:
\begin{equation}\label{(Proj)}
\rho_{\xi}(x,y)\equiv K_{\rho}(\mu^x,\mu^y),
\end{equation}
where $x,y \in X$ and $\mu^x,\mu^y$ are the conditional measures of the partition $\xi$ on the elements containing the points $x,y$, respectively.
\end{definition}

Here $K_{\rho}(\cdot,\cdot)$ is the Kantorovich metric on the simplex of measures on the metric space
 $(X,\rho)$; recall the definition of this metric: the distance between two probability measures $\alpha_1,\alpha_2$ on a metric compact space $(Z,r)$ is defined as
\begin{equation}\label{K}
   K_r(\alpha_1,\alpha_2)=\inf_{\psi \in \Psi}\int_Z\int_Z
  r(a,b)\,d\psi(a,b),
\end{equation}
where  $\psi \in \Psi$ runs over the set of all measures on the space $Z\times Z$ with marginal (= coordinate) projections  $\alpha_1,\alpha_2$. One often says that measures $\psi$ are  ``couplings'' for the measures $\alpha_1,\alpha_2$.
Thus $\Psi$ is the set of all couplings.

Note that if $\rho$ is a metric rather than a semimetric, then
$K_{\rho}$ is a metric, too.

Let us return to our construction.  Successively apply the operation of transferring a semimetric to the spaces $(X/\xi_k\equiv
 X_{\xi_k},\rho_{k},\mu_{\xi_k})$ and partitions
 $\xi_{k+1}/\xi_k$, obtaining metrics $\rho_{k+1}$ on the spaces $X_{\xi_{k+1}}$,  $k=0,1,\dots$, which can be regarded as semimetrics on the original space $(X,\mu)$.

\begin{definition}
 A filtration is called weakly standard if it is ergodic and for some initial admissible\footnote{The notion of an {\it admissible metric on a measure space} was introduced by the author \cite{MNU} and studied in \cite{VPZ}. An admissible triple is a triple $(X,\mu,\rho)$ where $(X,\mu)$ is a standard  Lebesgue space and $\rho$ is a measurable separable (semi)metric on $X$. In this case, $\rho$ is called an admissible (semi)metric. For details, see \cite{VPZ}. In the general case of a filtration in a measure space, it is useful, and even necessary, to use the notion of an admissible metric.} metric
 $\rho=\rho_0$ on the space $(X, \mu)$, the sequence of semimetrics $\rho_n$, $n>0$, satisfies the condition
 \begin{equation}\label{(S)}
  \lim_{n\to\infty}\int_X\int_X \rho_n(x,y)\,d\mu(x)\mu(y) = 0;
\end{equation}
in other words, the sequence of semimetric measure spaces
 $(X,\rho_n)$ collapses to the single-point measure space as $n \to \infty$.
 \end{definition}

The main problem is to give rough metric invariants of filtrations up to measure-theoretic (or Borel)
equivalence. For instance, to distinguish filtrations from the class of filtrations  finitely isomorphic to ergodic homogeneous filtrations (``asymptotically homogeneous filtrations'').

\begin{definition}
{\rm 1.} A filtration in a Lebesgue space is called a \textbf{filtration of product type} if it is
metrically (i.e., in the sense of measure-preserving isomorphisms $\bmod 0$) isomorphic  to the filtration of  pasts of a Bernoulli sequence of random variables: ${\frak A}_n$ is
the $\sigma$-algebra generated by the variables $\omega_r$, $n>r$, where $\{\omega_n\}_{n \in \Bbb N}$ are independent
random variables, each taking finite values.

{\rm 2.} A homogeneous weakly standard filtration is a filtration isomorphic to a homogeneous filtration of product type. A dyadic filtration of product type has the form ${\frak A}_n$, $n=0,1, \dots$, where ${\frak A}_n$ is the $\sigma$-algebra of  measurable sets on $[0,1]$ with the Lebesgue measure depending on digits $\epsilon_r(x)$ with $r>n$ in the dyadic decomposition of $x$: $$x=\sum_{r=1}^{\infty} \frac{\epsilon_r(x)}{2^r}.$$
So, a standard dyadic filtration is a filtration  isomorphic to this example.
\end{definition}

\subsection{The definition of standardness}

We will define the notion of standardness. The main definition will be similar to the definition of weak standardness  from the previous section, but now we will use a more complicated procedure of iterating a metric. For a semihomogeneous filtration, the new definition is equivalent to the previous one.

\subsubsection{How to measure the distance between finite filtrations}

We proceed to describe our main technical tool, the iteration of a metric with respect to a given measure; it allows us to introduce a general notion of standard filtrations, which constitute a fundamental class of filtrations. In the previous section, we defined weakly standard filtrations for the tail filtrations of some graded graphs. We defined the intrinsic metric, which is the result of a procedure of iterating semimetrics.  But the definition of this section gives a more general iteration procedure, as well as a more general notion of standardness.

We start with some preparations for the main definitions.
Assume that we have two finite filtrations $\omega^1$ and $\omega^2$  in given spaces $X^1$ and $X^2$, respectively, $|X_i|=N_i$, $i=1,2$ (here $|A|$ is the number of points in a set $A$),  which are finite decreasing sequences of partitions (or finite filtrations) of $X^i$, $i=1,2$: $\omega^i=\{\xi_k^i\}_{k=1}^n$, $ \xi^i_0\succ \xi^i_1\succ \cdots \succ \xi^i_n$; here $\xi^i_0$ is the partition of $X^i$ into singletons and $\xi^i_n$ is the trivial partition of $X^i$. We will write
  $x\sim_k x'$ (respectively, $x\nsim_k x'$)  if $x$ and $x'$ belong to the same element (respectively, different elements) of
  the partition $\xi_k$, $k=1, \dots, n$.

  It is helpful to regard a finite filtration on a finite measure space as a finite tree with labels; or, more exactly, a {\it tree of partitions}, or a {\it tree with a set of cotransition probabilities}, see Figure~\ref{fig:new}.

  \begin{figure}[h!]
  \includegraphics[scale=1.2]{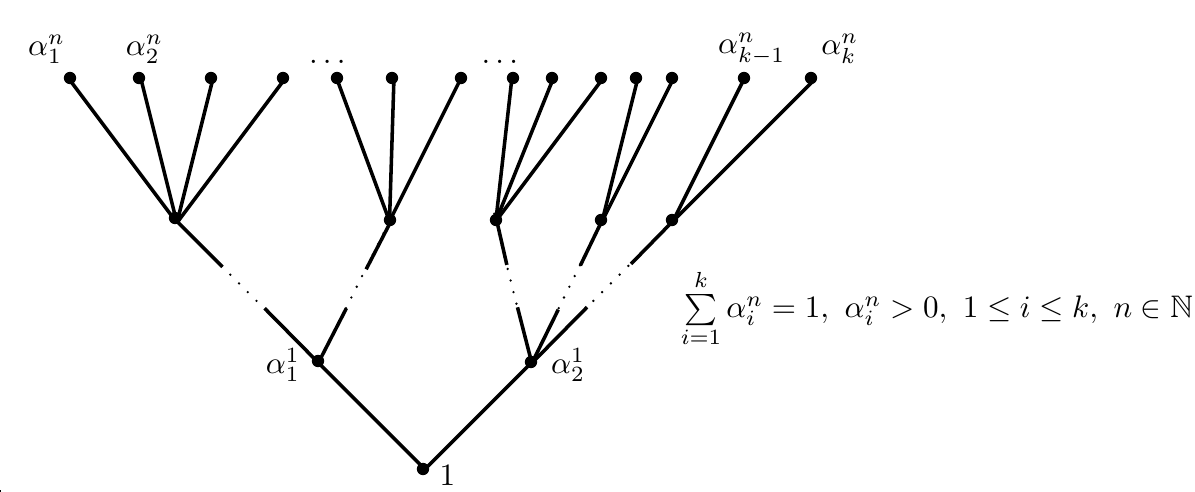}
\caption{Tree of partitions.}
\label{fig:new}
\end{figure}

Let $\mu^1$ and $\mu^2$ be probability measures  on $X^1$ and $X^2$, respectively, without points of zero measure.
  We regard an $N_1\times N_2$  matrix $$\psi=\{\psi_{x,y}: x\in X_1,\; y\in X_2\}$$ with nonnegative entries  as a correspondence, or a coupling (or a multi-valued map, or a polymorphism) between $X^1$ and $X^2$. {\it We consider  bistochastic coupling matrices $\psi$ with respect to the measures $\mu^1, \mu^2$: $\sum_x \psi(x,y)=\mu^2(y)$, $\sum_y \psi(x,y)=\mu^1(x)$.}

  We say that a coupling $\psi$ maps the filtration $\omega^1$ to the filtration $\omega^2$ if the following holds:
 if $\psi_{x,y}>0$, then for any $x'\sim_k x$ there exists $y'\sim_k y$ such that $\psi_{x', y'}>0$,  and
 for any $y''\sim_k y$ there exists $x''\sim_k y$ such that $\psi_{x'', y''}>0$. The simplest case where this condition is fulfilled is where $\psi$ is a bijection between $X^1$ and $X^2$ that is an isomorphism between the filtrations $\omega^1$ and~$\omega^2$. Denote the set of all such couplings by $\Psi(\omega^1,\omega^2)$.

It is easy to prove that $\Psi(\omega^1,\omega^2)$ is never empty, because the direct product $\mu^1\otimes \mu^2$  of the measures
provides a required coupling: $\psi(x,y)=\mu^1(x)\mu^2(y)$.

  Finally, assume that we have a matrix of distances  $\rho(x,y)$ for $x\in X^1$, $y\in X^2$, with $\rho(x,y)\geq 0$.
  An important example: let $f_1$ and $f_2$ be functions with values $0$ and $1$ defined  on $X^1$ and $X^2$, respectively, and let
  $\rho(x,y)=|f_1(x)-f_2(y)|$ be the distance between points of $X^1$ and $X^2$.

  Now we are ready to define the distance between two finite filtrations $\omega^1$ and $\omega^2$ on finite measure spaces  $(X^1,\mu^1)$ and $(X,\mu^2)$ using a distance $\rho$ between points of these spaces. It is defined by the following formula:
$$ r^{\rho}((X^1,\omega^1,\mu^1),(X^2,\omega^2,\mu^2))
=\inf_{\psi\in \Psi(\omega^1,\omega^2)} \sum_{x\in X^1,\,y\in X^2}  \rho(x,y)\psi(x,y).$$
It looks similar to the formula for the Kantorovich metric, but there is a crucial difference, namely, we use only those couplings that agree with  the filtrations.

 Using this definition, we can define a procedure of iterating a metric and the notion of standardness.

\subsubsection{Standardness}

Let $(X, \mu,\rho)$ be an admissible triple where $\rho$ is a metric (or a semimetric), and let $\{\xi_n\}_{n \in \Bbb N}$
be a discrete filtration on the space $X$.\footnote{The previous section allows one to give a complete metric invariant of finite filtrations in a standard measure space: it is a metric invariant
of the measurable function on the space $X/\xi_n$ with values in the space of $n$-trees of partitions that associates to an element~$C$ of $\xi_n$ the $n$-tree of partitions corresponding to the restriction of the filtration to this element.
This is a natural generalization of Rokhlin's invariants of a single measurable partition: $m_1(C)\geq m_2(C) \geq \dots \ge0$,
$\sum m_k(C)\le1$, where $m_k(C)$ are the atomic parts of the conditional measures.}
Let $C$ be an element of the partition $\xi_n$ and $\omega^C$ be the restriction to $C$  of the finite filtration $\xi_1,\xi_2, \dots, \xi_n$; let $\mu^C$ be the conditional measure on $C$ as an element of  $\xi_n$.

\begin{definition}
Define a sequence of semimetrics as follows:
$\bar\rho_0=\rho$, and $$\bar\rho_{n+1}(x,y)=\int\int_{X_{\xi_n}\times X_{\xi_n}}r^{\bar\rho_n}((C, \omega^C,\mu^C), (C',\omega^{C'},\mu^{C'}))\,d\mu_{\xi_n}(C)d\mu_{\xi_n}(C').$$
A filtration $\{{\frak A}_n\}_{n\in \Bbb N}$ is called standard
if
\begin{eqnarray}
\lim_{n\to \infty} \bar\rho_n=0
\label{standcrit}
\end{eqnarray}
for any initial metric $\rho$.
\end{definition}

The condition~\eqref{standcrit} is called the standardness criterion.

\begin{remark}{\rm
1. The choice of a metric in the definition is irrelevant:
if the condition is satisfied for some metric, then it is satisfied for all metrics (see below).

2. Standarness is not equivalent to weak standardness. See an example in the next section. It is obvious from the definition of the metrics that $\rho_n \leq \bar\rho_n$ and, in general, the inequality is strict. We may call  $\rho_n$ the weak intrinsic metric. The nature of the difference between weak standardness and standardness lies in the properties of the cocycle of the filtration (see Section~1); if the cocycle is equal to one (a homogeneous filtration, or a central measure), then both notions of standardness coincide. This difference between  $\rho_n$ and $\bar\rho_n$ is similar to the effect of the Bellman principle in dynamic programming: the optimal global strategy (in our case, the global coupling) can be or not be the result of a sequence  of optimal couplings.}
\end{remark}

\subsection{The properties of standard and weakly standard filtrations}

Many properties of filtrations are common for both notions of
standardness; for this reason, we use the word ``standard"
 when an assertion is true for both notions.

  The monotonicity property of the operation of transferring a metric and the continuity of this operation with respect to the pointwise convergence of metrics immediately imply that if the standardness condition (for either of the notions) is satisfied for a given initial metric, then it is satisfied for any initial metric. Indeed, it follows from the monotonicity and linearity  that the condition holds for any cylinder semimetric, and then one should use the fact that any almost cylinder metric is a limit of cylinder semimetrics.\footnote{The same argument allows one to infer that a filtration is standard with respect to an arbitrary admissible semimetric if it is standard with respect to a single admissible metric.}

In other words, the following theorem holds.

\begin{theorem}
The property of being a (weakly) standard Markov filtration is invariant under the group of all measure-preserving transformations and  does not depend on the choice of the initial metric.
\end{theorem}

It means that if a filtration is realized as the tail filtration of various Markov chains, then all these chains are standard or not simultaneously.

It should be noted that computing the iterations of functions or metrics and their expectations is not an easy task. In order to prove the existence of a nonstandard filtration, we need to check that the limit of the expectation of the metric $\rho_n$ does not vanish for some functions. The first example of a nonstandard filtration was suggested by the author in  \cite{V70} (see also \cite{V73, V94}): it is the filtration of pasts for a random walk over trajectories of a Bernoulli action of a free non-Abelean group. By now there are many such examples. A survey of the  state of the art in this field will be published in a separate paper.

The drawback of the above definition of standardness is that checking condition~\eqref{standcrit} requires computing the iterated metrics. It is desirable to have a criterion that would relate the metric $\rho_n$ directly to the initial metric, skipping the intermediate steps. We present such a criterion; its statement does not involve a Markov realization of a filtration, but whenever necessary we will use special couplings, in contrast to arbitrary couplings used in the definition of the Kantorovich metric. Another difference is that the condition must be satisfied not for one metric, but for all degenerate semimetrics of a special form.

The following observation is convenient for applications.

\begin{statement}
If the standardness (or weak standardness) condition holds for every semimetric of the form
$\rho_f(x,y)=|f(x)-f(y)|$ with functions $f \in L^{\infty}(X,\mu)$ that take finitely many values,
 then it holds for every metric.
\end{statement}

The statement of the standardness criterion implies the following important property.

\begin{remark}
{\rm If the  standardness (or weak standardness) criterion  is satisfied for a filtration  $\{\xi_n\}_{n>0}$, then it is satisfied for the quotient filtration
$\{\xi_{n+k}/\xi_k\}_{n>0}$ for every $k>0$.}
\end{remark}

Finally, we emphasize that the condition of the criterion is invariant by the very definition, i.e., if it holds for a filtration, then it also holds for any isomorphic filtration. Indeed, it involves only notions related to the filtration and no other notions (e.g., metrics, as in the first statement). Although we formulated the criterion for  discrete type filtrations, it applies with minimal modifications to arbitrary filtrations.

\begin{remark}{\rm
In  particular,
a Markov process (filtration) is standard if and only if  condition~\eqref{standcrit} is satisfied.
} \end{remark}

For dyadic filtrations, this theorem was the main fact of the first
period of the theory of filtrations (see \cite{V70,V73,V73a,V94}). For general filtrations,  it was stated as a definition in
\cite{V94,VGo}.

A detailed proof will be published in the paper mentioned in the Introduction.
Let us specify the standardness criterion for homogeneous filtrations and functions with finitely many values. For clarity, we restrict ourselves to dyadic filtrations and indicator functions of sets. What does the criterion mean in this case?

In this case, an element of the $n$th partition is a set consisting of $2^n$ points equipped with the uniform measure and a structure of a binary tree (of height~$n$), and the restriction
of an indicator function is a $0-1$ vector of dimension~$2^n$. In the case of a uniform measure,   as couplings  we may take not Markov,
but one-to-one maps
(which preserve the uniform measure), that is, elements of the group preserving the tree structure, i.e., automorphisms of the binary tree. Hence the distance between the restrictions of a function to two trees is the distance between two orbits of the group of automorphisms of the tree acting on the vertices of the unit cube of dimension
 $2^n$. Thus, the standardness criterion means that for every $\varepsilon>0$ there is $N$ such that for all elements of the partition $\xi_n$ with $n>N$ from some set of measure $>1-\varepsilon$, the restrictions of the indicator function lie on orbits of the action of the group $D_n$ of automorphisms of the tree for which the (Hamming) distance is less than
$\varepsilon$.
This observation can easily be extended to the case of other filtrations.

The meaning of the standardness criterion is that the restrictions of any function to various elements of the partition have asymptotically  the same behavior with respect to the tree structure (up to automorphisms, or couplings).

\textbf{Hence standardness can be interpreted as an
invariant form of the asymptotic independence property  for filtrations, or as an analog of the independence of a sequence of random variables.}

Historically, the statement of the standardness criterion for homogeneous filtrations preceded the standardness condition~(3) given above. The fact that if the criterion is satisfied then the filtration is standard (in the homogeneous case) was proved by the author (\cite{V70, V73},
see also \cite{V94}), who simultaneously gave the first example of a nonstandard filtration.

The following corollary of the criterion explains our claims from Section~3.7.

\begin{statement}
If a central ergodic measure $\mu$ on the path space  $T(\Gamma)$ of a graded graph $\Gamma$ is nonstandard, then the limit $\lim_n \rho_n(x,y)$ does not exist for $(\mu\times\mu)$-almost all pairs $(x,y)$. If this limit exists, then it must be equal to zero, and the measure (its tail filtration) is standard.
\end{statement}

Examples of nonstandard graphs, for which the sequence of metrics does not converge to the intrinsic metric, were given above: this is, for instance, the graph of unordered pairs (UP) or the graph of ordered pairs (OP).
  These graphs are also related to the notion of a tower of measures, which is a remarkable projective limit (see Figure~\ref{fig:simplices }) on which we can see the difference between standard and nonstandard central measures.

We state without proof one of the main facts, which is
a generalization of the lacunary theorem (\cite{V68}) and was proved
in \cite{2015}.

\begin{theorem}[Lacunary theorem]
For every equipped graph $(\Gamma=\bigcup_n \Gamma_n,\Lambda)$ (respectively, for every projective limit of simplices
 $\lim_n \{\Sigma_n, \{\pi_{n,m}\}_{n,m}\}$),
 one can choose a subsequence of positive integers $n_k$, $k=1,2,\dots $,
 such that the equipped multigraph $\Gamma' = \bigcup_k \Gamma_{n_k}$ obtained by removing all levels between
 $n_k$ and $n_{k+1}$, $k=1,2, \dots$, and preserving all paths connecting them
  (respectively, the projective limit $\lim_k\{\Sigma_{n_k}, \{\pi'_{k,s}\}_{k,s}\}$  with the lumped system of projections $\pi'_{k,s}$, where $\pi'_{k,k+1}=\prod_{i=n_k}^{i=n_{k+1}-1} \pi_{i,i+1}$)
  is standard.
\end{theorem}

\begin{remark}
{\rm A consequence of the lacunary theorem is that the standardness of a central measure on the path space of a graph is a property of the projective limit, but not of the limiting simplex: by telescoping the approximation (omitting several levels of the graph) one can change the intrinsic topology, and,  by the lacunary theorem, make the tail filtration  standard  for any given measure.}
\end{remark}

\begin{definition}A graded graph is called eventually standard if there exist a telescoping that makes all central ergodic measures  standard.
\end{definition}

\begin{statement}
The graph $UP$ is not eventually standard.
\end{statement}

Now we can divide Borel actions of amenable groups
into two classes: eventually standard and not  eventually standard.
\begin{definition}
An eventually standard action of an amenable countable group is an action that admits an adic approximation (see the previous section) with a standard graded graph.
\end{definition}

Actions of this kind admit a smooth description of invariant measures. Supposedly, a Bernoulli action is not eventually standard, at least for an Abelean group.

A more delicate classification of actions should take in account specific properties of the graph.

  This question depends on the classification of orbit equivalence relations on a Borel space. But if we consider this problem up to {\it Borel isomorphisms} of the space, then, by Kechris' theorem (\cite{Kech}), all the cases with a given number of invariant measures are isomorphic. So, this  classification is too rough. On the other hand, the topological classification is too detailed  to be useful.

  Here we suggest how to divide the cases
into ``smooth'' (standard) and ``wild'' (nonstandard) ones using a finite approximation of the orbit partition or the structure of the associated branching graph (see above). More exactly, we consider eventually standard and nonstandard approximations of the orbit partition and distinguish the cases depending on the existence of a standard approximation.

\subsubsection{Characterization of standard filtrations in terms of martingales}

Now we give yet another formulation of the standardness criterion, in terms closer to the theory of random processes. Here we use the terminology and properties of Markov processes.

Consider a sequence of scalar (e.g., real) random variables $f_n$, $n>0 $, that constitute a Markov chain with finite sets of transitions (but with arbitrary state sets), in general nonstationary, and the filtration generated by this chain: $\{{\frak A}_n= \{f_k,\, k \geq n\};\, n=1,2, \dots \}$.
Assume that the zero--one law holds, i.e., the filtration is ergodic.

We will rephrase the standardness criterion for this filtration assuming that all  $f_n$ have finite first moments, and explain in what sense the criterion is a strengthening of Doob's martingale convergence theorem. The latter  says, for instance, that for all $k$
$$ \lim_n E[f_k|{\frak A}_n] = Ef_k$$
(the almost everywhere convergence of conditional expectations).

The same theorem can be applied not to the random variables themselves, but to their conditional distributions: as $n \to \infty$, the conditional distribution of the random variable $f_1$ (or of several first variables $f_1,\dots, f_k$) with respect to the $\sigma$-algebra ${\frak A}_n$, $n>k$, converges almost everywhere to the unconditional distribution. This fact uses only the ergodicity of the filtration (the zero--one law). For our purposes, it is convenient to eliminate the limiting (unconditional) distribution from this statement and rephrase the theorem as follows: the distance between the conditional distributions given $f_n=x_n$ and $f_n=y_n$ tends to zero as
$n\to\infty$ for almost all pairs $(\{x_n\},\{y_n\})$ of trajectories of the Markov chain.

The following assertion is a maximal (``diagonal'') strengthening of these theorems in which $k$ is not fixed but equal to $n$, i.e., tends to infinity simultaneously with the number of the $\sigma$-algebra with respect to which the conditional expectation is taken. This condition does not hold for all Markov chains, but only for standard filtrations (standard Markov chains).
Let us  formulate it precisely.

Consider the conditional distributions $\mu^x_n$ and $\mu^y_n$ of the first $n-1$ random variables
$ \{f_1,f_2, \dots , f_n\}$ given $f_{n+1}=x$
and $f_{n+1}=y$; the Markov property means that these conditional distributions are discrete measures on the set~$\mathbb R^n$ of vectors   (trajectories of the chain)
which depend only on the values~$x$ and~$y$.  Take a separable metric $\rho$ in the space of all trajectories of the process, and approximate it by a metric on the set of these vectors. Then consider the value  $$\bar r_n(x,y)=\min_{\psi \in \Psi_m} E K_{\bar\rho}(\mu^x_n, \mu^y_n),$$
where the minimum is taken over all Markov couplings, as in the standardness criterion. We emphasize that the Markov condition imposed on couplings (i.e., the requirement that the projections to the coordinates should preserve not only the measure, but also the order structure) is of crucial importance.

The above argument implies the following theorem.

 \begin{theorem}
 A Markov chain is standard if and only if
  $$\lim_n\int_{X\times X}r_n(x,y)\,d\mu(x)d\mu(y)=0.$$
 \end{theorem}

This means that in the standard case, and only in this case, the conditional measures of the process with fixed $f_n$ concentrate in a very strong sense. This again remind us the independence property.

Now we will relate this statement to limit shape theorems. Regarding trajectories of the Markov chain as paths in the Bratteli diagram, the standardness of a central measure on paths can be formulated as follows:
{\it for every $\epsilon>0$ there is $N$ such that for every $n>N$ there exists a vertex $v_n$ of the $n$-th level of the graph such that the measure of the set of paths that meet this level at vertices from the
$\epsilon$-neighborhood of  $v_n$ is not less than $1-\epsilon$}. Here a neighborhood is understood in the sense of the $n$-iterated arbitrary initial metric. In examples where vertices of the graph are some sort of configurations, this fact turns into a theorem on the concentration of a random distribution near some configuration.

\subsection{The difference between standard and nonstandard filtrations}

The notion of standardness, introduced in Sections 3 and 4, allows us to divide central measures and, consequently,
extreme points of the simplex $\Sigma_{\infty}(\Gamma)$
into two classes. Note that almost all paths with respect to a given measure weakly tend to this measure (regarded as an extreme point), but if the extreme point is from the first class (standard), then the paths converge in the sense of the intrinsic metric (the convergence with respect to which in this case coincides with the weak convergence). On the contrary, for a nonstandard measure,
almost all paths do not converge to this extreme point.
This remark is nothing more than another formulation of the above statement in terms of martingales.

 The following simple example is very instructive from different points of view. We present two filtrations, which are the tail filtrations of stationary Markov processes, that are finitely isomorphic but not isomorphic. One of the filtrations is standard (Bernoulli), the other one is Markov stationary, not Bernoulli, and not even standard, but weakly standard. Both filtrations are far from being homogeneous. Namely, we consider the Markov processes on the space
$$X=\prod_1^{\infty} \{0;1\}$$
with the transition matrices
$$\left(
  \begin{array}{cc}
    p & q \\
    p & q \\
  \end{array}
\right)\qquad\text{and}\qquad\left(
  \begin{array}{cc}
    p & q \\
    q & p \\
  \end{array}
\right),$$
where $q=1-p>0$ and $p \ne q$. For the first process, the tail filtration is Bernoulli, and hence standard. For the second one, the filtration is not standard, which can easily be seen by checking the violation of the criterion condition: namely, the distance between two distinct vertices is equal to
 $|p-q|$ at all levels (see also below). In this case,
 for all $n$ we have two types of elements of the partitions
$\xi_n$, denote them by $a_n,b_n$; and $\bar \rho_n (a_n,b_n)=|p-q|\nrightarrow 0$. But $\rho_n(a_n,b_n)=|p-q|^n \rightarrow 0$. The first cylinders $\{\{x_n\}:x_0=0\}$ and
$\{\{x_n\}:x_0=1\}$ cannot be distinguished by the filtration, because
it is invariant with respect to the flip $0\Leftrightarrow 1$.

\newpage

 \section{Classification of metrics and measurable functions of several variables; the matrix distribution and invariant measures}

\subsection{Invariant measures for the infinite symmetric group and  related ``big'' groups}

As an example of the problem of describing the invariant measures, consider the diagonal action of the group $S_{\Bbb N}$ of all finite permutations on the space of infinite tensors of rank $n$ with values in any Borel space $E$: $\{t_{i_1,i_2,\dots, i_n}\}$, $i_k \in {\Bbb N}$, $k=1,\dots, n$,   $t_{i_1,i_2,\dots i_n}\in E$  (it suffices to consider the interval $E=[0,1]$, and the most interesting case is $E=\{0,1\}$).

We may consider special types of tensors (symmetric, antisymmetric, etc.) or multi-coordinate actions of the  group $S_{\Bbb N}^n$, for example, separate actions on each coordinate and so on. All actions of this type will be called {\it tensor actions} of the infinite symmetric group.
Because of the inductive (locally finite) structure of the group $S_{\Bbb N}$, the problem of describing the invariant measures can be included into the above context of central measures on path spaces of branching graphs.

\begin{theorem}
All actions of the infinite symmetric group in the space of tensors are eventually standard, and, consequently, the list of all ergodic invariant measures is precompact. This means that
there is a natural approximation of this action on an eventually standard graded graph.
\end{theorem}

The term ``eventually standard" has the same sense as in the previous definition for graded graphs.
\begin{remark}{\rm
1. By a ``natural approximation" we mean  an approximation by an inductive family of actions of the finite groups $S_n$ on spaces of finite-dimensional tensors. For the group $S_{\Bbb N}$, tensor actions play a role  similar to the role of actions with discrete spectra for commutative groups ($\Bbb Z$).

2. There are many examples of nonstandard actions of $S_{\Bbb N}$, e.g., the Bernoulli action on the space
$2^{S_{\Bbb N}}$. It is an interesting problem to describe all standard actions of this group.
}\end{remark}

The reason underlying the standardness is related to the fact that an action can be extended to the group $S^{\Bbb N}$ of all permutations of $\Bbb N$ (which is the completion of the group $S_{\Bbb N}$ in the weak topology), and the representations of this bigger group are exhausted by tensor representations (\cite{Lieber}). But here we have, in a sense, the equivalence of these two facts: the possibility to enumerate the ergodic invariant measures and the irreducible representations of the group $S^{\Bbb N}$.

A similar problem and a corresponding result can be stated for the infinite unitary group and other inductive
limits of finite or compact groups. There is also a link to the representation theory  of these groups.

Specific examples of invariant measures on the space of matrices are closely related to the classification of measurable functions, which we consider in the next section.

\subsection{Classification of functions and metrics, matrix distributions as random matrices}

We start with the following problem.

\begin{problem}To classify measurable (or continuous, smooth, etc.) functions of two or more independent variables
 up to the group of automorphisms of each variable
separately.
\end{problem}

More exactly, in the case of the measure-theoretic category, two real-valued measurable functions $f,g$ of $d$ variables on the Lebesgue space $X^d$ with product
continuous measure $\mu^d$  are {\it separately metrically isomorphic} (we assume that $d=2$)
 if $$g(x,y)=f(T_1x,T_2y),$$  and {\it jointly metrically isomorphic}
 if $$g(x,y)=f(Tx,Ty),$$
where  $T_1,T_2, T$ are arbitrary invertible measure-preserving  transformations
of $(X,\mu)$. In the latter case, it is natural to assume that $f,g$ are
symmetric functions.

A very important special case is as follows.
\begin{problem} To classify Polish (= separable complete metric) spaces $(X,\rho,\mu)$
with metric $\rho$ and Borel probability measure $\mu$ up to measure-preserving isometries:
 $$(X,\mu,\rho) \sim (X_1,\mu_1,\rho_1) \Leftrightarrow \rho_1(Tx,Ty)=\rho(x,y),$$
 where $T:X\rightarrow X_1$, $T\mu=\mu_1$.
 \end{problem}

 M.~Gromov \cite{Gr} considered a Polish space with a Borel measure (called an ``mm-space'');  the author considered the more general case of an ``admissible triple'': a standard measure space with a separable metric that is a measurable function of two variables, see \cite{VPZ}. It is obvious that the classification of metric spaces is the same
as the classification of  metrics as measurable functions of two variables on the square of the space, and this is a special case of the previous problem.

\smallskip
For simplicity, we will speak about  measurable symmetric functions and the diagonal group: $T_1=T_2$.
Let $f$ be a real symmetric measurable function of two variables on the space
$(X \times X, \mu \times \mu)$ with values in some standard Borel space $R$.

The function $f$ is called a {\it pure function} if the inequality  $$\mu\{x:f(gx,\cdot)=f(x,\cdot) \bmod 0\}>0$$ holds only when $g={\rm Id} \bmod 0$,
and a {\it completely pure} function if the equality  $$f(gx,hy)=f(x,y) \bmod 0 \quad \mbox{with respect to the measure} \quad \mu\times\mu, $$ holds only when the measure-preserving transformations $g,h$ are  identical: $g=h={\rm Id}$.

Consider the infinite product space $(X^{\mathbb N},
\mu^{\mathbb N}) $ of the domain spaces.
Let $M_{\mathbb N}(R)$  be the space of all symmetric matrices
with entries from $R$.  We will omit $[0,1]$
in the notation: $M_{\mathbb N}([0,1])=M_{\mathbb N}.$
Define a map
$$ F_f:X^{\mathbb N} \to M_{\mathbb N}(R)$$
by the formula
$$F_f(\{x_n\}_{n\in \Bbb N})=\{f(x_i,x_j)\}_{i,j=1}^\infty.$$

\begin{definition}
The push-forward measure $F_f(\mu^{\mathbb N})$
 on the space of matrices
$M_{\mathbb N}(R)$ (the image of the measure  $\mu^{\mathbb N}$
under the map $F_f$) will be called the {\it matrix distribution of the measurable function $f$} and denoted by $D_f$.
\end{definition}

The matrix distribution is a generalization of the notion of the distribution of a function of one variable.
It is easy to check that the matrix distribution is an ergodic $S_{\Bbb N}$-invariant measure
with respect to the diagonal action of the group on the space of matrices.

\begin{theorem}[Classification of pure measurable functions \cite{V02}]
Let $f$ be a pure symmetric measurable function on the space
$(X \times X, \mu \times \mu)$ with values in a standard Borel space $R$.
Then the matrix distribution measure $D_f$
is a complete invariant of the function  $f$ in the sense of the above equivalence.
\end{theorem}

 In other words,

\smallskip
1) if two (not necessarily pure) real functions $f_1$ and $f_2$ are isomorphic, then
$D_{f_1}=D_{f_2}$;

\smallskip

2) if two pure measurable functions  $f_1,f_2$
are defined on spaces\break $(X_1 \times X_1, \mu_1 \times\mu_1)$ and
$ (X_2 \times X_2, \mu_2 \times\mu_2)$, respectively,
and  have the same measures  $D_{f_1}=D_{f_2}$
on the space $M_{\mathbb N}(\mathbb R)$,
then they are isomorphic, i.e.,
there exists a measure-preserving automorphism $T:(X_1,\mu_1) \to (X_2,\mu_2)$ such that
$f_2(x, y)=f_1(Tx , Ty)$ for almost all $(x,y) \in (X_2 \times X_2)$.

The special case of metrics is as follows.

\begin{theorem}[Gromov \cite{Gr}, Vershik \cite{V98}]
Let $(X,\rho,\mu)$ be a metric separable space with a nondegenerate Borel probability measure (= there are no
nonempty open sets of zero measure).
Then the matrix distribution $D_{\rho}$ is a complete invariant with respect to measure-preserving isometries.
Here the matrix distribution $D_{\rho}$ is the measure on the space of infinite distance matrices $\{r_{i,j}\}$
that is the image of the Bernoulli measure $\mu^\infty$
under the map $F_{\rho}:X^{\infty}\rightarrow M_N(\Bbb R)$ given by
$F_{\rho}(\{x_n\})=\{\rho(x_i,x_j)\}$.
\end{theorem}

Gromov's first proof used some analytic tools (the moment problem); the proof in \cite{V98} is very simple and based on the individual ergodic theorem.

The matrix distributions of  symmetric functions of two variables (e.g., metrics) is a special class of
ergodic invariant measures on the space $M_{\Bbb N}$ which can be directly characterized. Here we consider
only the case of metrics regarded as measurable functions of two variables.

The previous theorem implies the uniqueness of an admissible triple with a given matrix distribution up to measure-preserving isometries. The next theorem gives conditions that guarantee the existence of an admissible triple with a given matrix distribution.

Denote by $\cal R$ the closed cone of all distance matrices in $M_{\infty}({\Bbb R}_+)$:
$$ {\cal R}=\{\{r_{i,j}\}: r_{i,j}\geq 0,\; r_{i,j}=r_{j,i}, \; r_{i,j}+r_{j,k}\geq r_{i,k},\; i,j,k
\in {\Bbb N}\}.$$

\begin{theorem}[Characterization of matrix distributions, reconstruction of the metric triple, \cite{V02,V12}]
A probability measure $\tau$ on the space of distance matrices $\cal R$ is the matrix distribution $D_{\rho}$ of some
metric if and only if the following conditions are satisfied:

{\rm 1)} The measure $\tau$ is invariant under the action of the infinite symmetric group $S_{\Bbb N}$ by simultaneous permutations of rows and columns.

{\rm 2)} The measure  $\tau$ is ``simple,'' which means the following. Consider the map $T$ that, by definition, sends $\tau$-almost every matrix $r\equiv\{r_{i,j}\}$ to the {\it empirical distribution} of the columns of  $r$; this is a well-defined (because of the ergodic theorem) map $T:({\cal R},\tau) \rightarrow (\mbox{\rm Meas}[R^{\infty}_+], \theta\equiv T_{*}\tau)$.
The condition is that it is an isomorphism between these measure spaces.
 Here $\mbox{\rm Meas}[R^{\infty}]$ is the space of all Borel probability measures on the space $R^{\infty}$.
\end{theorem}

Note that the map $T$ is well  defined on a set of full $\tau$-measure.
The second condition means that one can restore the measure-metric space if one knows the joint distribution of the distance from a random point to a countable dense set of points.

 \begin{proof}
The necessity of the first condition is trivial by the definition of the matrix distribution $D_{\rho}$. Condition 2 means that
almost all matrices with respect to the measure $D_{\rho}$ can be restored if we know the empirical joint distribution of
the distances from one given point. In other words, this means that with probability 1 the coincidence of the
joint distributions of the sequence of functions of the second variable $\{\rho(x_i,\cdot)\}_i$ and  $\{\rho(x'_i,\cdot)\}_i$
implies the equality $x_i=x'_i$, $i=1,2, \dots$. But this follows from the pureness of the function $f$.

Assume that we have a simple $S_{\Bbb N}$-invariant measure $\tau$ on the cone $\cal R$. Then for $\tau$-almost all
matrices $r$, which we consider as distance matrices on~$\Bbb N$,
 we can define a metric space $X_r$, which is the completion of the
metric space $({\Bbb N}, r)$ with respect to the metric $r$, and define a unique Borel measure~$\mu_r$ on $X_r$.
The group $S_{\Bbb N}$ acts naturally on $\Bbb N$, and, consequently, this action can be extended by continuity to the space $X_r$. Denote the extension of the metric $r$ to $X_r$ by $\rho_r$.  We obtain the metric triple $(X_r, \rho_r, \mu_r)$.
Note that,  up to a measure-preserving isometry, the triple $(X_r, \rho_r, \mu_r)$ does not depend on the matrix $r$. Indeed,
if we have another matrix of the type $r'=grg^{-1}$, $g \in S_{\Bbb N}$, then the action of $g$ can be extended to $X_r$ as a measure-preserving isometry. But because of the ergodicity of the action of $S_{\Bbb N}$ in $X_r$,  for $\tau$-almost all distance matrices $r$ and for $\tau$-almost all distance matrices $r'$ we can choose in $X_r$ a dense sequence of points $\{x_i\}$ whose distance matrix is~$r'$. To complete the proof, it suffices to prove that the matrix distribution
of the triple $(X_r,\rho_r, \mu_r)$ is the measure~$\tau$.
\end{proof}

\subsection{ The problem of classification of functions of several variables and Aldous' theorem}

Now we relate this classification to a remarkable result by D.~Aldous. In his paper \cite{A81}, as well as in some subsequent papers (see, e.g., the book \cite{Kal}), the list of such invariant measures for the group $S_{\Bbb N}$ is obtained. The answer is given in terms of some measurable function in an undefined space. The missing link is  a canonical form of a function with a given matrix distribution and, more generally,  a connection to the intimately related classification problem. This classification problem plays, in a sense, the role of the dual problem. In my papers \cite{V02,V12,VH},  an example of such an approach was suggested, which looks like a generalization of the ``ergodic method.'' In the spirit of this connection, we obtain an explanation for the following question: what is the meaning of the third variable in  Aldous' theorem?
The matrix distributions of measurable functions do not exhaust the list of ergodic $(S_{\infty}\times S_{\infty}$)-invariant measures in the space of matrices $M_{\infty}(\Bbb R)$. We formulate a classification problem for which the class of
 invariants --- a new kind of matrix distributions ---
coincides with the set of all ergodic $(S_{\infty}\times S_{\infty})$-invariant measures.

\begin{problem} Consider the space of all measurable functions of three variables $(x_1,x_2,x_3) \mapsto f(x_1,x_2,x_3)$, $x_i \in X$, $i=1,2,3$ (here $(X,\mu)$ is a standard measure space) that are symmetric in the sense that $f(x_1,x_2,z)=f(x_2,x_1,z)$ with the following equivalence relation:
 $$g \sim f \Leftrightarrow g(x_1,x_2,x_3)=f(T_1 x_1, T_2 x_2, S_{x_1,x_2}x_3),$$
 where $T_1, T_2, S_{x_1,x_2}$ are measure-preserving transformations of $(X,\mu)$
 and $\{S_{x_1,x_2}\}$ is a measurable function on $(X,\mu)^2$ with values in $\Aut(X,\mu)$.

Thus, the group of automorphisms that underlies the classification is the  skew product
$$\Aut (X,\mu)^2 \rightthreetimes (X^2 \rightarrow \Aut(X,\mu)) \subset \Aut (X, \mu)^3.$$
We call this equivalence the $(2-1)$-equivalence.
\end{problem}

It is more convenient to reformulate this problem as a problem of classification of functions
of two variables with values in classes of metric equivalence of functions of one variable (roughly
speaking, with values in Borel probability measures on some Borel space).

 Let $f$ be a measurable function on the space $(X,\mu)^3$ with values in a standard Borel space $E$.
 Consider the map  $$\{x_i,z_{i,j}\}_{i,j} \mapsto \{f(x_i, x_j,z_{i,j})\}_{i,j},$$ where $\{x_i\}, \{z_{i,j}\}$ are mutually independent random variables ($z_{i,j}=z_{j,i}$) with values in $X$ and with the common distribution $\mu$. The images of the sequences under this map are matrices from the space of matrices $M_{\Bbb N}(E)$, and the image $Df$ of the measure $\mu^{N}\times\mu^{N^2}$ is called the ``matrix distribution'' of the function $f$ (in the framework of the given problem).

A generalization of Aldous' theorem can be formulated as follows.

\begin{theorem}
{\rm1.} A complete invariant with respect to the $(2-1)$ classification of measurable pure functions $f:X^3\rightarrow E$ of three variables with values in the space $E$ is the matrix distribution $Df$ of $f$.

{\rm 2.} Every ergodic $S_{\Bbb N}$-invariant (with respect to the diagonal action) measure on the space of matrices $M_{\Bbb N}(E)$ with entries in a Borel
  space $E$ is the matrix distribution of some function $f$, which is unique up to the $(2-1)$ equivalence.
\end{theorem}

\subsection{Generalized metric measure spaces}

 We briefly describe a generalization of the previous classification problem and consider the starting point of the theory of nonstandard filtrations, which is closely related
 to this problem.

The set of all matrix distributions $D_f$ of measurable functions of two variables $f$ is not weakly closed in the space of all measures on the space of matrices $M_{\Bbb N}(\Bbb R)$. The same is true if we consider only metrics as functions $f$. What meaning do elements of the closure have? We will give an interpretation of limit elements; they are also matrix distributions,  but  of ``generalized,'' or ``random,'' or ``virtual'' measurable functions. We restrict ourselves only to the case when functions are metrics.

A generalized metric on a measure space $(X,\mu)$ is a symmetric random field $\xi(\cdot,\cdot)$ on $(X\times X, \mu\times\mu)$ with values in ${\Bbb R}_+$ subject to the triangle inequality $\xi(x,y)+\xi(y,z)\geq \xi(x,z)$ almost everywhere.
We can regard  this field in the traditional way as a function of three variables: $\xi(x,y)=f(x,y,\omega )$, where $\omega \in \Omega$, with $\Omega$ being an ``indeterminate space,'' and $\nu$ is a measure on~$\Omega$.  We define the {\it generalized matrix distribution} $D_{\xi}$ as the measure on the space   $M_{\Bbb N}(\Bbb R)$ that is the image of the measure $\mu^{\infty}\times\nu^{{\infty}^2}$ under the map
$$\{x_i\}_i\times \{\omega_{i,j}\}_{i,j} \rightarrow \{\xi(x_i,x_j,\omega_{i,j})\}_{i,j}.$$

If $\xi$ is not a deterministic function (i.e., $\omega$ is not a function of $x_i,x_j$), then this generalized matrix distribution is not a matrix distribution of any measurable function.\footnote{The correct definition of a random field on a measure space is not in the spirit of Kolmogorov's definition of random processes, but is more similar to, although not the same as, the definition of Gelfand--Ito generalized processes. The main
point is just the existence of a generalized matrix distribution.}

\begin{statement}
Every generalized matrix distribution is an $S_{\Bbb N}$-invariant ergodic measure on $M_{\Bbb N}(\Bbb R)$.
The set of all generalized matrix distributions is a weakly closed subset of the space of probability measures on $M_{\Bbb N}(\Bbb R)$.
\end{statement}

A typical situation when a generalized metric appears is as follows  (we met it when we introduced the notion of a standard graph in Section 3).

Consider a sequence of admissible triples (= admissible measure metric spaces) $(X,\rho_n,\mu)$; assume that the
space and the measure are fixed and the metric varies. Assume that $\{\rho_n\}_n$, regarded as a sequence of functions of two variables, does not converge literally, but the sequence of its distributions (with respect to the measure $\mu\times\mu$), regarded as a sequence of measures on ${\Bbb R}_+$, does weakly converge. More generally, assume that the sequence of matrix distributions $D_{\rho_n}$ weakly converges to a measure $D_{\rho_{\infty}}$. Then the limit measure can be or not the matrix distribution
of some admissible triple. In other words, a measure metric space is a generalized space, but not all generalized spaces are ordinary
measure metric spaces. So, the notion of a generalized matrix distribution is a generalization of the notion of a matrix distribution.

\begin{conjecture}
Consider a graded graph $\Gamma$ and an ergodic central measure~$\mu$ on the path space $T(\Gamma)$. Let $\rho=\rho_0$
be a metric on  $T(\Gamma)$, and assume that  $(T(\Gamma),\rho, \mu)$
is an admissible triple. Then the sequence of spaces $(T(\Gamma),\rho_n, \mu)$ converges to a generalized admissible triple. If the graph is standard, then,  by definition, the limit is again an admissible triple; if the tail filtration with respect to a central measure on a nonstandard graph is not standard, then the limit is a generalized but not ordinary metric triple.
\end{conjecture}

I can confirm this conjecture in several cases. It is important that some characteristics of measure metric spaces
still exist for generalized admissible triples, for example,  $\varepsilon$-entropy. We will return to this very interesting question in future publications; see \cite{2015}.

This analysis should make clear the following comment to the statement of Aldous' theorem in the previous section:
every ergodic $S_{\infty}$-invariant measure on the space of symmetric matrices is the generalized matrix distribution of a generalized measurable function.

\newpage
\section{Examples: the exit boundaries for the Pascal and Young graph, for the dynamic tree, and for random subgroups}

 In this section, we will consider several examples.
 Some of them are old, but the point of view is rather new.
 We consider the absolute boundary for the Pascal graph of an arbitrary dimension and the Young graph. All these graphs are standard. We formulate a conjecture on the entropy of standard central measures on the path space which generalizes the old entropy conjecture from \cite{VeK85}.  We consider Thoma's theorem on the characters of the infinite symmetric group from the point of view of {\it totally nonfree actions of the infinite symmetric group}. A new result on the phase transition and
 the loss of ergodicity in the absolute boundary problem on the dynamical graph of a tree is presented.  Finally, we mention a
 link between central measures and factor representations of type II$_1$ for the infinite symmetric group.
We also mention results on ``random subgroups'' (IRS) and their relation to totally nonfree actions. Also, we give a link between these questions and the theory of factor representations.

There are many papers and results on list of characters and invariant measures for graphs, groups, etc. These results constitute the
foundation of the new area of the representation theory of groups and algebras, the asymptotic representation theory, which has many intersections with combinatorics, probability theory, etc.  We mention only several works in this area, where one can find further references: \cite{ICM, Kerov, Olsh, Ok, VO, CST}.

\subsection{The absolute boundary for the Pascal graph of an arbitrary dimension}

Consider the infinite Pascal triangle $\Gamma$ as a graded graph
(see Figure~\ref{fig:Pascal}). The space $T(\Gamma)$ of infinite paths is, in a natural sense, isomorphic to the product $X=\prod_{n=1}^{\infty}\{0;1\}$.

\begin{theorem}
The absolute boundary $\Erg(\Gamma)$ for the infinite Pascal triangle is the unit interval $[0,1]$.
More exactly, every ergodic central  measure $\mu$  on $T(\Gamma)$ is a Bernoulli measure on $X$: $\mu=\mu_p=\prod_{n=1}^{\infty} (p,1-p)$  with $ p\in [0,1].$
\end{theorem}

Indeed, in terms of the space $X$, the centrality of a measure on $T(\Gamma)$  is the invariance of the measure under the action of the infinite symmetric group $S_{\Bbb N}$ by permutations of coordinates. The classical de Finetti theorem gives the required result.

For the same reason, the absolute boundary of the $d$-dimensional Pascal graph $ ({\Bbb Z}_+)^d$ is the $(d-1)$-dimensional simplex
(here Bernoulli measures are of the form $(p_1,p_2,\dots, p_d)$, $\sum_i p_i=1$,  $p_i\geq 0$, $i=1, \dots, d-1$).

A much more complicated example of calculating the absolute boundary, which we briefly discuss below, is for the
so-called Young graph, the graph of Young diagrams.

All these examples of graded graphs are the so-called Hasse diagrams of lattices which are the
lattices of finite ideals of some posets, and hence
distributive. We want to formulate a conjecture on the absolute boundary in this situation. Recall that every distributive lattice $L$ is the lattice $L=L(Y)$ of all finite ideals of a poset $Y$.

\begin{conjecture}
The absolute boundary of a graded graph $\Gamma$ that is the Hasse diagram of a distributive lattice $L(Y)$ is the set of all monotone positive functions $f\leq 1$ on the space of all minimal infinite ideals of the poset $Y$. \end{conjecture}

For example, the Young lattice is the lattice of finite ideals of $[{\Bbb Z}_+]^2$, and infinite minimal ideals are unions  of rows and columns. Thus, a monotone function is $f:\Bbb N\cup\Bbb N\cup \{\infty\}\rightarrow [0,1]$, its values are
$f(n,0)=\alpha_n$, $f(0,n)=\alpha_{-n}$, $f(\infty)=\gamma$, where $\{\alpha_{\pm n},\gamma\}$ are Thoma's parameters. The characters of the infinite symmetric group are parameterized by the frequencies of rows and columns of increasing sequences of Young diagrams.
For the multidimensional Young lattices, this conjecture was suggested by Vershik and Kerov in the 1980s, but only now it is close to being proved.

\subsection{Dynamic graphs, or pascalization, and random walks on groups}
Now we consider a very interesting
and new result on the absolute boundary of the dynamic Cayley graph of a group with finitely many generators. We start with the definition of the following operation on graphs.

 Let $\gamma$ be an arbitrary connected locally finite graph without multiple edges with a distinguished vertex~$v_0$. The
graph $\Gamma (\gamma,v_0)$, called the \emph{dynamic graph}, or the pascalization, of the graph $\gamma$, is the  $N$-graded graph whose $n$th level is a copy of the set of vertices of $T$ connected with the distinguished vertex $v_0$ by walks
of length~$n$ (or, which is the same, by paths of length at most~$n$ and of the same parity as~$n$); the unique vertex of the zero level will be denoted by~$\varnothing$. Two vertices in $\Gamma (\gamma,v_0)$ are adjacent if and only if they lie at adjacent levels and their copies are connected by an edge in the graph $T$. If a connected graph $\gamma$ has no odd cycles, then choosing the initial vertex turns it into an $\Bbb N$-graded graph for which the dynamic graph $\Gamma (\gamma,v_0)$ is defined in~\cite{VN} and called the  \emph{pascalization}\footnote{The term is due to the fact that if $\gamma$ is the chain (i.e.,  $\gamma=\{n\in \Bbb Z\}$ with the grading
 $n\rightarrow |n|\in \Bbb N$), then the ``pascalization'' of $\gamma$ is the Pascal graph (the infinite Pascal triangle). As shown in~\cite{VN}, the branching graph of the infinite-dimensional Brauer algebra is the pascalization of the Young graph; for another example of pascalization, see~\cite{GK}.} of $T$.

\def\Cay{\operatorname{Cay}}

If $G$ is a countable group with a fixed finite collection of generators $A=A^{-1}$, then $\Cay(G,A)$ is the Cayley graph of $G$ with respect to the set of generators $A$.

Now we apply the above  construction with  $\gamma$ being the Cayley graph of the group $G$ and obtain a graded graph which we denote by $\Gamma(G,A)$; we will call it the  {\it dynamic graph}, or the {\it pascalization}, of the Cayley graph $\Cay(G,A)$. We may say also that this is the graph of the simple random walk on the Cayley graph $\Cay(G,A)$.

For $q>0$, let $T_{q+1}$ be a tree  of valence~$q+1$. The case  $q=2k$ corresponds to the Cayley graph of the free group with $k$ generators.  We consider the simple random walk starting from the origin and having equal probabilities of all edges. The graded graph $\Gamma (T_{q+1},v_0)$   is an  $\Bbb N$-graded graph whose $n$th level is a copy of the set of vertices of $T_{q+1}$ connected with the distinguished vertex $v_0$ by walks of length $n$.
We want to find the set $\Erg (\Gamma(T_{q+1}))$ of all central measures of the graph $\Gamma (T_{q+1},v_0)$, or, in other words, the absolute boundary $\Erg (\Gamma(T_{q+1}))$ of this dynamic graph $\Gamma (T,v_0)=\Gamma(T_{q+1})$.

\begin{theorem}[Vershik--Malyutin \cite{VM2015}]
For $q\ge2$, the set $\Erg (\Gamma(T_{q+1},v_0))$ of all ergodic central measures on the space $T(\Gamma(T_{q+1},v_0))$ of infinite paths in the dynamic graph $\Gamma(T_{q+1},v_0)$ over the $(q+1)$-homogeneous tree~$T_{q+1}$ (i.e., the absolute  boundary) coincides with the following family of Markov measures:
\begin{equation*}
\Lambda_q:=\left\{\lambda_{\omega,r} \mid \omega\in\partial T_{q+1}, r\in\left[1/2,1\right]\right\}.
\end{equation*}
Thus, the absolute boundary is homeomorphic (in the weak topology) to the product
$$
\partial T_{q+1} \times \left[1/2,1\right].
$$
\end{theorem}

On the other hand, this set can be identified with the set of all minimal positive eigenfunctions of the Laplace operator
on the tree with eigenvalues greater than some constant ($\sqrt q$). For eigenvalues that are less than this constant, we obtain a nonergodic measure, corresponding in the formula above to values $r<\frac{1}{2}$. So, when $r$ passes from greater values through~$\frac{1}{2}$, we have a kind of phase transition, namely, we lose the ergodicity of the central measure. The central measures that correspond to these values
can be decomposed into integrals over ergodic measures on the Poisson--Furstenberg boundary with full support.


\begin{figure}[h!]
\center{\includegraphics[scale=0.9]{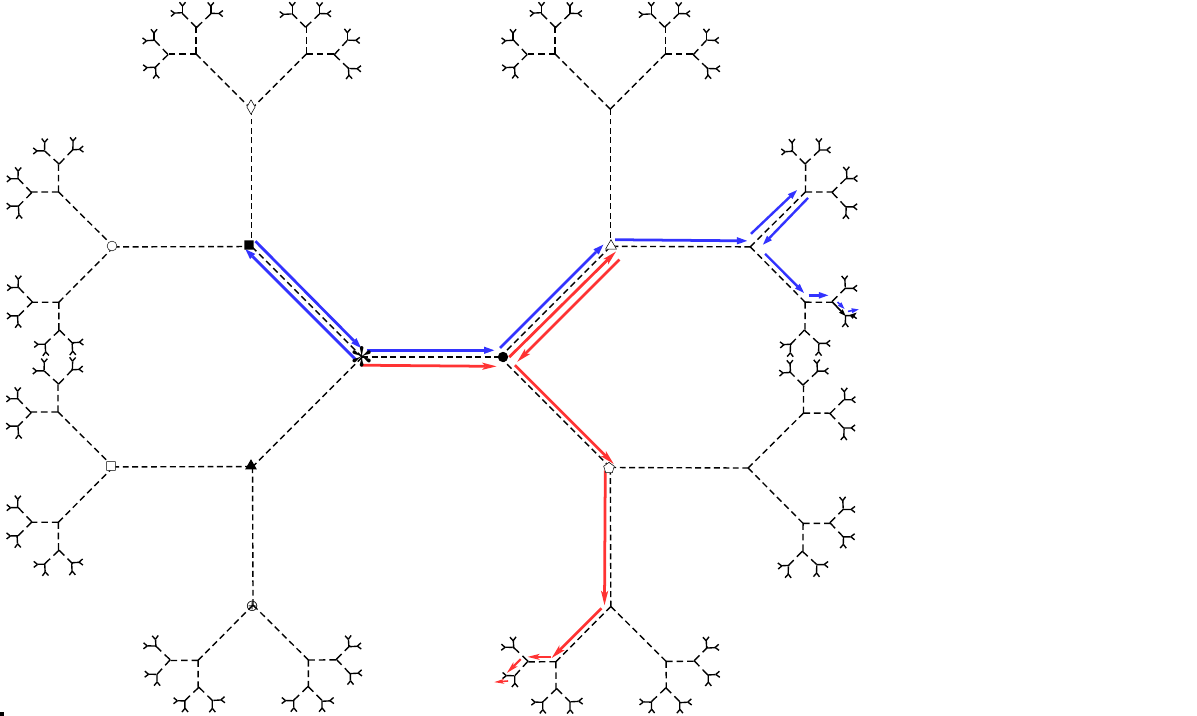}}
\caption{A random walk on a tree.}
\label{fig:T3}
\end{figure}



\begin{figure}[h!]
\center{\includegraphics[scale=0.85]{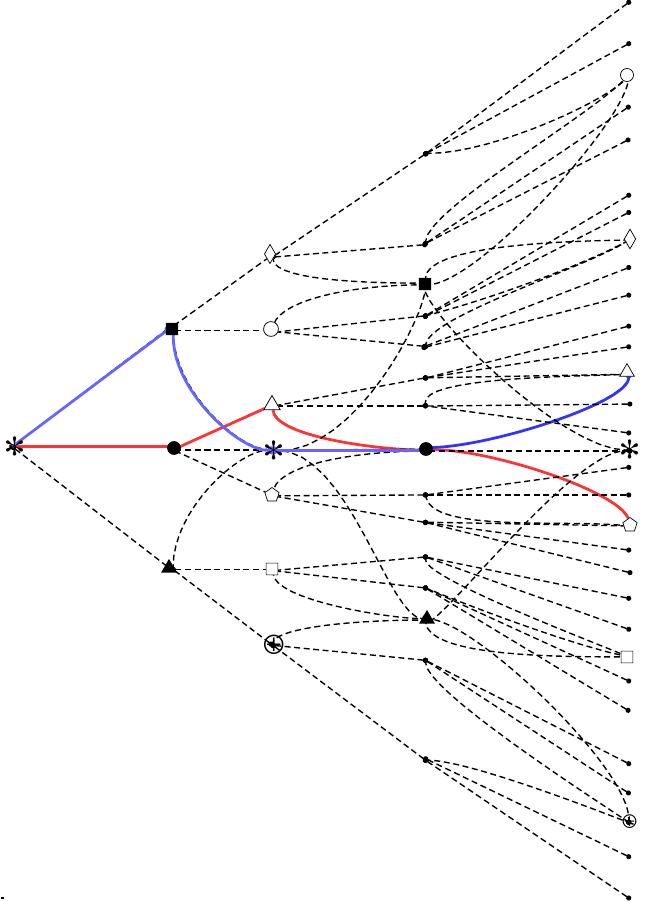}}
\caption{The boundary of  $T_3$.}
\label{fig:Tree Pascalization}
\end{figure}

The problem of calculating the exit boundaries is more general than, and slightly different from, the corresponding problem
for the Poisson--Furstenberg boundaries for random walks. There are few groups for which this problem is solved.

\subsection{Random subgroups, characters, and their description
for the infinite symmetric group}

Now we will discuss another application of  invariant measures.

Let $G$ be a  group and $L(G)$ be the lattice of all its subgroups. The group~$G$ acts on $L(G)$ by conjugation:
 $$L(G) \ni H\mapsto gHg^{-1}\in L(G).$$
We will call this action the adjoint (AD) action, and measures invariant under the AD-action will be called AD-invariant measures, or merely AD-measures.

What are continuous AD-measures (finite or $\sigma$-finite) on $L(G)$?
Sometimes, such measures are called ``random subgroups'' (IRS), see \cite{Ab}.

The following important observation gives a link between characters on the group $G$ and invariant measures on the lattice $L(G)$.
\begin{statement}
The function
$$\chi(g)=\mu\{x:gx=x\}$$
 on the group, where $\mu$ is an invariant measure on a $G$-space $(X,\mu)$, is a character of the group $G$.
This means that $\chi(e)=1$, $\chi(hgh^{-1})=\chi(g)$, and $\chi$ is a positive definite function.
\end{statement}

Of course, not all characters of a general group have this form, but for some groups, including the infinite symmetric group $S_{\Bbb N}$, this is a universal formula.

If an AD-measure $\mu$ on the lattice $L(G)$ is concentrated on the
set of {\it self-normalizers}, or subgroups $H$ such that $gH=Hg\Rightarrow g \in H$, then the formula above looks as
$$\chi_{\mu}(g)=\mu\{H:g\in H\};$$
this is the measure of the set of subgroups that contain the element $g$, and we called such measures ADS-measures.

From this point of view, two measures $\mu_1$ and $\mu_2$ are congruent if $\chi_{\mu_1}(g)=\chi_{\mu_2}(g)$. This is a useful
identification of invariant measures on $L(G)$.

We formulate the following problem.

\begin{problem}
 For a given group $G$, describe all continuous ergodic invariant measures on $L(G)$ and those of them that are concentrated on the set of self-normalizers (AD-measures and ADS-measures).
 \end{problem}

 In the next section, we will formulate the same problem from the point of view of  totally nonfree actions of groups.

In the papers  \cite{Fr2010, Fr2012} by the author, the problem was solved for the infinite symmetric group. The relation to the theory of exit boundaries is very simple: the solution of the problem gives also a description of the absolute boundary of the Young graph, or the list of Thoma characters.

Recall Thoma's theorem \cite{Tho}.
In our terms, the absolute boundary of the Young graph is the so-called Thoma simplex $\{\alpha_n\}_{n\in \Bbb Z}$, where
\begin{multline}\{0 \leq \dots \alpha_{-n}\leq \alpha_{-(n-1)}\leq \dots \leq \alpha_{-1}; \quad\alpha_0; \\ \alpha_1  \geq \dots \geq \alpha_{n-1} \geq \alpha_n \dots \geq 0;\quad \sum_{i\in \Bbb Z}\alpha_i=1\}. \end{multline}

\begin{figure}[h]
\hspace*{2cm}{\includegraphics[scale=1.2]{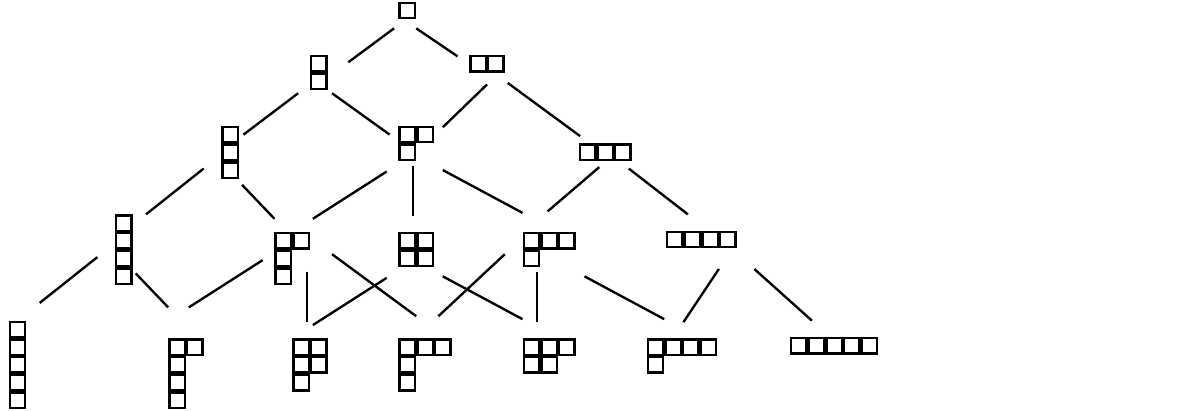}}
\caption{The Young graph, which is the branching diagram of irreducible representations of the symmetric groups.}
\label{fig:Young}
\end{figure}

The lattice $L(S_{\Bbb N})$ is very large and contains very different
types of subgroups. Nevertheless,
the support of any invariant measure on $L(S_{\Bbb N})$ consists of subgroups of a very
special kind: so-called signed Young groups.
The topology and the Borel structure on $L(S_{\Bbb N})$ are
defined as usual; this is a compact (Cantor-like) space.

\begin{definition}[Signed partitions]
 A signed partition $\eta$ of the set $\Bbb N$
 is a finite or countable partition $\Bbb N=\cup_{B\in{\cal B}}B$ of
 $\Bbb N$ together with a decomposition ${\cal
 B}={\cal B}^+\cup{\cal B}^-\cup{\cal B}^0$ of
 the set of its blocks, where ${\cal B}^0$ is the set of all
 single-point blocks, elements of ${\cal B}^+$ are called positive
 blocks, and elements of ${\cal B}^-$ are called negative blocks (thus
 each positive or negative block contains at least two points), and we
 denote by $B_0$ the union of all single-point blocks:
 $B_0=\cup_{\{x\}\in{\cal B}^0}\{x\}$.

 Denote the set of all signed partitions of $\Bbb N$ by ${\rm SPart}(\Bbb N)$.
\end{definition}

Recall that in the theory of finite symmetric groups, the Young
subgroup~$Y_\eta$ corresponding to an ordinary partition
$\eta=\{B_1,B_2,\dots, B_k\}$
is $\prod_{i=1}^k S_{B_i}$,
where $S_B$ is the symmetric group acting on $B$. We will define a more general
notion of a {\it signed Young subgroup}, which makes sense both for
finite and infinite symmetric groups.
We will use the following notation: $S^+(B)$ is the symmetric group
of all finite permutations of elements of a set $B\subset \Bbb
N$, and $S^-(B)$ is the alternating
group on $B$.\footnote{Traditionally, the alternating group is denoted by $A_n$;
V.~I.~Arnold was very enthusiastic about the idea to denote it by
$S^-_n$ in order not to confuse it with the Lie algebra $A_n$; I agree
with this idea.}

\begin{definition}[Signed Young subgroups]
The signed Young subgroup $Y_{\eta}$ corresponding to a signed
partition $\eta$ of $\Bbb N$ is
$$Y_{\eta}=\prod_{B\in{\cal B}^+}S^+(B)\times \prod_{B\in{\cal B}^-}S^-(B).$$
\end{definition}

Note that  on the set $B_0\subset
\Bbb N$, the subgroup $Y_{\eta}$ acts identically, so that the partition into the orbits of $Y_{\eta}$
coincides with $\eta$.

It is not difficult to describe the conjugacy class of Young subgroups
with respect to the group of inner automorphisms:
$Y_{\eta}\sim Y_{\eta'}$ if and only if $\eta$ and $\eta'$ are
equivalent up to the action of $S_{\Bbb N}$. But it is more important to consider the conjugacy
with respect to the group of outer automorphisms. This is the group
$S^{\Bbb N}$
of all permutations of $\Bbb N$.
Denote by $r_0^\pm$ the number of infinite positive
(respectively, negative) blocks, and by $r^{\pm}_s$ the number of
finite positive
(respectively, negative) blocks of length $s>1$. Obviously, the list
of numbers $\{r^{\pm}_0,r^{\pm}_1, \dots \}$ is a complete set of
invariants of the group of outer automorphisms.

\subsection{Limit shapes and entropy}
A very important notion concerning central measures (especially standard measures) is the notion of entropy. A special case of this notion, the entropy of the Plancherel measure on the space of Young tableaux, was suggested by the author in the 1980s and appeared in the paper \cite{VeK85} joint with S.~V.~Kerov (the ``entropy conjecture''). In that paper, a two-sided bound on the entropy was obtained, and the question was about the existence of the a.e.\ limit of the entropy. Recently, A.~Bufetov \cite{Buf} proved the existence of the limit in the sense of $L^2$.

The general problem, in the spirit of C.~Shannon's theory, is as follows.

Consider a branching graph $\Gamma$ and a central measure $\mu$ on the path space~$T(\Gamma)$.
Let $\mu_n$ be the projection of $\mu$ to the level $\Gamma_n$, and let $h_n=H(\mu_n)$ be the entropy of the
measure $\mu_n$. When does the limit
$$\lim \frac{\log(\mu(\gamma_n))}{h_n}={\rm const}$$
exist for  $\mu$-almost all paths $\gamma$?

 Here $\gamma_n$ is the vertex of the path $\gamma$ at the $n$th level.
   For the case of the Plancherel measure on the Young graph, we have just $\sqrt n$:
   $$\sum_{\lambda\vDash n} \frac{\dim (\lambda)^2}{n!}\log \frac{\dim (\lambda)^2}{n!}\equiv E_{\mu_n}[{\mu_n}(\lambda)] \approx \sqrt n. $$
So, we obtain the following conjecture:
$$\lim_n \frac{1}{\sqrt n}\log \mu(t_n)={\rm const}, \quad t_n \in Y_n,\,  t \in T(Y),$$
for almost all infinite Young tableaux $t$ with respect to the Plancherel measure~$\mu$.

 It may happen that the answer is positive for standard central measures on all graphs. This question is closely related to the theory  of entropy of random walks on  groups and graphs.

\subsubsection{The list of AD-measures for $ S_{\Bbb N}$}

Consider a sequence of positive numbers $\alpha=\{\alpha_i\}_{i\in
\Bbb Z}$ such that
$$\alpha_i\geq \alpha_{i+1}\geq 0\mbox{ for }i>0;
\quad \alpha_{i+1} \geq \alpha_i\geq 0\mbox{ for }i<0;\quad
\alpha_0\geq 0;\quad \sum_{i\in \Bbb Z} \alpha_i= 1.$$

Consider a sequence of $\Bbb Z$-valued independent random variables $\xi_n$, $n\in
\Bbb N$, with the distribution
$$\operatorname{Prob}\{\xi_n=v \}=\alpha_v \quad \mbox{for all} \quad
n\in {\Bbb N},\, v \in \Bbb Z.$$
Thus we have defined a Bernoulli measure  $\mu_{\alpha}$ on the space of
integer sequences $$\Bbb Z^{\Bbb N}=\{\xi=\{\xi_n\}_{n\in \Bbb N}:  \xi_n \in \Bbb Z\}.$$

\begin{definition}[A random signed Young subgroup and the measures
$\nu_{\alpha}$]
Fix a sequence $\alpha=\{\alpha_i,\, i\in \Bbb Z\}$ and the corresponding
Bernoulli measure  $\mu_{\alpha}$; for each
realization of the random sequence $\{\xi_n\}$, $n\in \Bbb N$,
  with the distribution~$\mu_\alpha$, define a random signed partition
  $\eta(\xi)$ of $\Bbb N$ as follows:
$$\eta(\xi)=\{B_i\subset {\Bbb N},\, i\in \Bbb Z\},\qquad
B_i:=\{n \in {\Bbb N}:\xi_n=i\};$$
here ${\cal B}^+=\{B_i,\,i>0\}$, ${\cal B}^-=\{B_i,\,i<0\}$,
and $B_0$ is understood as the union of one-point blocks.
The correspondence $\xi\mapsto \eta(\xi)$ defines a probability
measure on the set $\operatorname{SPart}(\Bbb N)$ of signed partitions
(or a random signed partition), which is  the image of the Bernoulli measure
$\mu_{\alpha}$. The correspondence $\xi \mapsto Y_{\eta(\xi)}$ defines
a measure, which we denote by $\nu_{\alpha}$, on the set of signed
Young subgroups, i.e., a measure on the lattice
$L(S_{\Bbb N})$ of subgroups of  $S_{\Bbb N}$.
\end{definition}

Note that all nonempty blocks of the random signed partition $\eta({\xi})$
that consist of more than one point are infinite
with $\nu_{\alpha}$-probability one.

\def\St{\operatorname{St}}
\def\Stab{\operatorname{Stab}}

\begin{definition}
Let $G$ be a countable group. A measure-preserving action of~$G$ on a measure space $(X,\mu)$ is called totally nonfree (TNF) if the map
$$ \St: X \rightarrow L(G),\qquad \St(x)=\Stab_x \in L(G),$$
where $\Stab_x=\{g\in G :gx=x\}
$, is an isomorphism (i.e., an injection $\bmod 0$).
\end{definition}

Now we describe the list of all AD and TNF measures for the group
$S_{\Bbb N}$.

\begin{theorem}
Every measure $\nu_{\alpha}\in L(S_{\Bbb N})$ is Borel ergodic AD-invariant.
Every ergodic probability Borel AD-invariant measure on the lattice $L(S_{\Bbb N})$ coincides up to congruence
 with the measure $\nu_{\alpha}$ for some $\alpha$.
\end{theorem}

\subsection{The link to dynamics and factor representations of groups}

The previous result gives an important example of a TNF action.
\begin{statement}
 The adjoint action of the group $S_{\Bbb N}$ on the lattice $L(S_{\Bbb N})$
with any AD-measure is a TNF action.
\end{statement}

Why such actions are important for representation theory?

Recall that in the framework of the well-known von Neumann, or groupoid, construction of the $W^*$-factor of type II$_1$ generated by this action of the group~$G$ on $(X, \mu)$,  we have an action of the group $G\times G$ on $X\times X$ and the corresponding equivalence relation
$$\tau\equiv\{(x,y):\exists g\in G:y=gx\}$$ with a $\sigma$-finite $(G\times G)$-invariant measure $\Psi_{\mu}$ on $X\times X$.

The following theorem was proved in \cite{Fr2010}.
\begin{theorem}\label{th:17}
Assume that we have an ergodic, measure-preserving TNF action of a countable group $G$ on a standard measure space $(X,\mu)$.
Then the (Koopman) representation of the group $G\times G$ in $L^2_{\Psi}(\tau)$ is irreducible. The restrictions of this representation to the
left and right components $G$ are factor representation of type ${\rm II}_1$.
\end{theorem}
Note that this factor as a $W^*$-algebra is  the weak closure of the set of all operators of the semidirect product of the group of unitary operators corresponding to the elements of the group $G$ and the commutative algebra $L^{\infty}(X,\mu)$ of  measurable bounded functions on
 $(X,\mu)$. If the action is TNF, then this factor is generated by the operators of the group $G$ only; in other words,  multiplicators from the $W^*$-algebra $L^{\infty}(X,\mu)$ belong to the weak closure of the algebra generated by the operators from the group $G$.

This is a new source of factor representations of groups. For the group~$S_{\Bbb N}$, Theorem~\ref{th:17} includes the
following result by Vershik and Kerov (\cite{VeK81}).

\begin{theorem}
Every factor representation of type ${\rm II}_1$ of the group $S_{\Bbb N}$ can be realized in the framework of the  groupoid
construction based on the action of~$S_{\Bbb N}$ on $([0,1]^{\Bbb N}, \nu_{\alpha})$, where $\nu_{\alpha}$ is a Bernoulli measure.
\end{theorem}

Now we may ask about the class of groups (which, in general, are not of type~I) for which the set of
representations of type ${\rm II}_1$ has a parameterization by a precompact space. In other words, when the space of indecomposable finite traces (or characters if we consider representations of groups)
is totally bounded? Of course, this question is natural if the group has sufficiently many traces, i.e., every pair of elements of the algebra that are not conjugate can be distinguished by some indecomposable trace. The infinite symmetric group is one of such groups. The question can be included into our general problem about central measures on  graded graphs.
The conclusion is that central measures in the case of the symmetric group give a very interesting series of  irreducible representations of the double group.


\newpage

\begin{thebibliography}{99}

\bibitem{Ab}M.~Abert, N.~Bergeron, I.~Biringer, T.~Gelander, N.~Nikolov, and J.~Raimbault. On the growth of Betti numbers of locally symmetric spaces. C. R. Math. Acad. Sci. Paris 349, No.~15--16, 831--835 (2011).

\bibitem{A81}D.~Aldous. Exchangeability and Related Topics. \'Ecole d'\'Et\'e de Probabilit\'es de Saint-Flour XIII--1983, Lect. Notes Math. 1117, Springer, 1985.

\bibitem{Brat}O.~Bratteli. Inductive limits of finite dimensional C*-algebras. Trans. Amer. Math. Soc. 171, 195--234 (1972).

\bibitem{Br}O.~Bratteli and D.~Robinson. Operator Algebras and Quantum Statistical Mechanics. Springer, 1997.

\bibitem{Buf}A.~I.~Bufetov. On the Vershik--Kerov conjecture concerning the Shannon--McMillan--Breiman theorem for the Plancherel family of measures on the space of Young diagrams. Geom. Funct. Anal. 22, No. 4, 938--975 (2012).

\bibitem{CST} T.~Ceccherini-Silberstein, F.~Scarabotti, and F.~Tolli.
Representation Theory of the Symmetric Groups: The Okounkov--Vershik Approach, Character Formulas, and Partition Algebras. Cambridge Univ. Press, 2010.

\bibitem{CFW}A.~Connes, J.~Feldman, and B.~Weiss. An amenable equivalence relation is generated by a single transformation. Ergodic Theory Dynam. Systems 1, No.~4, 431--450 (1981).

\bibitem{EffHa}E.~G.~Effros, D.~E.~Handelman,  and C.~L.~Shen.  Dimension groups and their affine representations. Amer. J. Math. 102, 385--402 (1980).

\bibitem{Ell} G.~A.~Elliott. On the classification of inductive limits of sequences of semi-simple finite dimensional algebras. J. Algebra 38, 29--44  (1976).

\bibitem{Gi} T.~Giordano, I.~Putnam, and C.~Skau. Full groups of Cantor minimal systems. Israel J. Math. 111, No.~1, 285--320 (1999).

\bibitem{GO}A.~Gnedin and G.~Olshanski. $q$-Exchangeability via quasi-invariance. Ann. Probab. 38, No.~6, 2103--2135 (2010).

\bibitem{GK}F.~Goodman and S.~Kerov. The Martin boundary of the Young--Fibonnaci lattice. J. Algebraic Combin. 11, No.~1, 17--48 (2000).

\bibitem{Gr}M.~Gromov. Metric Structure for Riemannian and Non-Riemannian Spaces. Springer, 1998.

 \bibitem{Haj} A.~Hajian, Y.~Ito, and S.~Kakutani. Invariant measures and orbits of dissipative transformations. Ann. Math. 9, 52--56 (1972).

\bibitem{It} S.~Ito. A construction of transversal flows for maximal Markov automorphisms. Tokyo J. Math. 1, No.~2, 305--324 (1978).

 \bibitem{Jan}E.~Janvresse and T.~de la Rue. The Pascal adic transformation is loosely Bernoulli. Ann. Inst. H. Poincar\'e. Probab. Statist. 40, No.~2, 133--139 (2004).

\bibitem{Del}E.~Janvresse, T.~de la Rue, and Y.~Velenik. Self-similar correction to the ergodic theorem of Pascal-adic transformation. Stoch. Dyn. 5, No.~1, 1--25 (2005).

 \bibitem{Kak}S.~Kakutani. A problem of equidistribution on the unit interval $[0, 1]$. Lect. Notes Math. 541, 369--375 (1976).

\bibitem{Kal}O.~Kallenberg. Probabilistic Symmetries and Invariance Principles. Springer-Verlag, New York (2005).

\bibitem{Kant}L.~V.~Kantorovich. On the translocation of masses. Dokl. Akad. Nauk SSSS 37, No.~7--8, 227--229 (1942). English translation: J. Math. Sci. 133, No.~4, 1381--1382 (2006).

\bibitem{Kat}A.~Katok and B.~Hasselblatt. Introduction to the Modern Theory of Dynamical Systems. Cambridge Univ. Press, 1997.

\bibitem{Kech}A.~Kechris. The structure of Borel equivalence relations in Polish spaces. In: H.~Judah, W.~Just, and W.~H.~Woodin (eds.), Set Theory of the Continuum, MSRI Publications 26, Springer-Verlag, 1992, pp.~89--102.

\bibitem{Kerov} S.~Kerov.
Asymptotic Representation Theory of the Symmetric Group and its Applications in Analysis.
Transl. Math. Monographs, Vol. 219, Amer. Math. Soc., 2003.

\bibitem{KOO}S.~Kerov, A.~Okounkov, and G.~Olshanski. The boundary of the Young graph with Jack edge multiplicities. Internat. Math. Res. Notices 1998, No. 4. 173--199 (1998).

\bibitem{Lieber}A.~Liebermann. The structure of certain unitary representations of infinite symmetric groups.
Trans. Amer. Math. Soc. 164, 189-198 (1972).

\bibitem{Lod} A.~Lodkin and A.~Minabutdinov. Limiting curves for the Pascal adic transformation. Zapiski Nauchn. Semin. POMI 437,  145--183 (2015). English translation to appear in J. Math. Sci. (2016).

\bibitem{Mel} X.~Mela and K.~Petersen. Dynamical properties of the Pascal adic transformation. Ergodic Theory Dynam. Systems 25, No.~1, 227--256  (2005).

\bibitem{Ok} A.~Okounkov. On representations of the infinite symmetric group. J. Math. Sci.  96, No.~5, 3550--3589 (1999).

\bibitem{VO} A.~Okounkov and A.~Vershik. A new approach to representation theory of symmetric groups.
Selecta Math. 2, No.~4, 581--605 (1996).

\bibitem{Olsh} G.~Olshanski. Unitary representations of $(G,K)$-pairs that are connected with the infinite symmetric group $S(\infty)$. Leningrad Math. J. 1, No.~4, 983--1014 (1990).	

\bibitem{Fel}R.~Phelps. Lectures on Choquet's Theorem. Springer Lecture Notes Math., 2001.

\bibitem{Pims} M.~Pimsner. Embedding some transformation group $C^*$-algebra into AF-algebras. Ergodic Theory Dynam. Systems 3, 613--626 (1983).

\bibitem{Ror}M.~R{\o}rdam. Classification of Nuclear C*-Algebras. Encyclopaedia of Mathematical Sciences 126, Springer, 2002.

\bibitem{Schm78}K.~Schmidt. A probabilistic proof of ergodic decomposition. Sankhya: The Indian Journal of Statistics
    40, Pt.~4, 10--18 (1978).

\bibitem{Schm}K.~Schmidt and G.~Greschonig. Ergodic decomposition of quasi-invariant probability measures. Coll. Math. 84/85, 495--514
(2000).

\bibitem{Erg} Ya.~G.~Sinai (ed.).  Dynamical Systems, Ergodic Theory and Applications. Encyclopaedia of Mathematical Sciences 100,
Springer, 2000.

 \bibitem{Strat}S.~Stratila and D.~Voiculescu. Representations of AF-algebras and of the Group $U_{\infty}$. Springer Lecture Notes Math., 1975.

\bibitem{Tho} E.~Thoma. Die unzerlegbaren, positiv-definiten Klassenfunktionen der abzahlbar unendlichen symmetrischen Gruppe.  Math. Z. 85, No.~1, 40--61 (1964).

\bibitem{Thoma2} E.~Thoma.
Eine Charakterisierung diskreter Gruppen vom Typ I.
Invent. Math. 6, 190--196 (1968).

\bibitem{V68}A.~Vershik. A theorem on the lacunary isomorphism of monotonic sequence of partitions.
Funkts.  Anal. i Prilozh. 2, No.~3, 17--21  (1968).

\bibitem{V70}A.~Vershik. Decreasing sequences of measurable partitions and their applications.
Sov. Math. Dokl. 11, 1007--1011 (1970).

\bibitem{V73}A.~Vershik. Approximation in measure theory.
Dissertation, Leningrad State University, Leningrad, 1973.

\bibitem{V73a}A.~Vershik. Four definitions of the scale of an automorphism. Funct. Anal. Appl. 7, 169--181 (1973).

\bibitem{V81}A.~Vershik. Uniform algebraic approximation of shift and multiplication operators. Dokl. Akad. Nauk 259, No.~3, 526--529  (1981).

\bibitem{V82} A.~Vershik. A theorem on periodical Markov approximation in ergodic theory. J. Sov. Math. 28, 667--674 (1985).

\bibitem{V94}A.~Vershik. Theory of decreasing sequences of measurable partitions. 
St.~Petersburg Math. J. 6, No.~4, 705--761 (1995).


\bibitem{ICM} A.~Vershik. Asymptotic combinatorics and algebraic analysis.
In: Proceedings of the International Congress of Mathematicians (Zurich, 1994), Vol.~II, Birkhauser, Basel, 1995, pp.~1384--1394.


\bibitem{Ve96}A.~Vershik. Statistical mechanics of combinatorial partitions, and their limit configurations.
  Funct. Anal. Appl. 30, No.~2, 90--105 (1996).

 \bibitem{V98}A.~Vershik. The universal Uryson space, Gromov's metric triples, and random metrics on the series of natural numbers. Russian Math. Surveys 53, No.~5, 921--928 (1998).

 \bibitem{V02}A.~Vershik. Classification of measurable functions of several arguments, and invariantly distributed random matrices. Funct. Anal. Appl. 36, No.~2, 93--105 (2002).

 \bibitem{MNU}A.~Vershik. Random and universal metric spaces.
In: Fundamental Mathematics Today (S.~K.~Lando and O.~K.~Sheinman, eds.), Independent University of Moscow, 2003, pp.~54--88.

 \bibitem{V2003}A.~Vershik. Random metric spaces and universality. Russian Math. Surveys 59, No.~2, 259--295 (2004).

\bibitem{2013}A.~Vershik. The Pascal automorphism has a continuous spectrum. Funct. Anal. Appl. 45, No. 3, 173--186 (2011).

\bibitem{Fr2010}A.~Vershik. Nonfree actions of countable groups and their characters. J. Math. Sci. 174, No.~1, 1--6 (2011).

\bibitem{Fr2012}A.~Vershik. Totally nonfree actions and the infinite symmetric group.
Moscow Math. J. 12, No. 1, 193--212 (2012).

\bibitem{V12}A.~Vershik. On classification of measurable functions of several variables. J. Math. Sci. 190, No.~3, 427--437 (2013).

\bibitem{V100}A.~Vershik.  Long history of the Monge--Kantorovich transportation problem. Math. Intelligencer 35, No.~4, 2--7 (2013).

 \bibitem{In}A.~Vershik. Intrinsic metric on graded graphs, standardness, and invariant measures. J. Math. Sci. 200, No.~6, 677--681 (2014).

 \bibitem{2014}A.~Vershik. The problem of describing central measures on the path spaces of graded graphs.
Funct. Anal. Appl. 48, No. 4, 256--271 (2014).

 \bibitem{V15}A.~Vershik. Several remarks on Pascal automorphism and infinite ergodic theory. Armenian J. Math.
7, Issue 2, 85--96 (2015);
{\tt arXiv:1512.03721}.


 \bibitem{Sm}A.~Vershik. Smoothness and standardness in the theory of AF-algebras and the problem about invariant measures. In: Proceedings of Symposia in Pure Mathematics, Vol.~91,  Probability and Statistics in St.~Petersburg, Amer. Math. Soc., 2015; preliminary version: {\tt arXiv:1304.2193}.

\bibitem{Tak} A.~Vershik. Invariant measures: new aspects of dynamics, combinatorics and representation theory.
    In: The Fifteenth Takagi Lectures. Japan. J. Math. (2015), ISSN 2187-3267, pp.~38--79.

\bibitem{VEq}A.~Vershik. Equipped graded graphs, projective limits of simplices, and their boundaries. J. Math. Sci. 206, No.~6, 860--873 (2015).

 \bibitem {2015}A.~Vershik. Standardness as an invariant formulation of independence. Funct. Anal. Appl. 49, No.~4, 253--263 (2015). 

\bibitem{VGo}A.~Vershik and A.~Gorbulsky. Scaled entropy of filtrations of sigma-fields.
Probab. Theory Appl. 52, No. 3, 446--467 (2007).

\bibitem{VH}A.~Vershik and U.~Haboek. On the classification problem of matrix distributions of measurable functions in several variables. Zapiski Nauchn. Semin. POMI 441, 119--143 (2015).
 English translation to appear in J. Math. Sci. (2016); {\tt arXiv:1512.06760}.

\bibitem{VeK81D}A.~Vershik and S.~Kerov. Characters and factor representations of the infinite symmetric group.
Sov. Math. Dokl. 23, 389--392 (1981).

 \bibitem{VeK81} A.~Vershik and S.~Kerov. Asymptotic theory of characters of the symmetric group.
Funct. Anal. Appl. 15, 246--255 (1982).

\bibitem{VeK85}A.~Vershik and S.~Kerov. Asymptotics of maximal and typical dimensions of irreducible representations of a symmetric group.  Funct. Anal. Appl. 19, 21--31 (1985).

\bibitem{VeKe} A.~Vershik and S.~Kerov. Locally semisimple algebras. Combinatorial theory and the K-functor.
J. Sov. Math. 38, 1701--1733 (1987).

 \bibitem{VM2015}A.~Vershik and A.~Malyutin.  Phase transition in the exit boundary problem for random walks on groups.
Funct. Anal. Appl. 49, No.~2, 86--96 (2015).

\bibitem{VN}A.~Vershik and P.~Nikitin. Description of characters and factor representations of the infinite symmetric inverse semigroup. Funct. Anal. Appl. 45, No.~1, 13--24 (2011).

\bibitem{VPZEu}A.~Vershik, F.~Petrov, and P.~Zatitskiy. Geometry and dynamics of admissible metrics in measure spaces.
Central Europ. J. Math. 11, No.~3, 379--400 (2013).

\bibitem{VPZ}A.~Vershik, P.~Zatitskiy, and F.~Petrov. Virtual continuity of measurable functions and its applications.
Russian Math. Surveys 69, No.~6, 81--114 (2014).

\bibitem{Wink} G.~Winkler. Choquet Order and Simplices. Springer Lecture Notes Math., 1985.



\end{thebibliography}
\end{document}